\documentclass[10pt]{amsart}

\usepackage[margin=1in,headheight=13.6pt]{geometry}

\usepackage{textcmds} 
\usepackage{amsmath}
\usepackage{amsxtra}
\usepackage{amscd}
\usepackage{amsthm}
\usepackage{amsfonts}
\usepackage{amssymb}
\usepackage{eucal}
\usepackage[all]{xy}
\usepackage{graphicx}
\usepackage{comment}
\usepackage{epsfig}
\usepackage{psfrag}
\usepackage{mathrsfs}
\usepackage{amscd}
\usepackage{rotating}
\usepackage{lscape}
\usepackage{amsbsy}
\usepackage{verbatim}
\usepackage{moreverb}
\usepackage{url}
\usepackage{extarrows}

\usepackage{bbm}

\usepackage[hidelinks]{hyperref}

\usepackage{tikz-cd}
\usepackage{tikz}
\usetikzlibrary{matrix,arrows,decorations.pathmorphing}

\DeclareMathAlphabet{\mathpzc}{OT1}{pzc}{m}{it}
\newtheorem{cor}[subsubsection]{Corollary}
\newtheorem{lem}[subsubsection]{Lemma}
\newtheorem{prop}[subsubsection]{Proposition}

\newtheorem{conj}[subsubsection]{Conjecture}
\newtheorem{thm}[subsubsection]{Theorem}

\theoremstyle{remark}
\newtheorem{rem}[subsubsection]{Remark}
\newtheorem{example}[subsubsection]{Example}

\theoremstyle{remark}

\numberwithin{equation}{section}

\newcommand{\nc}{\newcommand}
\nc{\renc}{\renewcommand}
\nc{\ssec}{\subsection}
\nc{\sssec}{\subsubsection}
\nc{\on}{\operatorname}

\nc\ol{\overline}
\nc\wt{\widetilde}
\nc\tboxtimes{\wt{\boxtimes}}
\nc\tstar{\wt{\star}}
\nc{\alp}{\alpha}

\nc{\ZZ}{{\mathbb Z}}
\nc{\NN}{{\mathbb N}}
\nc{\OO}{{\mathbb O}}
\renc{\SS}{{\mathbb S}}
\nc{\DD}{{\mathbb D}}
\nc{\GG}{{\mathbb G}}
\renewcommand{\AA}{{\mathbb A}}
\nc{\Fq}{{\mathbb F}_q}
\nc{\Fqb}{\ol{{\mathbb F}_q}}
\nc{\Ql}{\ol{{\mathbb Q}_\ell}}
\nc{\id}{\text{id}}
\nc\X{\mathcal X}

\nc{\Hom}{\on{Hom}}
\nc{\Lie}{\on{Lie}}
\nc{\Loc}{\on{Loc}}
\nc{\Pic}{\on{Pic}}
\nc{\Bun}{\on{Bun}}
\nc{\IC}{\on{IC}}
\nc{\Aut}{\on{Aut}}
\nc{\rk}{\on{rk}}
\nc{\Sh}{\on{Sh}}
\nc{\Perv}{\on{Perv}}
\nc{\pos}{{\on{pos}}}
\nc{\Conv}{\on{Conv}}
\nc{\Sph}{\on{Sph}}
\nc{\Sym}{\on{Sym}}
\nc{\BunBb}{\overline{\Bun}_B}
\nc{\BunNb}{\overline{\Bun}_N}
\nc{\BunTb}{\overline{\Bun}_T}
\nc{\BunBbm}{\overline{\Bun}_{B^-}}
\nc{\BunBbel}{\overline{\Bun}_{B,el}}
\nc{\BunBbmel}{\overline{\Bun}_{B^-,el}}
\nc{\Buno}{\overset{o}{\Bun}}
\nc{\BunPb}{{\overline{\Bun}_P}}
\nc{\BunBM}{\Bun_{B(M)}}
\nc{\BunBMb}{\overline{\Bun}_{B(M)}}
\nc{\BunPbw}{{\widetilde{\Bun}_P}}
\nc{\BunBP}{\widetilde{\Bun}_{B,P}}
\nc{\GUb}{\overline{G/U}}
\nc{\GUPb}{\overline{G/U(P)}}

\nc\syminfty{\on{Sym}^{\infty}}
\nc\lal{\ol{\lambda}}
\nc\xl{\ol{x}}
\nc\thl{\ol{\theta}}
\nc\nul{\ol{\nu}}
\nc\mul{\ol{\mu}}
\nc\Sum\Sigma
\nc{\oX}{\overset{\circ}{X}{}}
\nc{\hl}{\overset{\leftarrow}h{}}
\nc{\hr}{\overset{\rightarrow}h{}}
\nc{\M}{{\mathcal M}}
\nc{\N}{{\mathcal N}}
\nc{\F}{{\mathcal F}}
\nc{\D}{{\mathcal D}}
\nc{\Y}{{\mathcal Y}}
\nc{\G}{{\mathcal G}}
\nc{\E}{{\mathcal E}}
\nc{\CalC}{{\mathcal C}}
\nc\Dh{\widehat{\D}}
\renewcommand{\O}{{\mathcal O}}
\nc{\K}{{\mathcal K}}
\renewcommand{\S}{{\mathcal S}}
\nc{\T}{{\mathcal T}}
\nc{\V}{{\mathcal V}}
\renc{\P}{{\mathcal P}}
\nc{\A}{{\AA}}
\nc{\B}{{\BB}}
\nc{\U}{{\mathcal U}}

\nc{\frn}{{\check{\mathfrak u}(P)}}

\nc{\fC}{\mathfrak C}
\nc\f{{\mathfrak f}}

\nc{\qo}{{\mathfrak q}}
\nc{\po}{{\mathfrak p}}
\nc{\s}{{\mathfrak s}}
\nc\w{\text{w}}
\renewcommand{\r}{{\mathfrak r}}

\renewcommand{\mod}{{\on{-}\mathsf{mod}}}

\newcommand{\bimod}{{\on{-}\mathsf{bimod}}}

\nc\mathi\iota
\nc\Spec{\on{Spec}}
\nc\Mod{\on{Mod}}
\nc{\tw}{\widetilde{\mathfrak t}}
\nc{\pw}{\widetilde{\mathfrak p}}
\nc{\qw}{\widetilde{\mathfrak q}}
\nc{\jw}{\widetilde j}

\nc{\grb}{\overline{\Gr_{X^{\fset}}}}
\nc{\I}{\mathcal I}
\renewcommand{\i}{\mathfrak i}
\renewcommand{\j}{\mathfrak j}

\nc{\lambdach}{{\check\lambda}}
\nc{\Lambdach}{{\check\Lambda}{}}
\nc{\much}{{\check\mu}}
\nc{\omegach}{{\check\omega}}
\nc{\nuch}{{\check\nu}}
\nc{\etach}{{\check\eta}}
\nc{\alphach}{{\check\alpha}}

\nc{\rhoch}{{\check\rho}}

\nc{\Hb}{\overline{\H}}

\nc{\BA}{{\mathbb{A}}}
\nc{\BB}{\mathbb{B}}
\nc{\BC}{{\mathbb{C}}}
\nc{\BD}{{\mathbb{D}}}
\nc{\BE}{{\mathbb{E}}}
\nc{\BF}{{\mathbb{F}}}
\nc{\BG}{{\mathbb{G}}}
\nc{\BH}{{\mathbb{H}}}
\nc{\BI}{{\mathbb{I}}}
\nc{\BM}{{\mathbb{M}}}
\nc{\BN}{{\mathbb{N}}}
\nc{\BO}{{\mathbb{O}}}
\nc{\BP}{{\mathbb{P}}}
\nc{\BQ}{{\mathbb{Q}}}
\nc{\BR}{{\mathbb{R}}}
\nc{\BS}{{\mathbb{S}}}
\nc{\BT}{{\mathbb{T}}}
\nc{\BV}{{\mathbb{V}}}
\nc{\BZ}{{\mathbb{Z}}}

\nc{\bbone}{\mathbbm{1}}
\nc{\bbA}{{\mathbb{A}}}
\nc{\bbB}{\mathbb{B}}
\nc{\bbC}{{\mathbb{C}}}
\nc{\bbD}{{\mathbb{D}}}
\nc{\bbE}{{\mathbb{E}}}
\nc{\bbF}{{\mathbb{F}}}
\nc{\bbG}{{\mathbb{G}}}
\nc{\bbH}{{\mathbb{H}}}
\nc{\bbI}{{\mathbb{I}}}
\nc{\bbL}{{\mathbb{L}}}
\nc{\bbM}{{\mathbb{M}}}
\nc{\bbN}{{\mathbb{N}}}
\nc{\bbO}{{\mathbb{O}}}
\nc{\bbP}{{\mathbb{P}}}
\nc{\bbQ}{{\mathbb{Q}}}
\nc{\bbR}{{\mathbb{R}}}
\nc{\bbS}{{\mathbb{S}}}
\nc{\bbT}{{\mathbb{T}}}
\nc{\bbU}{{\mathbb{U}}}
\nc{\bbV}{{\mathbb{V}}}
\nc{\bbW}{{\mathbb{W}}}
\nc{\bbX}{{\mathbb{X}}}
\nc{\bbY}{{\mathbb{Y}}}
\nc{\bbZ}{{\mathbb{Z}}}

\nc{\CA}{{\mathcal{A}}}
\nc{\CB}{{\mathcal{B}}}

\nc{\CE}{{\mathcal{E}}}
\nc{\CF}{{\mathcal{F}}}
\nc{\CH}{{\mathcal{H}}}

\nc{\CL}{{\mathcal{L}}}
\nc{\CC}{{\mathcal{C}}}
\nc{\CG}{{\mathcal{G}}}
\nc{\CM}{{\mathcal{M}}}
\nc{\CN}{{\mathcal{N}}}
\nc{\CK}{{\mathcal{K}}}
\nc{\CO}{{\mathcal{O}}}
\nc{\CP}{{\mathcal{P}}}
\nc{\CQ}{{\mathcal{Q}}}
\nc{\CR}{{\mathcal{R}}}
\nc{\CS}{{\mathcal{S}}}
\nc{\CU}{{\mathcal{U}}}
\nc{\CV}{{\mathcal{V}}}
\nc{\CW}{{\mathcal{W}}}
\nc{\CX}{{\mathcal{X}}}
\nc{\CY}{{\mathcal{Y}}}
\nc{\CZ}{{\mathcal{Z}}}
\nc{\CI}{{\mathcal{I}}}

\nc{\csM}{{\check{\mathcal A}}{}}
\nc{\oM}{{\overset{\circ}{\mathcal M}}{}}
\nc{\obM}{{\overset{\circ}{\mathbf M}}{}}
\nc{\oCA}{{\overset{\circ}{\mathcal A}}{}}
\nc{\obA}{{\overset{\circ}{\mathbf A}}{}}
\nc{\ooM}{{\overset{\circ}{M}}{}}
\nc{\osM}{{\overset{\circ}{\mathsf M}}{}}
\nc{\vM}{{\overset{\bullet}{\mathcal M}}{}}
\nc{\nM}{{\underset{\bullet}{\mathcal M}}{}}
\nc{\oD}{{\overset{\circ}{\mathcal D}}{}}
\nc{\obC}{{\overset{\circ}{\mathbf C}}{}}
\nc{\obD}{{\overset{\circ}{\mathbf D}}{}}
\nc{\oA}{{\overset{\circ}{\mathbb A}}{}}
\nc{\op}{{\overset{\bullet}{\mathbf p}}{}}
\nc{\oU}{{\overset{\bullet}{\mathcal U}}{}}
\nc{\oZ}{{\overset{\circ}{\mathcal Z}}{}}
\nc{\ofZ}{{\overset{\circ}{\mathfrak Z}}{}}
\nc{\oF}{{\overset{\circ}{\fF}}}

\nc{\fa}{{\mathfrak{a}}}
\nc{\fb}{{\mathfrak{b}}}
\nc{\fc}{{\mathfrak{c}}}
\nc{\fd}{{\mathfrak{d}}}
\nc{\ff}{{\mathfrak{f}}}
\nc{\fg}{{\mathfrak{g}}}
\nc{\fgl}{{\mathfrak{gl}}}
\nc{\fh}{{\mathfrak{h}}}
\nc{\fj}{{\mathfrak{j}}}
\nc{\fl}{{\mathfrak{l}}}
\nc{\fm}{{\mathfrak{m}}}
\nc{\fn}{{\mathfrak{n}}}
\nc{\fu}{{\mathfrak{u}}}
\nc{\fp}{{\mathfrak{p}}}
\nc{\fr}{{\mathfrak{r}}}
\nc{\fs}{{\mathfrak{s}}}
\nc{\ft}{{\mathfrak{t}}}
\nc{\fz}{{\mathfrak{z}}}
\nc{\fsl}{{\mathfrak{sl}}}
\nc{\hsl}{{\widehat{\mathfrak{sl}}}}
\nc{\hgl}{{\widehat{\mathfrak{gl}}}}
\nc{\hg}{{\widehat{\mathfrak{g}}}}
\nc{\chg}{{\widehat{\mathfrak{g}}}{}^\vee}
\nc{\hn}{{\widehat{\mathfrak{n}}}}
\nc{\chn}{{\widehat{\mathfrak{n}}}{}^\vee}

\nc{\fA}{{\mathfrak{A}}}
\nc{\fB}{{\mathfrak{B}}}
\nc{\fD}{{\mathfrak{D}}}
\nc{\fE}{{\mathfrak{E}}}
\nc{\fF}{{\mathfrak{F}}}
\nc{\fG}{{\mathfrak{G}}}
\nc{\fK}{{\mathfrak{K}}}
\nc{\fL}{{\mathfrak{L}}}
\nc{\fM}{{\mathfrak{M}}}
\nc{\fN}{{\mathfrak{N}}}
\nc{\fP}{{\mathfrak{P}}}
\nc{\fU}{{\mathfrak{U}}}
\nc{\fV}{{\mathfrak{V}}}
\nc{\fZ}{{\mathfrak{Z}}}

\nc{\bb}{{\mathbf{b}}}
\nc{\bc}{{\mathbf{c}}}
\nc{\bd}{{\mathbf{d}}}
\nc{\bbf}{{\mathbf{f}}}
\nc{\be}{{\mathbf{e}}}
\nc{\bg}{{\mathbf{g}}}
\nc{\bi}{{\mathbf{i}}}
\nc{\bj}{{\mathbf{j}}}
\nc{\bn}{{\mathbf{n}}}
\nc{\bo}{{\mathbf{o}}}
\nc{\bp}{{\mathbf{p}}}
\nc{\bq}{{\mathbf{q}}}
\nc{\bt}{{\mathbf{t}}}
\nc{\bu}{{\mathbf{u}}}
\nc{\bv}{{\mathbf{v}}}
\nc{\bx}{{\mathbf{x}}}
\nc{\bs}{{\mathbf{s}}}
\nc{\by}{{\mathbf{y}}}
\nc{\bw}{{\mathbf{w}}}
\nc{\bA}{{\mathbf{A}}}
\nc{\bK}{{\mathbf{K}}}
\nc{\bB}{{\mathbf{B}}}
\nc{\bC}{{\mathbf{C}}}
\nc{\bG}{{\mathbf{G}}}
\nc{\bD}{{\mathbf{D}}}
\nc{\bH}{{\mathbf{H}}}
\nc{\bM}{{\mathbf{M}}}
\nc{\bN}{{\mathbf{N}}}
\nc{\bO}{{\mathbf{O}}}
\nc{\bT}{{\mathbf{T}}}
\nc{\bV}{{\mathbf{V}}}
\nc{\bW}{{\mathbf{W}}}
\nc{\bX}{{\mathbf{X}}}
\nc{\bZ}{{\mathbf{Z}}}
\nc{\bS}{{\mathbf{S}}}

\nc{\sA}{{\mathsf{A}}}
\nc{\sB}{{\mathsf{B}}}
\nc{\sC}{{\mathsf{C}}}
\nc{\sD}{{\mathsf{D}}}
\nc{\sF}{{\mathsf{F}}}
\nc{\sG}{{\mathsf{G}}}
\nc{\sK}{{\mathsf{K}}}
\nc{\sM}{{\mathsf{M}}}
\nc{\sO}{{\mathsf{O}}}
\nc{\sW}{{\mathsf{W}}}
\nc{\sQ}{{\mathsf{Q}}}
\nc{\sP}{{\mathsf{P}}}
\nc{\sV}{{\mathsf{V}}}
\nc{\sS}{{\mathsf{S}}}
\nc{\sT}{{\mathsf{T}}}
\nc{\sZ}{{\mathsf{Z}}}
\nc{\sfp}{{\mathsf{p}}}
\nc{\sll}{{\mathsf{l}}}
\nc{\sr}{{\mathsf{r}}}
\nc{\bk}{{\mathsf{k}}}
\nc{\sg}{{\mathsf{g}}}

\nc{\sff}{{\mathsf{f}}}
\nc{\sfb}{{\mathsf{b}}}
\nc{\sfc}{{\mathsf{c}}}
\nc{\sd}{{\mathsf{d}}}
\nc{\se}{{\mathsf{e}}}

\nc{\BK}{{\bar{K}}}

\nc{\tA}{{\widetilde{\mathbf{A}}}}
\nc{\tB}{{\widetilde{\mathcal{B}}}}
\nc{\tg}{{\widetilde{\mathfrak{g}}}}
\nc{\tG}{{\widetilde{G}}}
\nc{\TM}{{\widetilde{\mathbb{M}}}{}}
\nc{\tO}{{\widetilde{\mathsf{O}}}{}}
\nc{\tU}{{\widetilde{\mathfrak{U}}}{}}
\nc{\TZ}{{\tilde{Z}}}
\nc{\tx}{{\tilde{x}}}
\nc{\tbv}{{\tilde{\bv}}}
\nc{\tfP}{{\widetilde{\mathfrak{P}}}{}}
\nc{\tz}{{\tilde{\zeta}}}
\nc{\tmu}{{\tilde{\mu}}}

\nc{\urho}{\underline{\rho}}
\nc{\uB}{\underline{B}}
\nc{\uC}{{\underline{\mathbb{C}}}}
\nc{\ui}{\underline{i}}
\nc{\uj}{\underline{j}}
\nc{\ofP}{{\overline{\mathfrak{P}}}}
\nc{\oB}{{\overline{\mathcal{B}}}}
\nc{\og}{{\overline{\mathfrak{g}}}}
\nc{\oI}{{\overline{I}}}

\nc{\eps}{\varepsilon}
\nc{\hrho}{{\hat{\rho}}}

\nc{\one}{{\mathbf{1}}}
\nc{\two}{{\mathbf{t}}}

\nc{\Rep}{{\mathop{\operatorname{\rm Rep}}}}
\nc{\Tot}{{\mathop{\operatorname{\rm Tot}}}}
\nc{\Ker}{{\mathop{\operatorname{\rm Ker}}}}
\nc{\Hilb}{{\mathop{\operatorname{\rm Hilb}}}}
\nc{\Ext}{{\mathop{\operatorname{\rm Ext}}}}
\nc{\CHom}{{\mathop{\operatorname{{\mathcal{H}}\it om}}}}
\nc{\GL}{{\mathop{\operatorname{\rm GL}}}}
\nc{\gr}{{\mathop{\operatorname{\rm gr}}}}
\nc{\Id}{{\mathop{\operatorname{\rm Id}}}}
\nc{\de}{{\mathop{\operatorname{\rm def}}}}
\nc{\length}{{\mathop{\operatorname{\rm length}}}}
\nc{\supp}{{\mathop{\operatorname{\rm supp}}}}

\nc{\Cliff}{{\mathsf{Cliff}}}
\nc{\Fl}{\on{Fl}}
\nc{\Fib}{{\mathsf{Fib}}}
\nc{\Coh}{{\on{Coh}}}
\nc{\QCoh}{{\on{QCoh}}}
\nc{\IndCoh}{{\on{IndCoh}}}
\nc{\FCoh}{{\mathsf{FCoh}}}

\nc{\reg}{{\text{\rm reg}}}

\nc{\cplus}{{\mathbf{C}_+}}
\nc{\cminus}{{\mathbf{C}_-}}
\nc{\cthree}{{\mathbf{C}_*}}
\nc{\Qbar}{{\bar{Q}}}
\nc\Eis{\on{Eis}}
\nc\Eisb{\ol\Eis{}}
\nc\Eisr{\on{Eis}^{rat}{}}
\nc\wh{\widehat}
\nc{\Def}{\on{Def_{\check{\fb}}(E)}}
\nc{\barZ}{\overline{Z}{}}
\nc{\barbarZ}{\overline{\barZ}{}}
\nc{\barpi}{\overline\pi}
\nc{\barbarpi}{\overline\barpi}
\nc{\barpip}{\overline\pi{}^+}
\nc{\barpim}{\overline\pi{}^-}

\nc{\fq}{\mathfrak q}

\nc{\fqb}{\ol{\fq}{}}
\nc{\fpb}{\ol{\fp}{}}
\nc{\fpr}{{\fp^{rat}}{}}
\nc{\fqr}{{\fq^{rat}}{}}

\nc{\hattimes}{\wh\otimes}

\nc{\bh}{{\bar{h}}}
\nc{\bOmega}{{\overline{\Omega(\check \fn)}}}

\nc{\seq}[1]{\stackrel{#1}{\sim}}

\nc{\cT}{{\check{T}}}
\nc{\cG}{{\check{G}}}
\nc{\cM}{{\check{M}}}
\nc{\cB}{{\check{B}}}
\nc{\cP}{{\check{P}}}

\nc{\ct}{{\check{\mathfrak t}}}
\nc{\cg}{{\check{\fg}}}
\nc{\cb}{{\check{\fb}}}
\nc{\cn}{{\check{\fn}}}
\nc{\cp}{{\check{\fp}}}
\nc{\cm}{{\check{\fm}}}

\nc{\cLambda}{{\check\Lambda}}

\nc{\cla}{{\check\lambda}}
\nc{\cmu}{{\check\mu}}
\nc{\cnu}{{\check\nu}}
\nc{\ceta}{{\check\eta}}

\nc{\DefbE}{{\on{Def}_{\cB}(E_\cT)}}

\nc{\imathb}{{\ol{\imath}}}
\nc{\rlr}{\overset{\longrightarrow}{\underset{\longrightarrow}\longleftarrow}}

\nc{\oBun}{\overset{\circ}\Bun}
\nc{\BunBbb}{\ol{\ol{Bun}}_B}
\nc{\BunBr}{\Bun_B^{rat}}
\nc{\BunBrsg}{\Bun_B^{rat,\on{s.g.}}}
\nc{\BunBrp}{\Bun_B^{rat,polar}}
\nc{\BunBrpbg}{\Bun_B^{rat,polar,\on{b.g.}}}
\nc{\BunBrpsg}{\Bun_B^{rat,polar,\on{s.g.}}}
\nc{\BunTrp}{\Bun_T^{rat,polar}}
\nc{\BunTrpbg}{\Bun_T^{rat,polar,\on{b.g.}}}
\nc{\BunTrpsg}{\Bun_T^{rat,polar,\on{s.g.}}}
\nc{\BunNr}{\Bun_N^{rat}}
\nc{\BunNre}{\Bun_N^{enh,rat}}
\nc{\BunTr}{\Bun_T^{rat}}
\nc{\Vect}{\on{Vect}}
\nc{\Whit}{\on{Whit}}
\nc{\bTb}{\ol{\on{CT}}}

\nc{\bTr}{\on{CT}^{rat}{}}
\nc\jmathr{\jmath^{rat}{}}
\nc{\ux}{\underline{x}}
\nc{\clambda}{{\check\lambda}}
\nc{\calpha}{{\check\alpha}}

\nc{\inftyGrpd}{{\mathsf{Grpd}_\infty}}

\nc{\fset}{\mathsf{fSet}}
\nc{\LocSysG}{\LocSys_{\cG}}
\nc{\Sing}{{\on{Sing}}}
\nc{\dr}{{\on{dR}}}
\nc{\Ind}{\on{Ind}}
\nc{\Sat}{\on{Sat}}
\nc{\Ho}{\on{Ho}}
\nc{\Res}{\on{Res}}
\nc{\sotimes}{\overset{!}\otimes}
\nc{\mmod}{{\on{-}}{\mathbf{mod}}}
\nc{\Maps}{\on{Maps}}
\nc{\CMaps}{{\mathcal Maps}}
\nc{\bMaps}{{\mathbf{Maps}}}

\nc{\dgSch}{\on{DGSch}}
\nc{\dgindSch}{\on{DGindSch}}
\nc{\indSch}{\on{indSch}}
\nc{\Sch}{\mathsf{Sch}}
\nc{\affdgSch}{\on{DGSch}^{\on{aff}}}
\nc{\affSch}{\on{Sch}^{\on{aff}}}

\nc{\Groupoids}{\on{Grpd}}
\nc{\inftypic}{\infty\on{-PicGrpd}}
\nc{\inftyCat}{{\mathsf{Cat}_{\infty}}}
\nc{\MoninftyCat}{\infty\on{-Cat}^{Mon}}
\nc{\SymMoninftyCat}{\infty\on{-Cat}^{\on{SymMon}}}
\nc{\SymMonStinftyCat}{\on{DGCat}^{\on{SymMon}}}
\nc{\MonStinftyCat}{\on{DGCat}^{Mon}}
\nc{\inftystack}{\on{Stk}}
\nc{\inftystackalg}{Stk^{1\text{-}alg}}
\nc{\inftyprestack}{\on{PreStk}}
\nc{\inftydgnearstack}{\on{NearStk}}
\nc{\inftydgstack}{\on{Stk}}
\nc{\inftydgstackalg}{DGStk^{1\text{-}alg}}
\nc{\inftydgprestack}{\on{PreStk}}

\nc{\HC}{\CH\bC}

\nc{\csupp}{\supp}
\nc{\Arth}{\on{Arth}}
\nc{\ArthG}{{\on{Arth}_\cG}}
\nc{\ul}{\underline}

\nc{\Z}{\mathcal{Z}}

\nc{\calN}{\N}
\nc{\calW}{\mathcal{W}}
\nc{\calF}{\mathcal{F}}
\nc{\calH}{\mathcal{H}}
\nc{\calO}{\mathcal{O}}
\nc{\calK}{\mathcal{K}}

\nc{\Ran}{\mathsf{Ran}}
\nc{\Jets}{\on{Jets}}
\nc{\act}{\mathsf{act}}
\nc{\Av}{\mathsf{Av}}
\nc{\Ad}{\on{Ad}}
\nc{\BGRan}{BG_{\Ran}}
\nc{\colim}{\on{colim}}
\nc{\codim}{\on{codim}}
\nc{\cpt}{{\on{cpt}}}
\nc{\dR}{{\on{dR}}}
\nc{\DGCat}{\mathsf{DGCat}}
\nc{\DGCatcont}{\on{DGCat}_{cont}}
\nc{\glob}{{\on{glob}}}
\nc{\loc}{{\on{loc}}}

\renewcommand{\op}{{\on{op}}}
\nc{\pt}{{\on{pt}}}
\nc{\PreStk}{{\mathsf{PreStk}}}
\nc{\Cat}{{\mathsf{Cat}}}

\nc{\ShvCat}{{\mathsf{ShvCat}}}

\nc{\restr}[2]{\left. #1 \right |_{#2}}
\nc{\uprestr}[2]{\left. #1 \right |^{#2}}

\nc{\bLoc}{{\mathbf{Loc}}}
\nc{\bGamma}{{\mathbf{\Gamma}}}
\nc{\bLocA}{\mathbf{Loc}^\A}
\nc{\bGammaA}{\mathbf{\Gamma}^\A}
\nc{\bLocB}{\mathbf{Loc}^\B}
\nc{\bGammaB}{\mathbf{\Gamma}^\B}
\nc{\bLocH}{\mathbf{Loc}^\H}
\nc{\bGammaH}{\mathbf{\Gamma}^\H}
\nc{\gen}{\mathsf{gen}}

\nc{\hto}{\hookrightarrow}

\nc{\ext}{\mathsf{ext}}

\nc{\ev}{\mathsf{ev}}
\nc{\rat}{\mathsf{rat}}

\nc{\usotimes}[1]{\underset{#1}{\otimes}}
\nc{\ustimes}[1]{\underset{#1}{\times}}
\nc{\uscolim}[1]{\underset{#1}{\colim}}

\nc{\ch}{{\mathfrak{ch}}}

\renc{\fD}{{\Dmod}}
\nc{\fH}{{\mathfrak{H}}}

\nc{\p}{{\mathfrak{p}}}
\renc{\r}{{\mathfrak{r}}}

\nc{\xto}{\xrightarrow}

\renc{\sec}{\section}
\nc{\enh}{\mathsf{enh}}

\renc{\gen}{\mathsf{gen}}
\nc{\BunGBgen}{\Bun_G^{B-\gen}}
\nc{\BunGHgen}{\Bun_G^{H-\gen}}
\nc{\BunGNgen}{\Bun_G^{N-\gen}}

\nc{\Fun}{\mathsf{Fun}}
\nc{\End}{\mathsf{End}}

\nc{\lr}{\xymatrix{ \ar@<-0.1ex>[r] \ar@<.8ex>[l]  & } }
\nc{\rr}{\xymatrix{ \ar@<-0.1ex>[r] \ar@<.8ex>[r]  & } }
\nc{\rrr}{\xymatrix{ \ar@<.2ex>[r] \ar@<.9ex>[r] \ar@<-0.5ex>[r] & } }
\nc{\Stab}{\mathsf{Stab}}
\nc{\Orb}{\mathsf{Orb}}

\renc{\exp}{\mathit{exp}}

\renc{\q}{\mathfrak{q}}

\nc{\virg}[1]{``#1"}

\renc{\bold}[1]{\boldsymbol{#1}}

\nc{\bigt}[1]{\big( #1 \big) }
\nc{\Bigt}[1]{\Big( #1 \Big) }

\nc{\extwhit}{{\CW h}(G,\mathsf{ext})}

\nc{\footcite}{\footnote}

\nc{\GA}{{G(\AA)}}
\nc{\GO}{{G(\OO)}}

\nc{\Shv}{\mathsf{Shv}}
\nc{\inc}{\mathsf{inc}}

\nc{\Par}{\mathsf{Par}}

\renc{\i}{\mathfrak{i}}

\nc{\NA}{N(\AA)}
\nc{\VA}{V(\AA)}

\nc{\Glue}{\mathsf{Glue}}
\nc{\laxlim}{\text{laxlim}}

\nc{\FT}{\mathsf{FT}}

\nc{\out}{\mathsf{out}}

\nc{\hol}{\mathsf{hol}}
\nc{\Hol}{\on{Hol}}
\nc{\add}{\mathsf{add}}

\nc{\sto}{\rightsquigarrow}
\nc{\squigto}{\rightsquigarrow}

\nc{\fW}{\mathfrak{W}}

\nc{\vrho}{\varrho}

\nc{\counit}{\mathsf{counit}}
\nc{\unit}{\mathsf{unit}}
\nc{\corr}{\mathsf{corr}}
\nc{\Corr}{\mathsf{Corr}}

\nc{\IndSch}{\mathsf{IndSch}}
\nc{\Tate}{{\mathsf{Tate}}}

\nc{\surjto}{\twoheadrightarrow}

\renc{\j}{\mathfrak{j}}

\nc{\J}{\mathcal{J}}

\nc{\pro}{\mathsf{pro}}
\nc{\fty}{\mathsf{ft}}
\nc{\Pro}{\mathsf{Pro}}

\nc{\coact}{\mathsf{coact}}
\nc{\aff}{\mathsf{aff}}

\nc{\Nilp}{\on{Nilp}}

\nc{\Gch}{{\check{G}}}
\nc{\Pch}{{\check{P}}}
\nc{\Mch}{{\check{M}}}
\nc{\Qch}{{\check{Q}}}

\nc{\LL}{\mathbb{L}}

\nc{\LS}{{\on{LS}}}

\nc{\x}{\varkappa} %%to denote a point of Ran(X).

\nc{\Otimes}{\boldsymbol{\otimes}}
\nc{\Times}{\boldsymbol{\times}}
\nc{\flip}{\text{<}}

\nc{\coeffRan}{\mathsf{coeff}^{\Ran}}

\nc{\Ha}{H(\sA)}
\nc{\Groups}{\mathsf{Groups}}

\nc{\Groth}{\mathsf{Groth}}

\nc{\rlto}{\rightleftarrows}

\nc{\DGCatRan}{\ShvCatCrys(\Ran)}

\nc{\longto}{\longrightarrow}

\renc{\Jets}{\mathsf{Jets}}
\nc{\mer}{\mathsf{mer}}

\nc{\W}{\mathcal{W}}

\nc{\Sect}{\mathsf{Sect}}
\renc{\Maps}{\mathsf{Maps}}

\renc{\bf}{\mathbf{f}}

\nc{\y}{\mathtt{y}}
\renc{\x}{\mathtt{x}}

\nc{\un}{{\it un}}
\nc{\indep}{\mathsf{indep}}
\nc{\CoAlg}{\mathsf{CoAlg}}

\nc{\coeff}{\mathsf{coeff}}

\nc{\R}{\mathcal{R}}

\renc{\hat}{\widehat}

\nc{\TK}{T(\mathsf{K})} %%T(K) Ran with no twist
\nc{\TtKK}{\Tt(\mathpzc{K})} %%independent T(K) with twist
\nc{\TtK}{\Tt(\mathsf{K})} %%T(K) Ran with twist
\nc{\KK}{\mathpzc{K}}

\nc{\Dmod}{\mathfrak{D}}

\nc{\curs}[1]{\mathpzc{#1}}

\nc{\Bshv}{\bold{\B}}
\nc{\Bind}{\H_{\indep}}
\nc{\BRan}{\H_{\Ran}}
\nc{\ARan}{\A_{\Ran}}
\nc{\Aind}{\A_{\indep}}

\nc{\GrRan}{\Gr}
\nc{\Gr}{\mathsf{Gr}}
\nc{\GrGRan}{\Gr_{G}}
\nc{\GrGind}{\Gr_{G}^{\indep}}
\nc{\Grind}[1]{\Gr_{#1}^{\indep} }
\nc{\GrGdom}{\curs{Gr}_G}

\nc{\GMapsRan}[1]{\mathsf{GMaps}(X,{#1})}
\nc{\GSectRan}[1]{\mathsf{GSect}({#1}/X)}
\nc{\GMapsind}[1]{\mathsf{GMaps}(X,{#1})^\indep}
\nc{\GSectind}[1]{\mathsf{GSect}({#1}/X)^\indep}
\nc{\GMapsdom}[1]{\curs{GMaps}(X,{#1})}
\nc{\GSectdom}[1]{\curs{GSect}({#1}/X)}

\nc{\chind}{\ch^{\indep}}
\nc{\chdom}{\curs{ch}}

\nc{\QSect}[1]{\curs{QSect}(#1/X)} 
\nc{\QMaps}[1]{\curs{QMaps}(X,#1)} 

\nc{\Zar}{\mathit{Zar}}

\nc{\loccit}{\textit{loc.$\,$cit.}}
\nc{\Crys}{\on{Crys}}
\nc{\ShvCatCrys}{\ShvCat^{\Crys}}

\nc{\BPE}{{\BP E}}
\nc{\BVE}{{\BV E}}
\nc{\BBE}{{\BB E}}

\nc{\Wh}{{{\CW}h}}

\nc{\ChiralCat}{\mathsf{ChiralCat}}
\nc{\RRep}{\mathfrak{R}ep}
\nc{\SSph}{\mathfrak{S}ph}

\nc{\tto}{\twoheadrightarrow}
\nc{\disj}{{\mathsf{disj}}}

\nc{\C}{\CC}
\nc{\Tch}{{\check{T}}}

\nc{\good}{\mathsf{good}}
\nc{\triv}{\mathsf{triv}}

\nc{\Alg}{\mathsf{Alg}}
\nc{\CAlg}{\mathsf{CAlg}}
\nc{\Spread}{\mathsf{Spread}}
\nc{\Dom}{\mathsf{Dom}}

\nc{\Jac}{\on{Jac}}
\renc{\CD}[1]{{#1}^{\on{CD}}}

\nc{\String}{\on{String}}

\renc{\min}{{\mathit{min}}}

\nc{\rrep}{\on-\!\mathbf{rep}}

\nc{\WWh}{\mathfrak{W}h}
\nc{\Grpd}{\mathsf{Grpd}}

\nc{\timesdisj}{\overset{\circ}\times}

\renc{\NA}{N(\sA)}
\nc{\chiral}{\mathsf{chiral}}

\nc{\Hopf}{\mathsf{Hopf}}

\nc{\heart}{\heartsuit}
\nc{\kk}{\mathbbm{k}} %% ground field

\nc{\HHom}{\CH{om}} %%Hom in DG cat.
\nc{\Cone}{\on{Cone}}

\nc{\EE}{\mathbb{E}}
\renc{\HC}{{\on{HC}}}
\nc{\HH}{{\on{HH}}}
\nc{\even}{{\on{even}}}
\nc{\SingSupp}{\on{SingSupp}}
\nc{\Supp}{\on{Supp}}

\nc{\temp}{{\mathit{temp}}}
\nc{\geom}{{\mathit{geom}}}
\nc{\ren}{{\mathit{ren}}}
\nc{\naive}{{\mathit{naive}}}
\nc{\conaive}{{\mathit{conaive}}}
\nc{\spec}{\mathit{spec}}

\nc{\gch}{\mathfrak{\check{g}}}

\nc{\Hecke}{\on{Hecke}}

\nc{\LSGch}{{\LS_\Gch}}

\nc{\Hsx}[2]{\H_{{#1} \leftarrow {#2}}}
\nc{\Hdx}[2]{\H_{{#1} \to {#2}}}
\nc{\Hcorr}[3]{ \H_{{#1} \leftarrow {#2} \to {#3}} }
\nc{\Hopcorr}[3]{ \H_{{#1} \to {#2} \leftto {#3}} }

\nc{\Hcat}{\H^{\mathit{cat}}}
\nc{\Hgeom}{\H^{\mathit{geom}}}

\nc{\Hcatsx}[2]{\Hcat_{{#1} \leftarrow {#2}}}
\nc{\Hcatdx}[2]{\Hcat_{{#1} \to {#2}}}
\nc{\Hcatcorr}[3]{ \Hcat_{{#1} \leftarrow {#2} \to {#3}} }
\nc{\Hcatopcorr}[3]{ \Hcat_{{#1} \to {#2} \leftto {#3}} }

\nc{\Hgeomsx}[2]{\Hgeom_{{#1} \leftarrow {#2}}}
\nc{\Hgeomdx}[2]{\Hgeom_{{#1} \to {#2}}}
\nc{\Hgeomcorr}[3]{ \Hgeom_{{#1} \leftarrow {#2} \to {#3}} }
\nc{\Hgeomopcorr}[3]{ \Hgeom_{{#1} \to {#2} \leftto {#3}} }

\nc{\ICohsx}[2]{\ICohW_{{#1} \leftarrow {#2}}}
\nc{\ICohdx}[2]{\ICohW_{{#1} \to {#2}}}
\nc{\ICohcorr}[3]{ \ICohW_{{#1} \leftarrow {#2} \to {#3}} }
\nc{\ICohopcorr}[3]{ \ICohW_{{#1} \to {#2} \leftto {#3}} }

\nc{\QCohsx}[2]{\QCohW_{{#1} \leftarrow {#2}}}
\nc{\QCohdx}[2]{\QCohW_{{#1} \to {#2}}}
\nc{\QCohcorr}[3]{ \QCohW_{{#1} \leftarrow {#2} \to {#3}} }
\nc{\QCohopcorr}[3]{ \QCohW_{{#1} \to {#2} \leftto {#3}} }

\renc{\AA}{\bbA}
\nc{\Asx}[2]{\AA_{{#1} \leftarrow {#2}}}
\nc{\Adx}[2]{\AA_{{#1} \to {#2}}}
\nc{\Acorr}[3]{ \AA_{{#1} \leftarrow {#2} \to {#3}} }
\nc{\Aopcorr}[3]{ \AA_{{#1} \to {#2} \leftto {#3}} }

\nc{\Bsx}[2]{\B_{{#1} \leftarrow {#2}}}
\nc{\Bdx}[2]{\B_{{#1} \to {#2}}}
\nc{\Bcorr}[3]{ \B_{{#1} \leftarrow {#2} \to {#3}} }
\nc{\Bopcorr}[3]{ \B_{{#1} \to {#2} \leftto {#3}} }

\nc{\ICohzero}[3]{\ICoh_0 \bigt{#1 \times_{{#2}_\dR} #3}}
\nc{\IndCohzero}{\ICohzero}

\nc{\form}[3]{#1 \times_{{#2}_\dR} #3 }

\nc{\ind}{{\mathsf{ind}}}
\nc{\oblv}{{\mathsf{oblv}}}

\nc{\Aff}{\mathsf{Aff}}
\nc{\dgAff}{\Aff}

\nc{\deloop}{\mathsf{deloop}}
\renc{\loop}{\mathsf{loop}}

\nc{\coev}{\mathsf{coev}}

\nc{\bE}{\mathbf{E}}

\nc{\ShvCatH}{\ShvCat^{\bbH}}
\nc{\ShvCatQW}{\ShvCat^{\QCohW}}
\nc{\ShvCatQ}{{\ShvCat^\Q}}

\nc{\bbimod}{\on{-}\mathbf{bimod}}

\nc{\Tw}{\mathsf{Tw}}
\nc{\Arr}{\mathsf{Arr}}

\nc{\bDelta}{\bold\Delta}
\nc{\BiCat}{\mathsf{BiCat}}
\nc{\Seg}{\mathsf{Seg}}
\nc{\Cart}{\mathsf{Cart}}
\nc{\Bimod}{\mathsf{Bimod}}
\nc{\lax}{\mathit{lax}}

\nc{\pr}{\mathsf{pr}}

\nc{\zero}{ \{ 0 \}   }
\nc{\Perf}{\mathsf{Perf}}

\nc{\leftto}{\leftarrow}
\nc{\lto}{\leftto}
\nc{\xlto}[1]{\xleftarrow{#1}}

\nc{\ltemp}{{}^\temp}
\nc{\TwCorr}{\mathsf{TwCorr}}

\nc{\Affover}[1]{{\Aff_{/#1}}}
\nc{\Affoverop}[1]{{( \Affover{#1})^\op}}

\nc{\AffOver}[2]{{(\Aff_{#1})_{/#2}}}

\nc{\AffOverop}[2]{{( \AffOver{#1}{#2})^\op}}

\nc{\aft}{{\mathit{aft}}}

\renc{\vert}{{\mathit{vert}}}
\nc{\horiz}{{\mathit{horiz}}}
\nc{\type}{{\mathit{type}}}
\nc{\adm}{{\mathit{adm}}}

\nc{\g}{\mathfrak{g}}

\nc{\free}{\mathsf{free}}

\nc{\Sform}{{S \times_{S_\dR} S}}
\nc{\Yform}{{\Y \times_{\Y_\dR} \Y}}
\nc{\SdR}{ {S_{\dR}}}

\nc{\laft}{{\mathit{laft}}}

\nc{\Affevcocaft}{\Aff_{\aft}^{< \infty}}
\nc{\Affaftevcoc}{\Aff_{\aft}^{< \infty}}

\nc{\Affevcoclfp}{\Aff_{\lfp}^{< \infty}}

\nc{\Schevcoclfp }{\Sch_{\lfp}^{< \infty}}
\nc{\Schevcocaft}{\Sch_{\aft}^{< \infty}}
\nc{\Schaftevcoc}{\Sch_{\aft}^{< \infty}}

\nc{\Stkevcoclfp}{\Stk_{\lfp}^{< \infty}}
\nc{\Stkevcoc}{\Stk^{< \infty}}

\nc{\evcoc}{\mathit{bdd}}

\nc{\ICoh}{\IndCoh}

\nc{\citep}{\cite}

\renc{\H}{\bbH}

\nc{\uno}{\mathbbm{1}}

\nc{\CohBig}{{\Coh^{-\infty}}}

\nc{\Tang}{\mathbb{T}}

\nc{\LieAlg}{\mathsf{LieAlg}}

\nc{\Serre}{{\on{Serre}}}

\nc{\MPreStk}{\CoeffSys}
\nc{\CoeffSys}{\mathsf{CoeffSys}}

\nc{\all}{{\on{all}}}

\nc{\QCohwedge}{\bbQ^\wedge}
\nc{\ICohwedge}{\bbI^\wedge}
\nc{\ICohW}{\ICohwedge}
\nc{\QCohW}{\QCohwedge}

\nc{\ShvCatA}{\ShvCat^{\AA}}
\nc{\ShvCatB}{{\ShvCat^\B}}

\nc{\naiveto}{{\xto{\naive}}}
\nc{\conaiveto}{{\xto{\conaive}}}

\nc{\strong}{\mathit{strong}}
\nc{\costrong}{\mathit{costrong}}

\nc{\conv}{\mathit{conv}}

\nc{\Q}{\bbQ}

\nc{\bY}{\mathbf{Y}}
\nc{\Loop}{\mathsf{LOOP}}

\nc{\DG}{{\on{DG}}}

\nc{\coind}{\mathsf{coind}}

\nc{\co}{\on{co}}

\nc{\laftdef}{{\mathit{laft-def}}}

\nc{\qsmooth}{{\mathit{q-smooth}}}
\nc{\smooth}{{\mathit{smooth}}}

\nc{\LKE}{\on{LKE}}
\nc{\RKE}{\on{RKE}}

\nc{\ShvCatAco}{\ShvCatA_{\co}}
\nc{\ShvCatHco}{\ShvCatH_{\co}}

\nc{\Stk}{\mathsf{Stk}}
\nc{\Stklfp}{\Stk_{\lfp}}

\renc{\2}{(\infty,2)}
\renc{\1}{(\infty,1)}

\nc{\doubleCat}{\mathsf{doubleCat}}
\nc{\Spaces}{\mathcal{S}\!\mathit{paces}}

\nc{\ALG}{\mathsf{ALG}}
\nc{\MAPS}{\mathsf{MAPS}}

\nc{\CAT}{\mathsf{CAT}}

\nc{\oneCat}{{\Cat_{\1}}}
\nc{\oneCAT}{{\CAT_{\1}}}

\nc{\twoCat}{{\Cat_{\2}}}
\nc{\twoCAT}{{\CAT_{\2}}}

\nc{\DGCAT}{\mathsf{DGCAT}}
\nc{\twoCatDG}{{\CAT_{\2}^\DG}}
\nc{\twoCATDG}{{\CAT_{\2}^\DG}}

\nc{\twoCATDGw}{{\CAT_{\2, w*}^\DG}}
\nc{\twoCATDGww}{{\CAT_{\2, ww*}^\DG}}

\nc{\AlgBimod}{\Alg^{\mathit{bimod}}}
\nc{\AlgBimodDGCat}{\AlgBimod(\DGCat)}
\nc{\ALGBimod}{\ALG^{\mathit{bimod}}}
\nc{\twoAlgBimod}{\ALGBimod}

\nc{\rev}{{\on{rev}}}

\nc{\lfp}{{\mathit{lfp}}}

\nc{\RBeck}{{\on{R-BC}}}
\nc{\LBeck}{{\on{L-BC}}}

\nc{\schem}{\mathit{schem}}
\nc{\proper}{\mathit{proper}}

\nc{\UQ}{\U^{\QCoh}}

\nc{\FormMod}{\mathsf{FormMod}}
\nc{\FormModunderStkevcoc}{\FormMod_{\Stkevcoc /}^\lfp}

\nc{\acts}{\mathrel{\reflectbox{$\circlearrowright$}}}

\nc{\irred}{\mathit{irred}}
\nc{\cusp}{\mathit{cusp}}

\nc{\Gm}{{\GG_m}}

\nc{\cl}{{\mathit{cl}}}

\nc{\vDmod}{\virg{\Dmod}}
\nc{\LY}{L\Y}

\nc{\MOD}{\on - \mathbf{MOD}}
\nc{\SHVCAT}{\on{SHVCAT}}

\nc{\EEis}{\bbE\!\on{is}}

\begin{document}

\title{Sheaves of categories with local actions of Hochschild cochains}
\author{Dario Beraldo}
%%%%%%%%\email{darioberaldo@gmail.com}
%%%%\address{Institut de math\'ematiques de Toulouse, Universit\'e Paul Sabatier, 118 Route de Narbonne, 31400 Toulouse, France}

%\classification{}

\begin{abstract}
The notion of \emph{Hochschild cochains} induces an assignment from $\Aff$, affine DG schemes, to monoidal DG categories.
We show that this assignment extends, under appropriate finiteness conditions,
to a functor $\H: \Aff \to \AlgBimod(\DGCat)$, where the latter denotes the category of monoidal DG categories and bimodules.
Any functor $\A: \Aff \to \AlgBimod(\DGCat)$ gives rise, by taking modules, to a theory of sheaves of categories $\ShvCatA$.

In this paper, we study $\ShvCatH$. Vaguely speaking, this theory categorifies the theory of $\fD$-modules, in the same way as Gaitsgory's original $\ShvCat$ categorifies the theory of quasi-coherent sheaves. We develop the functoriality of $\ShvCatH$, its descent properties and the notion of $\H$-affineness.

We then prove the $\H$-affineness of algebraic stacks: for $\Y$ a stack satisfying some mild conditions, the $\infty$-category $\ShvCatH(\Y)$ is equivalent to the $\infty$-category of modules for $\H(\Y)$, the monoidal DG category of higher differential operators.
The main consequence, for $\Y$ quasi-smooth, is the following: if $\C$ is a DG category acted on by $\H(\Y)$, then $\C$ admits a theory of singular support in $\Sing(\Y)$, where $\Sing(\Y)$ is the space of singularities of $\Y$.

As an application to the geometric Langlands program, we indicate how derived Satake yields an action of $\H(\LSGch)$ on $\Dmod(\Bun_G)$, thereby equipping objects of $\Dmod(\Bun_G)$ with singular support in $\Sing(\LSGch)$.

\end{abstract}

\maketitle

%\tableofcontents

\sec{Introduction}

\ssec{Overview}

The present paper is a contribution to the field of \emph{categorical algebraic geometry}. 
In this field, one studies schemes and stacks via their categorical invariants, as opposed to their usual linear invariants.
Among the usual invariants, typical examples are the coherent cohomology, the de Rham cohomology, the Picard group.
An example of a categorical invariant is the symmetric monoidal category of quasi-coherent sheaves; other examples, including the invariant $\ShvCatH$ appearing in the title of this paper, will be given below.

%%%
%%%The difference between categorical and ordinary algebraic geometry is just a change in point of view: in the former context, one treats the usual categories appearing in algebraic geometry (most notably: categories of sheaves of a certain kind on schemes or stacks) as interesting objects in their own right; in the latter context, the very same categories are often regarded just as convenient containers of the objects of interest.

The extra level of categorical abstraction might appear unjustified at first sight, but it turns out to be quite useful in several concrete situations. 
In this paper, we will encounter a few: for instance, in Section \ref{sssec:consequences of H action}, Section \ref{sssec:second consequence} and Section \ref{ssec:LS}.

The interplay between categorical and ordinary algebraic geometry is likely to be very fruitful. For more on the comparison between the two points of view, we recommend the discussion and the dictionary appearing in \cite[Page 720]{Lurie:SAG}.

In the rest of this overview, after discussing some illuminating examples, we will roughly state the goals and the main results of this paper. These results and goals will be further clarified in the later sections of the introduction.

\sssec{}

As mentioned earlier, given a scheme or an algebraic stack $\Y$, its most basic categorical invariant is the symmetric monoidal differential graded (DG) category $\QCoh(\Y)$.

It turns out that there are strong analogies between the behaviour of $\QCoh(\Y)$ for an algebraic stack $\Y$ and the behaviour of $H^*(Y,\O_Y)$ for an affine scheme\footnote{We will soon be forced to consider DG schemes. By construction, the cohomology $H^*(Y,\O_Y)$ of an affine DG scheme is possibly nonzero in negative degrees: this explains the notation $H^*(Y,\O_Y)$ in place of the more tempting $H^0(Y,\O_Y)$.} 
 $Y$. In other words, categorical algebraic geometry has many more affine objects than ordinary algebraic geometry. Let us illustrate this principle with three examples.

\sssec{Tannaka duality.}

For $\Y$ an algebraic stack satisfying mild conditions, 
Tannaka duality (\cite[Chapter 9]{Lurie:SAG}) allows to \virg{recover} $\Y$ from the symmetric monoidal DG category $\QCoh(\Y)$. On the other hand, the DG algebra $H^*(\Y,\O_\Y)$ does not recover $\Y$, unless $\Y$ is an affine DG scheme.

\sssec{Tensor products.} \label{sssec:example-tensor product}

Given a diagram $X \to Z \leftto Y$ of (DG) affine schemes, one has
$$
H^*(X \times_Z Y, \O_{X \times_Z Y})
\simeq
H^*(X, \O_X) \underset{H^*(Z,\O_Z)}{\otimes^\mathbf{L}} H^*(Y,\O_Y).
$$
Note that it is essential that the fiber product be taken in the derived sense. This formula obviously fails for very simple non-affine schemes and stacks.
On the other hand, the categorical counterpart is the tensor product formula
\begin{equation} \label{eqn:overview-tensor-product}
\QCoh(\X \times_\Z \Y)
\simeq
\QCoh(\X )\usotimes{\QCoh(\Z)} \QCoh(\Y),
\end{equation}
which holds true for most algebraic stacks $\X$, $\Y$, $\Z$ that one encounters in practice: see, for instance, \cite{BFN}.

The RHS of the above formula involves the tensor product of DG categories (\cite{Lurie:HA}), which plays a crucial role in the theory. Note that $\QCoh(\Z)$ acts on $\QCoh(\X)$ and on $\QCoh(\Y)$ by pullback along the given maps $\X\xto f \Z\xleftarrow g \Y$.

\sssec{$1$-affineness.} \label{sssec:example-1-affineness}

In the categorical context, one considers \emph{categorified quasi-coherent sheaves} over a scheme or a stack $\Y$.
These categorified sheaves are defined in \cite{shvcat} under the name of \virg{sheaves of categories}, and in \cite[Chapter 10]{Lurie:SAG} under the name of \virg{quasi-coherent stacks}.
They assemble into an $\infty$-category denoted $\ShvCat(\Y)$.
We will recall and generalize the notion of $\ShvCat$ in Section \ref{ssec:shvcat-intro}.

In the above papers, it is proven that most algebraic stacks, while far from being affine schemes, are nevertheless \emph{$1$-affine}: by definition, $\Y$ is $1$-affine if the $\infty$-category $\ShvCat(\Y)$ is equivalent to the $\infty$-category of modules DG categories for $\QCoh(\Y)$.
This categorifies the classical fact that, for $Y$ an affine DG scheme, a quasi-coherent sheaf is the same as a module over $H^*(Y,\O_Y)$. 

\sssec{}

The above examples illustrate the point of view that $\QCoh(\Y)$ is the categorical counterpart of the algebra of  functions on an \emph{affine} DG scheme.

In \cite{centerH}, we introduced another monoidal DG category, $\H(\Y)$, which is the categorical counterpart of the algebra of differential operators on an affine DG scheme.

In a nutshell, the goal of the present paper is to develop the tensor product formula and the $1$-affineness result with $\H(\Y)$ in place of $\QCoh(\Y)$. 

%%%
%%% to render the two examples of Sections \ref{sssec:example-tensor product} and \ref{sssec:example-1-affineness} 
%%% 
%%% 
%%% (that is, the tensor product formula and the $1$-affineness result) to another monoidal DG category, $\H(\Y)$, which can be regarded as the categorification of the algebra of differential operators on an affine DG scheme.

\sssec{Tensor products for $\H$.}

The tensor product formula in the $\H$-situation is by necessity slightly different from \eqref{eqn:overview-tensor-product}. Indeed, as explained in detail later, there is no natural action of $\H(\Z)$ on $\H(\X)$. Rather, these two monoidal DG categories are connected by a \emph{transfer} bimodule category $\Hdx \X \Z$. (This is in perfect agreement with the situation of rings of differential operators, from which the notation is borrowed.) Under some conditions to be discussed later, the tensor product formula reads:
$$
\Hsx{\X}{\X \times_\Z \Y} 
\usotimes{\H(\X \times_\Z \Y)}
\Hdx{\X \times_\Z \Y}\Y
\simeq
\Hdx \X \Z 
\usotimes{\H(\Z)}
\Hsx \Z \Y.
$$
For some pleasing applications of this formula, the reader might look ahead at Sections \ref{ssec:Harish-Chandra} and \ref{ssec:LS}.

\sssec{$1$-affineness for $\H$ (aka, $\H$-affineness).}

The $1$-affineness mentioned in Section \ref{sssec:example-1-affineness} corresponds, in the $\H$-setup, to our main Theorem \ref{main-thm-intro}, which establishes a tight link between modules categories for $\H(\Y)$ and \emph{categorified D-modules} on $\Y$.
The latter are also called \emph{sheaves of categories over $\Y$ with local actions of Hochschild cochains}, and denoted by $\ShvCatH(\Y)$. As we explain in the following sections, the objects of $\ShvCatH(\Y)$ are the sheaves of categories for which a notion of singular support is defined and well-behaved.

\ssec{Singular support via the $\H$-action}

\sssec{}

In \cite{centerH}, we introduced a monoidal DG category $\H(\Y)$ attached to a quasi-smooth stack $\Y$.
Contrarily to $\QCoh(\Y)$, which can defined in vast generality, the construction of $\H(\Y)$ requires some (mild) conditions on $\Y$.
The definition of $\H(\Y)$ and the necessary conditions on $\Y$ are recalled in Section \ref{ssec:definition of HY}. For now, let us just say that  any quasi-smooth stack $\Y$ satisfies those conditions.

\sssec{}

As a brief reminder of the notion of quasi-smoothness: an algebraic stack $\Y$ is quasi-smooth if it is smooth-locally a global complete intersection. It follows that, for any geometric point $y \in \Y$, the $y$-fiber $\LL_{\Y,y}:= \restr{\LL_\Y}y$ of the contangent complex has cohomologies concentrated in degrees $[-1,1]$.

Thus, to a quasi-smooth stack $\Y$, we associate the stack $\Sing(\Y)$ that parametrizes pairs $(y, \xi)$ with $y \in \Y$ and $\xi \in H^{-1}(\LL_{\Y,y})$. This is the space that controls the singularities of $\Y$, see \cite{AG1}, and it is equipped with a $\Gm$-action that rescales the fibers of the projection $\Sing(\Y) \to \Y$.

\sssec{}

Suppose that a DG category $\C$ carries an action of $\H(\Y)$.
The goal of this paper is to explain how rich this structure is. 
As an example, let us informally state here the most important consequence of our main results:

\begin{thm} \label{vague-thm:sing-support}
Let $\Y$ be a quasi-smooth stack and $\C$ a left $\H(\Y)$-module. Then $\C$ is equipped with a \emph{singular support theory relative to $\Sing(\Y)$}.
\end{thm}

\sssec{}

To make sense of this, we need to explain what we mean by \virg{singular support theory}. First and foremost, this means that there is a map (the singular support map) from objects of $\C$ to closed conical subsets of $\Sing(\Y)$. For each such subset $\N \subseteq \Sing(\Y)$, we set $\C_\N$ to be the full subcategory of $\C$ spanned by those objects with singular support contained in $\N$. 

The second feature of a singular support theory is that any inclusion $\N \subseteq \N'$ yields a colocalization (that is, an adjunction whose left adjoint is fully faithful) $\C_\N \rightleftarrows \C_{\N'}$.

\sssec{} \label{sssec:consequences of H action}

Thus, the datum of an action of $\H(\Y)$ on $\C$ immediately produces a multitude of semi-orthogonal decompositions of $\C$, one for each closed conical subset of $\Sing(\Y)$. 
Obviously, these decompositions help compute Hom spaces between objects of $\C$.

More generally, the philosophy\footnote{Strictly speaking, this is not a consequence of the results of this paper. We refer to the analysis of \cite{gluing}.}
 is that, in the presence of an $\H(\Y)$-action on $\C$, any decomposition of $\Sing(\Y)$ into \emph{atomic blocks} induces a decomposition of $\C$ into atomic blocks. 
By \virg{atomic blocks}, we mean closed conical subsets of $\Sing(\Y)$ that are of a particular significance or simplicity; for instance: the zero section, a particular fiber, or more generally the conormal bundle of a closed subset of $\Y$.
See \cite{AG2, gluing} for applications of this principle.

\sssec{}

It is also natural to require that singular support be functorial in $\C$. Namely, given an $\H(\Y)$-linear functor $F:\C \to \D$ and $\N \subseteq \Sing(\Y)$, we would like $F$ to restrict to a functor $\C_\N \to \D_\N$. Fortunately, this is also guaranteed by our theory.  Hence the statement of informal Theorem \ref{vague-thm:sing-support} could be improved as follows.

\begin{thm} \label{vague-thm:sing-support-IMPROVED}
For $\Y$ a quasi-smooth stack, $\H(\Y)$-module categories admit a singular support theory relative to $\Sing(\Y)$.
\end{thm}

\begin{rem}
The proof of this theorem is an easy consequence of the construction of $\H(\Y)$ (namely, the relation with Hochschild cochains as in Section \ref{ssec:Hochschild}) and our $\H$-affineness theorem, Theorem \ref{main-thm-intro}.
\end{rem}

\begin{rem}
Our expectation on possible usages of this theorem is the following. It is generally difficult to directly equip $\C$ with a  singular support theory relative to $\Sing(\Y)$; instead, one should try to exhibit an action of $\H(\Y)$ on $\C$. In Section \ref{ssec:Hecke}, we will illustrate a concrete application of this point of view on the geometric Langlands program.
\end{rem}

\sssec{}

There exists a monoidal functor $\QCoh(\Y) \to \H(\Y)$: hence, an $\H(\Y)$-action on $\C$ means in particular that $\C$ admits a $\QCoh(\Y)$-action. Thus, our theorem above can be regarded as an improvement of the following one in the setting of quasi-smooth stacks.

\begin{thm}
Let $\Y$ be an algebraic stack (not necessarily quasi-smooth).Then left $\QCoh(\Y)$-modules are equipped with a \emph{support theory} relative to $\Y$.
\end{thm}

\ssec{The monoidal category $\H(\Y)$} \label{ssec:definition of HY}

Let us now recall the elements that go into the definition of $\H(\Y)$, following \cite{AG2} and \cite{centerH}. Although the applications of this theory so far concern only $\Y$ quasi-smooth, the natural setup for $\H(\Y)$ is more general. Namely, we assume that $\Y$ is a quasi-compact algebraic stack which is perfect, bounded\footnote{alias: eventually coconnective} and locally of finite presentation (lfp). See \cite{BFN} for the notion of \virg{perfect stack}.

\sssec{}

The definition of $\H$ requires some familiarity with the theory of ind-coherent sheaves on formal completions. We refer to \cite[Chapter III]{Book}, or to \cite{centerH} for a quick review. 

Nevertheless, let us recall the most important concepts. First, $\Y_\dR$ denotes the \emph{de Rham} prestack of $\Y$, whence $\Yform$ is the formal completion of the diagonal $\Delta: \Y \to \Y \times \Y$.
Second, we have the standard functor
$$
\Upsilon_\Y:
\QCoh(\Y)
\longto
\ICoh(\Y),
$$
which is the functor of acting on the dualizing sheaf $\omega_\Y \in \ICoh(\Y)$. The boundedness condition on $\Y$ is imposed so that $\Upsilon_\Y$ is fully faithful.

\sssec{}

We define $\H(\Y)$ to be the full subcategory of $\ICoh(\Yform)$ cut out by the requirement that the image of the pullback functor $\Delta^!: \ICoh(\Yform) \to \ICoh(\Y)$ be contained in the subcategory $\Upsilon_\Y(\QCoh(\Y)) \subseteq \ICoh(\Y)$.
Now, $\ICoh(\Yform)$ has a monoidal structure given by convolution, that is, pull-push along the correspondence
$$
\Yform \times \Yform
\xleftarrow{p_{12} \times p_{23}}
\form \Y \Y \Y \times_{\Y_\dR} \Y \xto{p_{13}} \form \Y \Y \Y.
$$
The lfp assumption on $\Y$ is crucial: it ensures that $\H(\Y)$ is preserved by this multiplication, thereby inheriting a monoidal structure.

\begin{example}

Of course, $\H(\Y)$ admits two obvious module categories: $\ICoh(\Y)$ and $\QCoh(\Y)$.
For $\ICoh(\Y)$, the theory of singular support of Theorem \ref{vague-thm:sing-support} reduces to the one developed by \cite{AG1} and before by \cite{BIK}.

\end{example}

\begin{example}

By \cite{AG1}, objects of $\QCoh(\Y)$ have singular support contained in the zero section of $\Sing(\Y)$: in our language, this is expressed  by the fact that the action of $\H(\Y)$ on $\QCoh(\Y)$ factors through the monoidal localization 
$$
\H(\Y)
\tto
\QCoh(\Yform).
$$
The construction and the study of this monoidal localization is deferred to another publication.
For now, let us say that we will call $\C \in \H(\Y) \mmod$ \emph{tempered} if the $\H(\Y)$-action factors through the above monoidal quotient.

\end{example}

\ssec{$\H$ for Hecke} \label{ssec:Hecke}

In this section, we anticipate a future application of Theorem \ref{vague-thm:sing-support}.
The reader not interested in geometric Langlands might well skip ahead to Section \ref{ssec:Hochschild}.

\sssec{} \label{sssec:second consequence}

Let us recall the rough statement of the geometric Langlands conjecture, see \cite{AG1}: there is a canonical equivalence $\Dmod(\Bun_G) \simeq \ICoh_\N(\LSGch)$.
This conjecture predicts in particular that any $\F \in \Dmod(\Bun_G)$ has a (nilpotent) singular support in $\Sing(\LSGch)$. The question that prompted the writing of this paper and the study of $\H$ is the following: \emph{is it possible to exhibit this structure on $\Dmod(\Bun_G)$ independently of the geometric Langlands conjecture?}

\medskip

Having such a notion is evidently desirable, as it allows to cut out $\Dmod(\Bun_G)$ into several subcategories by imposing singular support conditions.
For instance, the zero section $O_{\LSGch} \subseteq \Sing(\LSGch)$ ought to give rise to the DG category $\Dmod(\Bun_G)_{O_{\LSGch}}$ of \emph{tempered} $\fD$-modules.

\sssec{} \label{sssec:discussion of vanishing}

Our Theorem \ref{vague-thm:sing-support} gives a way to answer the above question. We make the following claim, which we plan to address elsewhere: \emph{there is a canonical action of $\H(\LSGch)$ on $\Dmod(\Bun_G)$.}

Modulo technical and foundational details, the construction of such action goes as follows:
\begin{itemize}
\item
consider the action of the \emph{renormalized}\footnote{See \cite[Section 12.2.3]{AG1} for the pointwise (as opposed to $\Ran$) version} spherical category $\Sph^\ren_{G,\Ran}$ on $\Dmod(\Bun_G)$;
\item
\emph{derived geometric Satake} over $\Ran$ yields a monoidal equivalence between $\Sph^\ren_{G,\Ran}$ and the (not yet defined) convolution monoidal DG category
$$
\Sph^{\spec, \ren}_{\Gch, \Ran}
:=
\ICoh
\Bigt{
\bigt{
 \LSGch(D) \times_{\LSGch(D^\times) } \LSGch(D) 
 }^\wedge_{\LSGch(D)}
 }_\Ran;
$$

\item
the argument of \cite{Roz-thesis} yields a monoidal localization 
$$
\Sph^{\spec, \ren}_{\Gch, \Ran}
\tto 
\H(\LSGch),
$$
with kernel denoted by $\K$;
\item
now consider the \emph{spherical category} $\Sph^{\spec, \naive}_{G, \Ran}$, the monoidal localization 
$$
\Sph^{\spec,\naive}_{G, \Ran} \tto \QCoh(\LSGch)
$$
with kernel denoted $\K^{\naive}$, and the monoidal functor
$$
\Sph^{\spec, \naive}_{G, \Ran} \longto \Sph^{\spec, \ren}_{\Gch, \Ran};
$$
\item
by construction, the essential image of the resulting functor $\K^{\naive} \to \K$ generates the target under colimits;
\item
the \emph{vanishing theorem} (\cite{Outline}) states that objects of $\K^{\naive}$ act by zero on $\Dmod(\Bun_G)$, whence the same is true for objects of $\K$: in other words, the $\Sph^\ren_{G,\Ran}$-action on $\Dmod(\Bun_G)$ factors through an action of $\H(\LSGch)$.
\end{itemize}

In particular, the construction implies that $\H(\LSGch)$ acts on $\Dmod(\Bun_G)$ by Hecke functors.

\ssec{$\H$ for Hochschild} \label{ssec:Hochschild}

To motivate the definition of $\H(\Y)$ and to explain the connection with singular support, it is instructive to look at the case $\Y =S$ is an affine DG scheme. Under our standing assumptions, $S$ is of finite type, bounded and with perfect cotangent complex. (Hereafter, we denote by $\Affevcoclfp$ the $\infty$-category of such affine schemes.)
In this case, the monoidal category $\H(S)$ is very explicit: it is the monoidal DG category of right modules over the $E_2$-algebra 
$$
\HC(S) := \End_{\QCoh(S \times S)}(\Delta_*(\O_S))
$$
of Hochschild cochains on $S$.
Under the equivalence $\H(S) \simeq \HC(S)^\op \mod$, the monoidal functor $\QCoh(S) \to \H(S)$ corresponds to induction along the $E_2$-algebra map $\Gamma(S,\O_S) \to \HC(S)^\op$.

\sssec{}

From this description, one observes that Theorem \ref{vague-thm:sing-support} is obvious in the affine case. 
Indeed, as we have just seen, the datum of $\C \in \H(S) \mmod$ means that $\C$ is enriched over $\HC(S)^\op$. Now, the HKR theorem yields a graded algebra map
$$
\Sym_{H^0(S,\O_S)} (H^1(\Tang_S)[-2])
\longto
\HH^\bullet(S),
$$ 
and, by definition, singular support for objects of $\C$ is computed just using the action of the LHS on $H^\bullet(\C)$.

\sssec{} \label{sssec:hierarchy}

In summary,  there is a hierarchy of structures that a DG category $\C$ might carry:
\begin{itemize}
\item
an action of the $E_2$-algebra $\HC(S)^\op$;
\item
an action of the commutative graded algebra $\Sym_{H^0(S,\O_S)} H^1(S, \Tang_S)[-2]$ on $H^\bullet(\C)$;
\item
an action of the commutative algebra ${H^0(S,\O_S)}$ on $H^\bullet(\C)$. 
\end{itemize}
The first two data endow objects of $\C$ with singular support, which is a closed conical subset of $\Sing(S)$, see \cite{AG1}. The third datum only allows to define ordinary support in $S$.

\ssec{Sheaves of categories} \label{ssec:shvcat-intro}

Next, we would like to generalize the above constructions to non-affine schemes and then to algebraic stacks.
The key hint is that singular support of quasi-coherent and ind-coherent sheaves can be computed smooth locally. Thus, we hope to be able to glue the local $\HC$-actions as well.

\sssec{}

The first step towards this goal is to understand the functoriality of $\H(S) \mmod$ along maps of affine schemes. This is not immediate, as $\HC(S)$ is not functorial in $S$. In particular, for $f:S \to T$ a morphism in $\Affevcoclfp$, there is no natural monoidal functor between $\H(T)$ and $\H(S)$.
However, these two monoidal categories are connected by a canonical bimodule 
$$
\Hdx ST
:= \ICoh_0((S \times T)^\wedge_S).
$$

\begin{example}
Observe that $\Hdx S\pt \simeq \QCoh(S)$, and $\Hdx SS = \H(S)$.
\end{example}

\sssec{} \label{intro-sssec:H on affine}

Moreover, for any string $S \to T \to U$ in $\Affevcoclfp$, there is a natural functor
\begin{equation} \label{eqn:lax-composition-H}
\Hdx ST \usotimes{\H(T)} \Hdx TU
\longto
\Hdx SU,
\end{equation}
given by convolution along the obvious correspondence
$$
(S \times T)^\wedge_S
\times 
(T \times U)^\wedge_T
\longleftarrow
(S \times T \times U)^\wedge_S
\longto
(S \times U)^\wedge_S.
$$
We will prove in Theorem \ref{thm:H-strict-on-schemes} that (\ref{eqn:lax-composition-H}) is an equivalence of  $(\H(S), \H(U))$-bimodules. It follows that the assignment $[S \to T] \squigto \Hdx ST$ upgrades to a functor
$$
\H: 
\Affevcoclfp
\longto
\AlgBimod(\DGCat),
$$
where $\AlgBimod(\DGCat)$ is the $\infty$-category whose objects are monoidal DG categories and whose morphisms are bimodules.

\sssec{}

A functor 
$$
\A: \Aff \to \AlgBimod(\DGCat)
$$ 
(or a slight variation, e.g. the functor $\H:\Affevcoclfp \to \AlgBimod(\DGCat)$) will be called a \emph{coefficient system} in this paper.
Informally, $\A$ consists of the following pieces of data:
\begin{itemize}
\item
for an affine scheme $S$, a monoidal DG category $\A(S)$;
\item 
for a map of affine schemes $f: S \to T$, an $(\A(S), \A(T))$-bimodule $\Adx ST$;
\item
for any string of affine schemes $S \to T \to U$, an $(\A(S), \A(U))$-bilinear equivalence
\begin{equation} 
\nonumber
\Adx ST \usotimes{\A(T)} \Adx TU
\longto
\Adx SU,
\end{equation}
\item
a system of coherent compatibilities for higher compositions.
\end{itemize}
The reason for the terminology is that each $\A$ is the coefficient system for a sheaf of categories attached to it. More precisely, the datum of $\A$ as above allows to define a functor
$$
\ShvCat^\A: 
\PreStk^\op
\longto
\inftyCat
$$
as follows:
\begin{itemize}
\item
for $S$ affine, we set $\ShvCatA(S) = \A(S) \mmod$;

\item
for $f: S \to T$ a map in $\Aff$, we have a structure pullback functor
$$
f^{*,\A}:
\ShvCatA(T) = \A(T) \mmod
\xto{\Adx ST \usotimes{\A(T) } -}
\ShvCatA(S) = \A(S) \mmod;
$$

\item
for $\Y$ a prestack, we define $\ShvCatA(\Y)$ be right Kan extension along the inclusion $\Aff \hto \PreStk$, that is, 
$$
\ShvCatA(\Y)
=
\lim_{S \in (\Aff_{/\Y})^\op}
\A(S) \mmod.
$$
\end{itemize}
Thus, an object of $\ShvCatA(\Y)$ is a collection of $\A(S)$-modules $\C_S$, one for each $S$ mapping to $\Y$, together with compatible equivalences $\Adx ST \otimes_{\A(T)} \C_T \simeq \C_S$.

\begin{example}

The easiest nontrivial example of coefficient system is arguably the one denoted by $\Q$ and defined as 
$$
\Q(S) := \QCoh(S), 
\hspace{.4cm}
\Q_{S \to T}
:= 	\QCoh(S) \in (\QCoh(S), \QCoh(T)) \bbimod.
$$
The theory of sheaves of categories associated to $\Q$ is the \virg{original one}, developed by D. Gaitsgory in \cite{shvcat}. In \loccit, such theory was denoted by $\ShvCat$; in this paper, for the sake of uniformity, we will instead denote it by $\ShvCatQ$.
\end{example}

\begin{example}

Parallel to the above, consider the coefficient system $\DD: \Aff_\aft \to \AlgBimod(\DGCat)$ defined by
$$
\DD(S) := \Dmod(S), 
\hspace{.4cm}
\DD_{S \to T}
:= 	\Dmod(S) \in (\Dmod(S), \Dmod(T)) \bbimod.
$$
The theory $\ShvCat^\DD$ is the theory of \emph{crystals of categories}, also discussed in \cite{shvcat}.
\end{example}

\begin{rem} \label{rem:analogies}
The following list of analogies is sometimes helpful: $\ShvCatQ$ categorifies quasi-coherent sheaves, $\ShvCat^\DD$ categorifies locally constant sheaves, $\ShvCatH$ categorifies $\fD$-modules.
\end{rem}

\ssec{$\H$-affineness}

In line with the first of the above analogies, the foundational paper \cite{shvcat} constructs an explicit adjunction
\begin{equation} \label{adj:Loc-bGamma-intro}
\nonumber
\begin{tikzpicture}[scale=1.5]
\node (a) at (0,1) {$\bLoc_\Y: \QCoh(\Y) \mmod $};
\node (b) at (3,1) {$\ShvCatQ(\Y): \bGamma_\Y$.};
\path[->,font=\scriptsize,>=angle 90]
([yshift= 1.5pt]a.east) edge node[above] { } ([yshift= 1.5pt]b.west);
\path[->,font=\scriptsize,>=angle 90]
([yshift= -1.5pt]b.west) edge node[below] {} ([yshift= -1.5pt]a.east);
\end{tikzpicture}
\end{equation}
In line with the analogy again, a prestack $\Y$ is said to be $1$-affine if these adjoints are mutually inverse equivalences.
This is tautologically true in the case $\Y$ is an affine scheme. However, there are several other examples: most notably many algebraic stacks (precisely, quasi-compact bounded algebraic stacks of finite type and with affine diagonal) are $1$-affine, see \cite[Theorem 2.2.6]{shvcat}.

For the sake of uniformity, we take the liberty to rename \virg{$1$-affineness} with \virg{$\Q$-affineness}.

\sssec{}

One of our main constructions is the adjunction
\begin{equation} \label{adj:Loc-bGammaH-intro}
\begin{tikzpicture}[scale=1.5]
\node (a) at (0,1) {$\bLoc^\H_\Y: \H(\Y) \mmod $};
\node (b) at (3,1) {$\ShvCatH(\Y): \bGamma^\H_\Y$,};
\path[->,font=\scriptsize,>=angle 90]
([yshift= 1.5pt]a.east) edge node[above] { } ([yshift= 1.5pt]b.west);
\path[->,font=\scriptsize,>=angle 90]
([yshift= -1.5pt]b.west) edge node[below] {} ([yshift= -1.5pt]a.east);
\end{tikzpicture}
\end{equation}
sketched below (and discussed thoroughly in Section \ref{ssec:localiz-global-sect-H-affineness}). Contrarily to the $\Q$-case, in the $\H$-case we do not allow $\Y$ to be an arbitrary prestack, but we need $\Y$ to be an algebraic stack satisfying the conditions that make $\H(\Y)$ well defined, see Section \ref{ssec:definition of HY}.

\sssec{}

The definition of the left adjoint $\bLoc^\H_\Y$ is easy. For a map $S \to \Y$ with $S \in \Affevcoclfp$, look at the $(\H(S), \H(\Y))$-bimodule $\Hdx S\Y := \ICoh_0((S \times \Y)^\wedge_S)$. Given $\C \in \H(\Y) \mmod$, we form the $\H$-sheaf of categories 
$$
\bLoc^\H_\Y(\C)
:=
\{
\Hdx S\Y
\usotimes{\H(\Y)}
\C
\}_S.
$$
To define the right adjoint $\bGamma^\H_\Y$, we need to make sure that each bimodule $\Hdx S\Y$ admits a right dual. Such right dual exists and it is fortunately the obvious $(\H(\Y), \H(S))$-bimodule 
$$
\Hsx \Y S := \ICoh_0((\Y \times S)^\wedge_S).
$$
From this, it is straightforward to see that 
$$
\bGamma^\H_\Y (\{\E_S\}_S)
\simeq
\lim_{S \in (\Affevcoclfp)^\op}
\Hsx \Y S \usotimes{\H(S)} \E_S,
$$
with its natural left $\H(\Y)$-module structure.

\sssec{}

Our main theorem reads:

\begin{thm} \label{main-thm-intro}
Any $\Y \in \Stkevcoclfp$ is \emph{$\H$-affine}, that is, the adjoint functors in (\ref{adj:Loc-bGammaH-intro}) are equivalences. 
\end{thm}

In the rest of this introduction, we will explain our two applications of this theorem: the relation with singular support as in Theorem \ref{vague-thm:sing-support}, and the functoriality of $\H$ for algebraic stacks.

\ssec{Change of coefficients} \label{ssec:change of coeff}

Coefficient systems form an $\infty$-category.
By definition, a morphism $\A \to \B$ consists of an $(\A(S), \B(S))$-bimodule $M(S)$ for any $S \in \Aff$, and of a system of compatible equivalences
\begin{equation} \label{eqn:morphisms-coeff-systems}
\Adx ST \usotimes{\A(T)} M(T)
\simeq
M(S) \usotimes{\B(S)}
\Bdx ST.
\end{equation}
Under mild conditions, a morphism of coefficient systems $\A \to \B$ gives rise to an adjunction
\begin{equation} \label{adj:ind-oblv-intro}
\begin{tikzpicture}[scale=1.5]
\node (a) at (0,1) {$\ind_{\Y}^{\A \to \B}: \ShvCat^\A(\Y) $};
\node (b) at (3,1) {$\ShvCat^\B(\Y)  : \oblv_\Y^{\A \to \B}$,};
\path[->,font=\scriptsize,>=angle 90]
([yshift= 1.5pt]a.east) edge node[above] { } ([yshift= 1.5pt]b.west);
\path[->,font=\scriptsize,>=angle 90]
([yshift= -1.5pt]b.west) edge node[below] {} ([yshift= -1.5pt]a.east);
\end{tikzpicture}
\end{equation}
which may be regarded as a categorified version of the usual \virg{extension/restriction of scalars} adjunction.

\begin{example}

For instance, $\QCoh$ yields a morphism $\H \to \DD$: i.e., $\QCoh(S)$ is naturally an $(\H(S), \Dmod(S))$-bimodule and there are natural equivalences
$$
\Hdx ST \usotimes{\H(T)} \QCoh(T)
\simeq
\QCoh(S)
\usotimes{\Dmod(S)}
\DD_{S \to T}
$$
for any $S \to T$. In fact, both sides are obviously equivalent to $\QCoh(S)$.

\end{example}

\begin{example}

Similarly, $\ICoh$ gives rise to a morphism $\DD \to \H$: indeed, both sides of
$$
\DD_{S \to T}
\usotimes{\DD(T)}
\ICoh(T)
\simeq
\ICoh(S)
\usotimes{\H(S)}
\Hdx ST
$$
are equivalent to $\ICoh(T^\wedge_S)$, as shown in the main body of the paper.

\end{example}

\begin{rem}
Continuing the analogies of Remark \ref{rem:analogies}, one may think of $\QCoh(\Y)$ as a categorification of the algebra $\O_\Y$ of functions on $\Y$ (a left $\fD$-module). Likewise, $\ICoh(\Y)$ categorifies the space of distributions on $\Y$ (a right $\fD$-module). Then the $\H$-affineness theorem states that $\H$ categorifies the algebra of differential operators on $\Y$.
These observations help remember/explain the directions of the morphisms $\H \to \DD$ and $\DD \to \H$ in the two examples above: $\QCoh$ is naturally a left $\H$-module, while $\ICoh$ is naturally a right $\H$-module.
\end{rem}

\begin{rem}
Our Theorem \ref{thm:H-functoriality-intro} shows that the morphism $\QCoh: \H \to \DD$ is \virg{optimal} in that the natural monoidal functor
$$
\Dmod(Y) 
\longto
 \Fun_{\H(Y)} (\QCoh(Y), \QCoh(Y))
$$
is an equivalence for any $Y \in \Schevcoclfp$. 
On the other hand, the morphism $\ICoh: \DD \to \H$ is not optimal: in another work (see \cite{gluing} for more in this direction), we plan to show that
\begin{equation} \label{eqn:what acts on ICOH}
  \Fun_{\H(Y)} (\ICoh(Y), \ICoh(Y))
 \simeq
\vDmod(LY),
\end{equation}
where $\vDmod(LY)$ is the monoidal DG category introduced in \cite{centerH}. For $Y$ quasi-smooth, $\vDmod(LY)$ is closely related to $\Dmod(\Sing(Y))$. Let us point out that the above equivalence (\ref{eqn:what acts on ICOH}) would provide an answer to the question \virg{\emph{What acts on $\ICoh$}?} raised in \cite[Remark 1.4.3]{AG2}.
\end{rem}

\begin{example}

Another morphism of coefficient systems of interest in this paper is $\Q \to \H$, the one induced by the monoidal functor $\QCoh(S) \to \H(S)$. In this case, the adjunction (\ref{adj:ind-oblv-intro}) categorifies the induction/forgetful adjunction between quasi-coherent sheaves and left $\fD$-modules.

\end{example}

\sssec{} \label{sssec:formal version}

Here is how the $\H$-affineness Theorem \ref{main-thm-intro} implies Theorem \ref{vague-thm:sing-support}. 
The datum of a left $\H(\Y)$-action $\C$ corresponds the datum of an object $\wt\C \in \ShvCatH(\Y)$. 
Now, on the one hand $\ShvCatH$ satisfies \emph{smooth descent}, see Theorem \ref{thm:smooth-descent-for-ShvCatH}. On the other hand, singular support is computed smooth locally. Hence, we are back to Theorem \ref{vague-thm:sing-support} for affine schemes, which has already been discussed.

%%%%By construction, DG categories equipped with a singular support theory relative to $\Sing(\Y)$ are exactly the objects of the $\infty$-category
%%%%$$
%%%%\QCoh(\Y) \mmod
%%%%\ustimes{\ShvCatQ(\Y)}
%%%%\ShvCat^\SS(\Y).
%%%%$$
%%%%Now, any $\C \in \H(\Y) \mmod$ has this structure thanks to $\H$-affineness and the functor $\oblv_\Y^{\SS \to \H}$.

\ssec{Functoriality of $\H$ for algebraic stacks}

The $\H$-affineness theorem has another consequence: it allows to extend the assignment 
$\Y \squigto \H(\Y)$ to a functor out of a certain $\infty$-category of correspondences of stacks.

\sssec{}

Indeed, as we prove in this paper, $\ShvCatH$ enjoys a rich functoriality: besides the structure pullbacks $f^{*,\H}: \ShvCatH(\Z) \to \ShvCatH(\Y)$ associated to $f: \Y \to \Z$, there are also pushforward functors  $f_{*,\H}$ (right adjoint to pullbacks) satisfying base-change along cartesian squares.

\medskip
 
Now, Theorem \ref{main-thm-intro} guarantees that the assignment $\Y \squigto \H(\Y)$ enjoys a parallel functoriality, as stated in the following theorem.

\begin{thm} \label{thm:H-functoriality-intro}
There is a natural functor
$$
\Corr(\Stkevcoclfp)_{\evcoc; \all}  
\longto 
\AlgBimod(\DGCat)
$$
that sends
$$
\X \squigto \H(\X),
\hspace{.4cm} 
[\X \leftto \W \to \Y]
\squigto
\Hcorr \X \W \Y := \ICoh_0((\X \times \Y)^\wedge_\W).
$$
Here $\Corr(\Stkevcoclfp)_{\evcoc; \all})$ is the $\infty$-category whose objects are objects of $\Stkevcoclfp$ and whose $1$-morphisms are given by correspondences $[\X \leftto \W \to \Y]$ with bounded left leg.
\end{thm}

\sssec{}

In the rest of this introduction, we exploit such functoriality in the case of classifying spaces of algebraic groups (Section \ref{ssec:Harish-Chandra}) and in the case of local systems over a smooth complete curve (Section \ref{ssec:LS}).

\ssec{$\H$ for Harish-Chandra} \label{ssec:Harish-Chandra}

For $\Y$ smooth, $\H(\Y)$ is equivalent to $\ICoh(\Yform)$, with its natural convolution monoidal structure. For instance, if $G$ is an affine algebraic group, we have
$$
\H(BG)
\simeq
\ICoh(G \backslash G_\dR /G).
$$
This is the monoidal category of Harish-Chandra bimodules for the group $G$, see \cite[Section 2.3]{Be} for the connection with the theory of weak/strong actions on categories.
Likewise, 
$$
\Hdx {\pt}{BG} \simeq \IndCoh(G_\dR/G)
$$
is the DG category $\g \mod$ of modules for the Lie algebra $\g = \Lie(G)$.
More generally, for a group morphism $H \to G$, we have
$$
\Hsx {BG}{BH}
=
\ICoh((BG \times BH)^\wedge_{BH})
\simeq
\ICoh(G \backslash G_\dR/ H )
\simeq
\g \mod^{H,w}.
$$
This is the correct derived enhancement of the ordinary category of Harish-Chandra $(\g,H)$-modules.

\sssec{}

Theorem \ref{thm:H-functoriality-intro} yields the following equivalences:
$$
\Hdx {BH}{BG} \usotimes{\H(BG)} \Hdx {BG}\pt
\xto{\;\; \simeq \;\;}
\Hdx {BH}{\pt}
\simeq
\QCoh(BH); 
$$
$$
\Hdx {\pt}{BG} \usotimes{\H(BG)} \Hsx {BG}{\pt}
\xto{\;\; \simeq \;\;}
\Dmod(G);
$$
$$
\Hsx {\pt}{BG} \usotimes{\H(BG)} \Hdx {BG}\pt
\xto{\;\; \simeq \;\;}
\Dmod(BG).
$$

\sssec{}

Another way to prove these is via the theory of DG categories with $G$-action, see \cite[Section 2]{Be}. 
For instance, it was proven there that, for any category $\C$ equipped with a right strong action of $G$, there are natural equivalences:
$$
\C^{G,w} \usotimes{\H(BG)} \Rep(G) \simeq \C^{G},
\hspace{.4cm}
\C^{G,w} 
\usotimes{\H(BG)}
\g \mod
\simeq
\C.
$$
Now, let $\C = {}^{H,w}\Dmod(G)$, $\C = \Dmod(G)$ and $\C= \Vect$ respectively.

\sssec{}

For generalizations of these computations to the topological setting, the reader may consult \cite{top-chiral}.

\ssec{The gluing theorems in geometric Langlands} \label{ssec:LS}

More interesting than $\H(BG)$ is the monoidal DG category $\H(\LS_G)$, to which we now turn attention. Observe that, by construction, we have 
$$
\Hcorr \Y \X \pt
\simeq
\ICoh_0(\Y^\wedge_\X).
$$
With this notation, the \emph{spectral gluing theorem} of \cite{AG2} may be rephrased as follows: there is an explicit $\H(\LS_\Gch)$-linear localization adjunction
\begin{equation} \label{adj:spectral-gluing-original}
\begin{tikzpicture}[scale=1.5]
\node (a) at (0,1) {$(\gamma^\spec)^L: \Glue_{P} \;
\Hcorr {\LS_\Gch}{\LS_\Pch}{\pt}$};
\node (b) at (4,1) {$\ICoh_\N(\LS_\Gch): \gamma^\spec$.};
\path[->>,font=\scriptsize,>=angle 90]
([yshift= 1.5pt]a.east) edge node[above] {$ $} ([yshift= 1.5pt]b.west);
\path[right hook ->,font=\scriptsize,>=angle 90]
([yshift= -1.5pt]b.west) edge node[below] {$ $} ([yshift= -1.5pt]a.east);
\end{tikzpicture}
\end{equation}
Here we have switched to the Langlands dual $\Gch$ as we are going to discuss Langlands duality, and it is customary to have Langlands dual groups on the spectral side.

\sssec{}

Let $\Mch$ be the Levi quotient of a parabolic $\Pch$. By Theorem \ref{thm:H-functoriality-intro}, we can rewrite
$$
\Hcorr {\LS_\Gch}{\LS_\Pch}{\pt}
\simeq
\Hcorr {\LS_\Gch}{\LS_\Pch}{\LS_{\Mch}}
\usotimes{\H(\LS_\Mch)}
\Hdx{\LS_\Mch}{\pt}
\simeq
\Hcorr {\LS_\Gch}{\LS_\Pch}{\LS_{\Mch}}
\usotimes{\H(\LS_\Mch)}
\QCoh(\LS_\Mch).
$$
By the $\H$-affineness theorem, we reinterpret the bimodule $\Hcorr {\LS_\Gch}{\LS_\Pch}{\LS_{\Mch}}$, or better the functor
$$
\EEis_{\Pch}:
\H(\LS_\Mch) \mmod
\longto
\H(\LS_\Gch) \mmod
$$
attached to it, as an \emph{Eisenstein series} functor in the setting of $\H$-sheaves of categories.

\sssec{}

These considerations shed light on the LHS of (\ref{adj:spectral-gluing-original}). Coupled with the construction of Section \ref{sssec:discussion of vanishing}, they allow to formulate a conjecture on the automorphic side of geometric Langlands. This conjecture explains how $\Dmod(\Bun_G)$ can be reconstructed algorithmically out of tempered $\fD$-modules for all the Levi's of $G$, including $G$ itself.

\begin{conj}[(Automorphic gluing)] \label{conj:aut-gluing}
There is an explicit $\H(\LS_\Gch)$-linear localization adjunction
\begin{equation} \label{adj:automorphic-gluing}
\begin{tikzpicture}[scale=1.5]
\node (a) at (0,1) {$\gamma^L: \Glue_{P} \;
\EEis_{\Pch} (\Dmod(\Bun_M)^\temp)$};
\node (b) at (3.5,1) {$\Dmod(\Bun_G): \gamma$.};
\path[->>,font=\scriptsize,>=angle 90]
([yshift= 1.5pt]a.east) edge node[above] {$ $} ([yshift= 1.5pt]b.west);
\path[right hook ->,font=\scriptsize,>=angle 90]
([yshift= -1.5pt]b.west) edge node[below] {$ $} ([yshift= -1.5pt]a.east);
\end{tikzpicture}
\end{equation}
\end{conj}

\sssec{}

Some comments on this conjecture and on some future research directions:
\begin{enumerate}
\item
We will construct the adjunction (\ref{adj:automorphic-gluing}) in a follow-up paper; this will be relatively easy. The difficult part is to show that the right adjoint is fully faithful.
\item
Actually, the conjecture can be pushed even further, as it is possible to guess what the essential image $\gamma$ is: this follows from an explicit description of the essential image of $\gamma^\spec$, see \cite{gluing}. 
\item
Clearly, Conjecture \ref{conj:aut-gluing} is related to the extended Whittaker conjecture, see \cite{Outline} and \cite{ext-whit}. The LHS of (\ref{adj:automorphic-gluing}) is expected to be smaller than the extended Whittaker category.
\end{enumerate}

\ssec{Conventions}

We refer \cite{Book}, \cite{shvcat} or \cite{centerH} for a review of our conventions concerning category theory and algebraic geometry.
In particular:
\begin{itemize}
\item
we always work over an algebraically closed field $\kk$ of characteristic zero;
\item
we denote by $\DGCat$ the (large) symmetric monoidal $\infty$-category of small \emph{cocomplete} DG categories over $\kk$ and continuous functors, see \cite{Lurie:HA} or \cite{Book}.
\end{itemize}

\ssec{Structure of the paper}

Section \ref{sec:cat-algebra} is devoted to recalling some higher algebra: a few facts about rigid monoidal DG categories and their module categories, as well as several $\2$-categorical constructions (correspondences, lax $\2$-functors, algebras and bimodules).

The first part of Section \ref{sec:ICoh zero} is a reminder of the theory of $\ICoh_0$, as developed in \cite{centerH}. In the second part of the same section, we discuss the $\2$-categorical functoriality of $\H$.

Section \ref{sec:Coeff systems} introduces the notion of coefficient system, providing several examples of interest in present, as well as future, applications. In particular, we define the (a priori lax) coefficient system $\H$ and prove it is strict.

In Section \ref{sec:base-change}, we discuss the (left, right, ambidextrous) Beck-Chevalley conditions for coefficient systems. These conditions (which are satisfied in the examples of interest) guarantee that the resulting theory of sheaves of categories is very rich functorially: e.g., it has pushforwards and base-change.

Finally, in Section \ref{sec:shvcatH}, we define $\ShvCatH$, the theory of \emph{sheaves of categories with local actions of Hochschild cochains}, and prove the $\H$-affineness of algebraic stacks.

\ssec{Acknowledgements}
The main idea behind the present paper was conceived as a result of five illuminating conversations with Dennis Gaitsgory (Paris, August 2015). It is a pleasure to thank him for his generosity in explaining and donating his ideas.
Thanks are also due to Dima Arinkin, David Ben-Zvi, Ian Grojnowski, David Jordan, Tony Pantev, Sam Raskin and Pavel Safronov for their interest and influence.
Finally, I am grateful to the referee for the very helpful technical report.

Research partially supported by EPSRC programme grant EP/M024830/1 Symmetries and Correspondences, and by grant ERC-2016-ADG-74150.

\section{Some categorical algebra} \label{sec:cat-algebra}

In this section we recall some $\1$- and $\2$-categorical algebra needed later in the main sections of the paper. All results we need concern the theory of algebras and bimodules.
More specifically, we first need criteria for dualizability of bimodule categories; secondly, we need some abstract constructions that relate \virg{algebras and bimodules} with $\2$-categories of correspondences.

We advise the reader to skip this material and get to it only if necessary.

\ssec{Dualizability of bimodule categories}

Recall that $\DGCat$ admits colimits (as well as limits) and its tensor product preserves colimits in each variable, \cite{Lurie:HA}. Hence, by \loccit, we have a good theory of dualizability of algebras and bimodules in $\DGCat$, whose main points we record below.
We will need a criterion that relates the dualizability of a bimodule to the dualizability of its underlying DG category.

\sssec{}

First, let us fix some terminology. 
Algebra objects in a symmetric monoidal $\infty$-category are always unital in this paper. In particular, monoidal DG categories are unital.
Given $A$ an algebra, denote by $A^\rev$ the algebra obtained by reversing the order of the multiplication.
For a left $A$-module $M$ and a right $A$-module $N$, we denote by $\pr: N \otimes M \to N \otimes_A M$ the tautological functor. 

\medskip

Our conventions regarding bimodules are as follows: an $(A,B)$-bimodule $M$ is acted on the left by $A$ and on the right by $B$. Hence, endowing $C \in \DGCat$ with the structure of an $(A,B)$-bimodule amounts to endowing it with the structure of a left $A \otimes B^\rev$-module.

\sssec{}  \label{sssec:right-left-dualizability of bimodules}

Let $M$ be an $(A, B)$-bimodule. We say that $M$ is \emph{left dualizable} (as an $(A, B)$-bimodule) if there exists a $(B,A)$-bimodule $M^L$ (called the \emph{left dual} of $M$) realizing an adjunction
$$ %%%best-adjunction-diagram-prototype
\begin{tikzpicture}[scale=1.5]
\node (a) at (0,1) {$M^L \otimes_A - : A \mod$};
\node (b) at (3,1) {$B \mod : M \otimes_B - $.};
\path[->,font=\scriptsize,>=angle 90]
([yshift= 1.5pt]a.east) edge node[above] {$ $} ([yshift= 1.5pt]b.west);
\path[->,font=\scriptsize,>=angle 90]
([yshift= -1.5pt]b.west) edge node[below] {$ $} ([yshift= -1.5pt]a.east);
\end{tikzpicture}
$$
Similarly, $M$ is \emph{right dualizable} if there exists $M^R \in (B,A)\bimod$ (the \emph{right dual} of $M$) realizing an adjunction
$$ %%%best-adjunction-diagram-prototype
\begin{tikzpicture}[scale=1.5]
\node (a) at (0,1) {$M \otimes_B - : B \mod$};
\node (b) at (3,1) {$A \mod : M^R \otimes_A - $.};
\path[->,font=\scriptsize,>=angle 90]
([yshift= 1.5pt]a.east) edge node[above] {$ $} ([yshift= 1.5pt]b.west);
\path[->,font=\scriptsize,>=angle 90]
([yshift= -1.5pt]b.west) edge node[below] {$ $} ([yshift= -1.5pt]a.east);
\end{tikzpicture}
$$
We say that an $(A,B)$-bimodule $M$ is \emph{ambidextrous} if both $M^L$ and $M^R$ exist and are equivalent as $(B,A)$-bimodules.

\begin{rem}

Being (left or right) dualizable as a $(\Vect, \Vect)$-bimodule is equivalent to being dualizable as a DG category. 
By definition, being \virg{left (or right) dualizable as a right $A$-module} means being \virg{left (or right) dualizable as a $(\Vect,A)$-module}. Similarly for left $A$-modules.

\end{rem}

\sssec{}

Let $M$ be an $(A,B)$-bimodule which is dualizable as a DG category. Then we can contemplate three $(B,A)$-bimodules: $M^L, M^R$ (if they exist) as well as $M^*$, the dual of $\oblv_{A,B}(M)$ equipped with the dual actions.

In particular, a monoidal DG category $A$ is called \emph{proper} if it
is dualizable as a plain DG category. In this case, we denote by $S_A := A^*$ its dual, equipped with the tautological $(A,A)$-bimodule structure.

\sssec{} \label{sssec:very rigid}

Recall the notion of \emph{rigid} monoidal DG category, see \cite[Appendix D]{shvcat}.
Any rigid $A$ is automatically proper.
Furthermore, its dual $S_A:=A^*$ comes equipped with the canonical object $1_A^{fake} := (u^R)^\vee(\kk)$, where $u^R$ is the (continuous) right adjoint to the unit functor $u:\Vect \to A$.
The left $A$-linear functor
$$
\sigma_A:
A 
\longto
S_A,
\hspace{.4cm}
a
\squigto
a \star 1_A^{fake}
$$
is an equivalence: in particular, any rigid monoidal category is self-dual.
We say that $A$ is \emph{very rigid} if the canonical equivalence $\sigma_A: A \to S_A$ admits a lift to an equivalence of $(A,A)$-bimodules.\footnote{Compare this notion with the more general notion of \virg{symmetric Frobenius algebra object}, discussed in \cite[Remark 4.6.5.7]{Lurie:HA}.}

\begin{prop} 
Let $A,B$ be rigid monoidal DG categories and $M$ an $(A,B)$-bimodule which is dualizable as a DG category. Then $M$ is right dualizable as an $(A,B)$-bimodule and $
M^R \simeq M^* \otimes_{A} S_A$. Likewise, $M$ is left dualizable and $M^L \simeq S_B \otimes_{B} M^*$. 
\end{prop} 

\begin{proof}
The formula for $M^R$ is proven as in the \virg{if} direction of \cite[Proposition D.5.4]{shvcat}, which in turn is a consequence of \cite[Corollary D.4.5]{shvcat}. In the notation there, the twist $(-)_{\psi_A}$ corresponds to our $- \otimes_A S_A$.
The formula for $M^L$ follows similarly.
\end{proof}

\begin{cor} \label{cor:very rigid has ambidextrous modules}
Let $A,B$ be very rigid and $M$ an $(A,B)$-bimodule which is dualizable as a DG category. Then we have canonical $(B,A)$-linear equivalences $M^R \simeq M^* \simeq M^L$.
\end{cor}

\ssec{Some $\2$-categorical algebra}

In this section, we recall some abstract $\2$-categorical nonsense and provide some examples of $\2$-categories and of lax $\2$-functors between them. 
All the statements below look obvious enough and no proof will be given.

\sssec{}

We assume familiarity with the notion of \emph{$\2$-category} and with the notion of (lax) $\2$-functor between $\2$-categories; a reference is, for instance, \cite[Appendix A]{Book}. 
For an $\2$-category $\bC$, we denote by $\bC^{1-\op}$ the $\2$-category obtained from $\bC$ by flipping the $1$-arrows. Similarly, we denote $\bC^{2-\op}$ the $\2$-category obtained by flipping the directions of the $2$-arrows.

\sssec{Correspondences} \label{sssec:Corresp}

Let $\C$ be an $\infty$-category equipped with fiber products.
We refer to \cite[Chapter V.1]{Book} for the construction of the $\infty$-category of correspondences associated to $\C$. 
In particular, for $\vert$ and $\horiz$ two subsets of the space morphisms of $\C$ satisfying some natural requirements, one considers the $\infty$-category $\Corr(\C)_{\vert; \horiz}$, defined in the usual way: objects of $\Corr(\C)_{\vert; \horiz}$ coincide with the objects of $\C$, while $1$-morphisms in $\Corr(\C)_{\vert; \horiz}$ are given by correspondences
$$
[c \leftto h \to d]
$$
with left leg in $\vert$ and right leg in $\horiz$.

\medskip

To enhance $\Corr(\C)_{\vert; \horiz}$ to an $\2$-category, we must further choose a subset $\adm \subset \vert \cap \horiz$ of \emph{admissible arrows}, closed under composition. Then, following \cite[Chapter V.1]{Book}, one defines the $\2$-category  
$$
\Corr(\C)_{\vert; \horiz}^{\adm}.
$$
This is one of the most important $\2$-categories of the present paper.

\medskip

To fix the notation, recall that a $2$-arrow 
$$
[c \leftto h \to d] \implies [c \leftto h' \to d]
$$ 
in $\Corr(\C)_{\vert; \horiz}^{\adm}$ is by definition an admissible arrow $h \to h'$ compatible with the maps to $c \times d$.

\medskip

As explained in \cite[Chapter V.3]{Book}, $\Corr(\C)_{\vert; \horiz}^{\adm}$ is symmetric monoidal with tensor product induced by the Cartesian symmetric monoidal product on $\C$.

\sssec{Algebras and bimodules}

The other important $\2$-category of this paper is $\ALGBimod(\DGCat)$, the $\2$-category of monoidal DG categories, bimodules, and natural transformations. We refer to \cite{Rune} for a rigorous construction. More generally, \loccit \,gives a  construction of $\ALGBimod(\S)$ for any (nice enough) symmetric monoidal $\2$-category $\S$. 

We denote by $\AlgBimod(\S)$ the $\1$-category underlying $\ALGBimod(\S)$: that is, the former is obtained from the latter by discarding non-invertible $2$-morphisms.

\sssec{}

There is an obvious functor
\begin{equation} \label{eqn:from-Alg-to-AlgBimod}
\iota_{\Alg \to \Bimod}: 
\Alg(\DGCat)^\op \longto \AlgBimod(\DGCat)
\end{equation}
that is the identity on objects and that sends a monoidal functor $A \to B$ to the $(B, A)$-bimodule $B$.

\medskip

The tautological functor
$$
\AlgBimod(\DGCat)^\op \xto{\mmod}
\inftyCat
$$
upgrades to a (strict) $\2$-functor
$$
\ALGBimod(\DGCat)^{1-\op}
\xto{\mmod}
\inftyCat,
$$
where now $\inftyCat$ is considered as an $\2$-category.

\sssec{} \label{sssec:example of lax 2 functor}

Let $\C$ denote an $\1$-category admitting fiber products and equipped with the cartesian symmetric monoidal structure. Let $F: \C^\op \longto \DGCat$ be a lax-monoidal functor. 
(The example we have in mind is $\C = \PreStk$ and $F = \QCoh$.)

These data give rise to a lax $\2$-functor
$$
\wt F:
\bigt{ \Corr(\C)_{\all; \all}^{\all} }^{2-\op}
\longto
\ALGBimod(\DGCat),
$$
described informally as follows:
\begin{itemize}
\item
an object $c \in \C$ gets sent to $F(c)$, with its natural monoidal structure; 
\item
a correspondence $[c \leftto h \to d]$ gets sent to the $(F(c), F(d))$-bimodule $F(h)$;
\item
a map between correspondences, given by an arrow $h' \to h$ over $c \times d$, gets sent to the associated $(F(c), F(d))$-linear arrow $F(h) \to F(h')$;
\item
for two correspondences $[c \leftto h \to d]$ and $[d \leftto k \to e]$, the lax composition is encoded by the natural $(F(c), F(e))$-linear arrow 
$$
F(h) \usotimes{F(d)} F(k)
\longto
F(h \times_d k).
$$
\end{itemize}

\sssec{} \label{sssec:from lax monoidal to alg and bimodules}

Here is another example of the interaction between lax-monoidal functors and lax $\2$-functors.
Let $F: \C \to \D$ be a lax-monoidal functor between \virg{well-behaved} monoidal $\1$-categories. Then $F$ induces a lax $\2$-functor
$$
\wt F:
\ALGBimod(\C)
\longto
\ALGBimod(\D).
$$
To define it, it suffices to recall that, since $F$ is lax monoidal, it preserves algebra and bimodule objects. The fact that $\wt F$ is a lax $\2$-functor comes from the natural map (not necessarily an isomorphism)
$$
F(c') \otimes_{F(c)} F(c'') \longto F(c' \otimes_{c} c'').
$$

\sssec{} \label{sssec:loop-mod}

Recall the $\infty$-category $\Mod(\DGCat)$ whose objects are pairs $(A, M)$ with $A$ a monoidal DG category and $M$ an $A$-module. Morphisms $(A,M) \to (B,N)$ consist of pairs $(\phi, f)$ where $\phi: A \to B$ is a monoidal functor and $f: M\to N$ an $A$-linear functor.

\medskip

There is a lax $\2$-functor
\begin{equation} \label{eqn:loop-mod}
\Loop_{\Mod}:
\Mod(\DGCat)^\op
\longto
\ALGBimod(\DGCat),
\end{equation}
described informally as follows:
\begin{itemize}
\item
an object $(A,M) \in \Mod(\DGCat)$ goes to the monoidal DG category $ \End_A(M) := \Fun_A(M,M)$;
\item
a morphism $(A,M) \xto{(\phi,f)} (B,N)$ gets sent to the $(\End_B(N),\End_A(M))$-bimodule $\Fun_A(M, N)$;
\item 
a composition $(A,M) \xto{(\phi,f)} (B,N) \xto{(\psi,g)} (C,P)$ goes over to the $(\End_C(P),\End_A(M))$-bimodule
$$
\Fun_B(N, P) \usotimes{\End_B(N)} \Fun_A(M, N);
$$
\item
the lax structure comes from the tautological morphism (not invertible, in general)
\begin{equation} \label{eqn:composition-of-Fun(C,C')}
\Fun_B(N, P) \usotimes{\End_B(N)} \Fun_A(M, N)
\longto 
\Fun_A(M,P)
\end{equation}
induced by composition.
\end{itemize}

\sssec{} \label{sssec: paradigm: subfunctor of a lax-functor}

\renc{\A}{{\mathbb{A}}}

For later use, we record here the following tautological observation. Let $\I$ be an $\1$-category and $\A: \I \to \ALGBimod(\DGCat)$ be a lax $\2$-functor. Assume given the following data:
\begin{itemize}
\item
for each $i \in \I$, a monoidal subcategory $\A'(i) \hto \A(i)$;

\smallskip

\item
for each $i \to j$, a full subcategory $\A'_{i \to j} \hto \Adx ij$ preserved by the $(\A'(i), \A'(j))$-action.
 
\smallskip

\end{itemize}
Assume furthermore that, for each string $i \to j \to k$, the functor 
$$
\A'_{i \to j} \otimes \A'_{j \to k} 
\hto
 \Adx ij \otimes \Adx jk 
 \xto{\; \pr \; }
 \Adx ij \otimes_{\A(j)} \Adx jk 
\xto{\; \eta_{i \to j \to k} \,} 
\Adx ik
$$
lands in $\A'_{i \to k} \subseteq \A_{i \to k}$.
Then the assigment 
$$
i \squigto \A'(i),
\hspace{.4cm}
(i \to j) \squigto \A'_{i \to j} 
$$
naturally upgrades to a \emph{lax} $\2$-functor $\A': \I \to \ALGBimod(\DGCat)$.

\section{$\ICoh_0$ on formal moduli problems} \label{sec:ICoh zero}

In the section, we study the sheaf theory $\ICoh_0$ from which $\H$ originates. As mentioned in the introduction of \cite{centerH}, $\ICoh_0$ enjoys $\1$-categorical functoriality as well as $\2$-categorical functoriality.
The former was developed in \loccit, and recalled here in Theorem \ref{thm:1-functoriality of ICohzero}. The latter is one of the main subjects of the present paper: it consists of an extension of the assignment $\Y \squigto \H(\Y)$ to a lax $\2$-functor from a certain $\2$-category of correspondences to $\ALGBimod(\DGCat)$.

\ssec{The $\1$-categorical functoriality}

In this section, we review the definition of the assignment $\ICoh_0$ and its basic functoriality. We follow \cite{centerH} closely.

\sssec{}

Let $\Stk$ denote the $\infty$-category of perfect quasi-compact algebraic stacks of finite type and with affine diagonal, see e.g. \cite{BFN}.
Inside $\Stk$, we single out the subcategory $\Stkevcoclfp$ consisting of those stacks that are bounded and with perfect cotangent complex (both properties can be checked on an atlas).

\sssec{}

For $\C$ an $\infty$-category, denote by $\Arr(\C) := \C^{\Delta^1}$ the $\infty$-category whose objects are arrows in $\C$ and whose $1$-morphisms are commutative squares.
We will be interested in the $\infty$-category $\Arr(\Stkevcoclfp)$ and in the functor
\begin{equation} \label{funct: ICOHzero as a functor out of formalmod-opposite}
\ICoh_0:
\Arr(\Stkevcoclfp) 
^\op
\longto
\DGCat
\end{equation}
defined by 
$$
[\Y \to \Z]
\squigto
\ICoh_0(\Z^\wedge_\Y).
$$
Recall from \cite{AG2} or \cite{centerH} that $\ICoh_0(\Z^\wedge_\Y)$ is defined by the pull-back square
\begin{equation} \label{diag:definition-ICoh0}
\begin{tikzpicture}[scale=1.5]
\node (M) at (0,1) {$\ICoh_0(\Z^\wedge_\Y)$ };
\node (IM) at (0,0) {$\ICoh(\Z^\wedge_\Y)$};
\node (N) at (2.5,1) {$\QCoh(\Y)$};
\node (IN) at (2.5,0) {$\ICoh(\Y).$};
\path[right hook->,font=\scriptsize,>=angle 90]
(M.south) edge node[right] { $\iota$ } (IM.north);
\path[right hook->,font=\scriptsize,>=angle 90]
(N.south) edge node[right] { $\Upsilon_\Y$ } (IN.north);
%%%vertical maps
\path[->,font=\scriptsize,>=angle 90]
(M.east) edge node[above] { $('f)^{!,0}$ } (N.west);
\path[->,font=\scriptsize,>=angle 90]
(IM.east) edge node[above] { $('f)^!$ } (IN.west);
\end{tikzpicture}
\end{equation}
In particular, when writing $\ICoh_0(\Z^\wedge_\Y)$ we are committing a potentially dangerous abuse of notation: it would be better to write $\ICoh_0 (\Y \to \Z^\wedge_\Y)$, as the latter category depends on the formal moduli problem $\Y \to \Z^\wedge_\Y$ and in particular on the derived structure of $\Y$.

\sssec{}

For two objects $[\Y_1 \to \Z_1]$ and $[\Y_2 \to \Z_2]$ in $\Arr(\Stkevcoclfp)$, a morphism $\xi$ from the former to the latter is given by a commutative square
\begin{equation} \label{diam:morphism in Arrow category}
\begin{tikzpicture}[scale=1.5]
\node (00) at (0,0) {$\Z_1$};
\node (10) at (1.5,0) {$\Z_2$.};
\node (01) at (0,.8) {$\Y_1$};
\node (11) at (1.5,.8) {$\Y_2$};
\path[->,font=\scriptsize,>=angle 90]
(00.east) edge node[above] {$\xi_{bottom}$} (10.west);
\path[->,font=\scriptsize,>=angle 90]
(01.east) edge node[above] {$\xi_{top}$} (11.west);
%%%VERT
\path[->,font=\scriptsize,>=angle 90]
(01.south) edge node[right] {$'f_1 $} (00.north);
\path[->,font=\scriptsize,>=angle 90]
(11.south) edge node[right] {$'f_2$} (10.north);
\end{tikzpicture}
\end{equation}
%By convention, we agree to always draw objects of $\Arr$ as vertical maps.
%
The structure pullback functor
$$
\xi^{!,0}:
\ICoh_0((\Z_2)^\wedge_{\Y_2})
\longto
\ICoh_0((\Z_1)^\wedge_{\Y_1})
$$
is the obvious one induced by the pullback functor $\xi^!: \ICoh((\Z_2)^\wedge_{\Y_2}) \to \ICoh((\Z_1)^\wedge_{\Y_1})$, where we are abusing notation again by confusing $\xi$ with the map $(\Z_1)^\wedge_{\Y_1} \to (\Z_2)^\wedge_{\Y_2}$. We will do this throughout the paper, hoping it will not be too unpleasant for the reader.

\sssec{}

Let us now recall the extension of (\ref{funct: ICOHzero as a functor out of formalmod-opposite}) to a functor out of a category of correspondences.
Notice that $\Arr(\PreStk)$ admits fiber products, computed objectwise; its subcategory $\Arr(\Stkevcoclfp)$ is closed under products, but not under fiber products. Thus, to have a well-defined category of correspondences, we must choose appropriate classes of horizontal and vertical arrows.

\medskip

We say that a commutative diagram (\ref{diam:morphism in Arrow category}), thought of as a morphism in $\Arr(\Stkevcoclfp)$, is schematic (or bounded, or proper) if so is the top horizontal map. It is clear that 
\begin{equation} \label{cat:Corr of Arrows 1-cat}
\Corr \Bigt{ \Arr(\Stkevcoclfp)}_{\schem \& \evcoc;\all}
\end{equation}
is well-defined. 

For the theorem below, we will need to further upgrade (\ref{cat:Corr of Arrows 1-cat}) to an $\2$-category by allowing as admissible arrows (see Section \ref{sssec:Corresp} for the terminology) those $\xi$'s that are schematic, bounded and proper. We denote by
 $$
\Corr \bigt{ \Arr(\Schevcoclfp) }_{\schem \& \evcoc;\all}
^{\schem \& \evcoc \& \proper}
$$
the resulting $\2$-category.

\sssec{}

If $\xi$ is bounded and schematic in the above sense, then the pushforward $\xi_*^\ICoh: \ICoh((\Z_1)^\wedge_{\Y_1}) \to \ICoh((\Z_2)^\wedge_{\Y_2})$ is continuous and it preserves the $\ICoh_0$-subcategories, thereby descending to a functor $\xi_{*,0}$. For the proof, see \cite{centerH}.

\begin{thm} \label{thm:1-functoriality of ICohzero}
The above pushforward functors upgrade the functor $\ICoh_0^!$ of (\ref{funct: ICOHzero as a functor out of formalmod-opposite}) to an $\2$-functor 
$$
\ICoh_0:
\Corr \bigt{ \Arr(\Schevcoclfp) }_{\schem \& \evcoc;\all}
^{\schem \& \evcoc \& \proper}
\longto
\DGCat,
$$ 
where $\DGCat$ is viewed as an $\2$-category in the obvious way.
\end{thm}

\begin{rem}
The existence of the above $\2$-functor is deduced (essentially formally) by the $\2$-functor
\begin{equation} \label{funct:ICOH out of corresp}
\ICoh:
\Corr(\PreStk_\laft)_{\text{ind-inf-schem; all}}^{\text{ind-inf-sch \& ind-proper}} 
\longto
\DGCat
\end{equation}
constructed in \cite[Chapter III.3]{Book}. For later use, we will also need another fact from \loccit: the above $\2$-category of correspondences possesses a symmetric monoidal structure, and (\ref{funct:ICOH out of corresp}) is naturally symmetric monoidal. See \cite[Chapter V.3]{Book}.
It follows that the $\2$-functor on Theorem \ref{thm:1-functoriality of ICohzero} is symmetric monoidal, too.
\end{rem}

\sssec{Example}
For $f: \Y \to \Z$, the admissible arrow $\Y \to \Z^\wedge_\Y$ yields an adjuction
\begin{equation} 	\label{adj:monadic adjunction for ICOHzero}
\begin{tikzpicture}[scale=1.5]
\node (a) at (0,1) {$\QCoh(\Y)$};
\node (b) at (4,1) {$\ICoh_0(\Z^\wedge_\Y)$.};
\path[->,font=\scriptsize,>=angle 90]
([yshift= 1.5pt]a.east) edge node[above] {$('f)_{*,0} \simeq ('f)_*^\ICoh \circ \Upsilon_\Y $} ([yshift= 1.5pt]b.west);
\path[->,font=\scriptsize,>=angle 90]
([yshift= -1.5pt]b.west) edge node[below] {$('f)^{!,0} := \Phi_\Y \circ ('f)^! $} ([yshift= -1.5pt]a.east);
\end{tikzpicture}
\end{equation}
Let us also recall that $\ICoh_0(\Z^\wedge_\Y)$ is self-dual and that these two adjoints $('f)_{*,0}$ and $('f)^{!,0}$ are dual to each other.

\ssec{$\2$-categorical functoriality} \label{ssec:2-cat functoriality Hgoem}

In this section we enhance the assignment 
$$
\X
\squigto
\ICoh_0((\X \times\X)^\wedge_\X) =: \H(\X), 
$$
$$
[\X \leftto \W \to \Y]
\squigto
\ICoh_0((\X \times \Y)^\wedge_\W)
=: \Hgeomcorr \X \W \Y
$$
to a lax $\2$-functor
$$
\H^\geom: 
\Corr \bigt{\Stkevcoclfp }_{\evcoc;\all}
^{\schem \& \evcoc \& \proper}
\longto
\ALGBimod(\DGCat),
$$
which we will be prove to be strict towards the end of the paper (Theorem \ref{main-thm-Hgoem is strict}).
Here, we have used the notation $\H^{\geom}$ for emphasis, as later we will encounter a categorical construction producing a lax $\2$-functor $\Hcat$. We will eventually show that these two lax $\2$-functors are identified and then denoted simply by $\H$.

\begin{rem}
The condition of boundedness of the horizontal arrows is necessary to have a well-defined $\infty$-category of correspondences.
\end{rem}

\sssec{}

We begin by observing that, for any $\X \in \Stk$, the DG category 
$$
\bbI^{\wedge,\geom}  (\X)
:=
\ICoh(\form \X \X \X)
$$
possesses a convolution monoidal structure and that, for any correspondence $[\Y \leftto \W \to \Z]$ in $\Stk$, the DG category
$$
\bbI^{\wedge,\geom}_{\Y \leftto \W \to \Z} 
:=
\ICoh((\X \times \Y)^\wedge_\W)
\simeq
\ICoh( \form \Y \Y \W_\dR \times_{\Z_\dR} \Z)
$$
admits the structure of an $(\bbI^{\wedge,\geom}(\Y), \bbI^{\wedge,\geom}(\Z))$-bimodule.

\sssec{} \label{sssec:ICOHW goem out of corresp}

Let us now enhance the assignment 
$$
\X
\squigto
\ICoh((\X \times\X)^\wedge_\X)
=: \bbI^{\wedge,\geom}  (\X), 
$$
$$
[\X \leftto \W \to \Y]
\squigto
\ICoh((\X \times \Y)^\wedge_\W)
=: \bbI^{\wedge,\geom}_{\X \leftto \W \to \Y}
$$
to a \emph{lax} $\2$-functor
\begin{equation}
\label{eqn:ICOHW-geom-extra-general}
\bbI^{\wedge,\geom}:
\Corr(\Stk)_{\all;\all}^{\schem \& \proper}
\longto
\ALGBimod(\DGCat).
\end{equation}
To construct it, we first appeal to the lax symmetric monoidal structure on (\ref{funct:ICOH out of corresp}): Section \ref{sssec:from lax monoidal to alg and bimodules} yields a lax $\2$-functor
$$ 
\ICoh:
\ALGBimod
\bigt{
\Corr(\PreStk_\laft)_{\text{ind-inf-schem; all}}^{\text{ind-inf-sch \& ind-proper}} 
}
\longto 
\ALGBimod(\DGCat).
$$
It remains to precompose with the lax $\2$-functor
\begin{equation} \label{eqn:LOOP geometric}
\Corr(\Stk)_{\all;\all}^{\schem \& \proper}
\longto
\ALGBimod(\Corr(\PreStk_\laft))_{\text{ind-inf-schem; all}}^{\text{ind-inf-sch \& ind-proper}})
\end{equation}
that sends 
$$
\Y \squigto \Yform;
$$
\vspace{.1cm}
$$
[\Y \leftto \W \to \Z]
\squigto
\Y^\wedge_{\W} \ustimes{\W_\dR} \Z^\wedge_{\W}
\simeq
\Y \times_{\Y_\dR} \W_\dR \times_{\Z_\dR} \Z;
$$
$$ %%%best-double-arrows-diagram-prototype
\begin{tikzpicture}[scale=1.5]
\node (0) at (0,0) {$[\Y$};
\node (1u) at (1,.3) {$\U$};
\node (1d) at (1,-.3) {$\W$};
\node (2) at (2,0) {$ \Z] \squigto  $};
\node (3) at (5,0)
{$
\bigt{
\Y \times_{\Y_\dR} \U_\dR \times_{\Z_\dR} \Z
\xto{\wt{f_\dR}}
\Y \times_{\Y_\dR} \W_\dR \times_{\Z_\dR} \Z
}$.};
%%%HORIZ
\path[->,font=\scriptsize,>=angle 90]
(1u.west) edge node[above] {} (0.north east);
\path[->,font=\scriptsize,>=angle 90]
(1d.west) edge node[above] {} (0.south east);
\path[->,font=\scriptsize,>=angle 90]
(1u.east) edge node[above] {} (2.north west);
\path[->,font=\scriptsize,>=angle 90]
(1d.east) edge node[above] {} (2.south west);
\path[->,font=\scriptsize,>=angle 90]
(1u.south) edge node[right] {$f$} (1d.north);
\end{tikzpicture}
$$
Observe that the requirement that $f$ be schematic and proper implies that $f_\dR$, and hence $\wt{f_\dR}$, is inf-schematic and ind-proper.

\begin{rem}
The lax $\2$-functor (\ref{eqn:LOOP geometric}) is a geometric version of the formally similar lax $\2$-functor (\ref{eqn:loop-mod}).
\end{rem}

\sssec{}

Let us now turn to the construction of $\H^\geom$. For $\Y \in \Stkevcoclfp$, the canonical inclusion
$$
\iota:
\ICoh_0(\form \Y \Y\Y)
\hto
\ICoh(\form \Y \Y\Y)
$$
is monoidal. 
Moreover, the left action of $\ICoh_0(\Yform)$ on $\ICoh( (\Y \times \Z)^\wedge_\W)$ preserves the subcategory $\ICoh_0( (\Y \times \Z)^\wedge_\W)$. This is an easy diagram chase left to the reader.

\medskip

Thus, we are in the position to apply the paradigm of Section \ref{sssec: paradigm: subfunctor of a lax-functor} to obtain a lax $\2$-functor 
\begin{equation} \label{funct:Hgeom out of corresp}
\H^\geom: 
\Corr \bigt{\Stkevcoclfp }_{\evcoc;\all}
^{\schem \& \evcoc \& \proper}
\longto
\ALGBimod(\DGCat),
\end{equation}
as desired. We repeat here that one of the goals of this paper is to show that such lax $\2$-functor is actually strict: this is accomplished in Theorem \ref{main-thm-Hgoem is strict}.
In the next section, we give an overview of the strategy of the proof of such theorem. This could serve as a guide through the constructions of the remainder of the present article.

\ssec{Outline of the proof of Theorem \ref{main-thm-Hgoem is strict}}

It suffices to prove that the lax $\2$-functor $\Hgeom:\Corr \bigt{\Stkevcoclfp }_{\evcoc;\all}
\to
\ALGBimod(\DGCat)$ is strict. We will proceed in stages.

\sssec{}

First, we look at the restriction of $\H^\geom$ along the functor 
$$
\Affevcoclfp \to \Corr \bigt{\Stkevcoclfp }_{\evcoc;\all}
$$
which is the natural inclusion on objects, and 
$[S \to T] \squigto [S \xleftarrow{=} S \to T]$ on $1$-morphisms.

Using results from the theory of ind-coherent sheaves, we show in Theorem \ref{thm:H-strict-on-schemes} that such lax $\2$-functor is strict. By definition, this is simply the functor $\H: \Affevcoclfp \to \AlgBimod(\DGCat)$ discussed in the introduction, Section \ref{intro-sssec:H on affine}.

\sssec{}

Next, we show that the restriction of $\H^\geom$ to $\Corr(\Affevcoclfp)_{bdd;\all}$ is strict (Corollary \ref{cor:H-geom-strict-on-Corr Aff}). We do so in an indirect way, by establishing some important duality properties of $\H$.
Namely, we show that, for each map $U \to T$ in $\Affevcoclfp$, the bimodule $\Hdx UT$ admits a right dual (which happens to be a left dual as well), denoted by $\Hsx TU$. Having such right duals allows to form the bimodules
$$
\Hcorr SUT := \Hsx SU 
\usotimes{\H(U)} \Hdx UT,
\hspace{.4cm}
\Hopcorr SVT 
:=
\Hdx SV
\usotimes{\H(V)} \Hsx VT.
$$
We also show that $\Hopcorr SVT \simeq \Hcorr S{S \times_V T} T$ naturally, provided that at least one arrow between $S \to V$ and $T \to V$ is bounded. This is enough to extend $\H$ to a strict functor 
$$
\H^\Corr: \Corr(\Affevcoclfp)_{bdd;\all} \longto \AlgBimod(\DGCat).
$$
By inspection, such functor coincides with the restriction of $\H^\geom$ to $\Corr(\Affevcoclfp)_{bdd;\all}$, whence the latter is also strict.

\begin{rem} \label{rem:ambidexterity-comment}
The fact that left and right duals coincide implies that we could have also defined $\H^\Corr$ on $\Corr(\Affevcoclfp)_{\all; bdd}$. These two versions of $\H^\Corr$, exchanged by duality, agree on $\Corr(\Affevcoclfp)_{bdd; bdd}$.
\end{rem}

\sssec{}

To study $\H^\geom$ on stacks, we introduce the sheaf theory $\ShvCatH$, which is the right Kan extension of the functor 
$$
(\Affevcoclfp)^\op \to \inftyCat,
\hspace{.4cm}
S \squigto \H(S) \mmod.
$$
Note that Theorem \ref{thm:H-strict-on-schemes} is essential to make this well-defined. 

In principle, $\ShvCatH$ comes equipped only with pullback functors. However, thanks to the existence of the right duals $\Hsx TS$, there are also $*$-pushforward functors (right adjoints to pullbacks), which turn out to satisfy base-change against pullbacks. Symmetrically, the existence of the left duals provides $!$-pushforward functors (left adjoints to pullbacks), also satisfying base-change against pullbacks.\footnote{We will eventually show that pullbacks in $\ShvCatH$ are ambidextrous (i.e., $*$-pushforwards coincide with $!$-pushforwards), but this requires the $\H$-affineness theorem first.}

\sssec{}

In Theorem \ref{thm:H-affineness-stacks}, we prove the $\H$-affineness theorem, which states that, for any $\Y \in \Stkevcoclfp$, the $\infty$-category $\ShvCatH(\Y)$ is equivalent to $\Hgeom(\Y) \mmod$.
This theorem, together with the above base-change properties, automatically upgrades the assignment $\Y \squigto \Hgeom(\Y)$ to a strict $\2$-functor out of $\Corr(\Stkevcoclfp)_{bdd;\all}$. Fortunately, such functor is easily seen to match with $\Hgeom$, thereby proving that the latter is strict, too.

\sssec{}

An important technical result, which we use frequently, is the smooth descent property for $\ShvCatH$, proven in Section \ref{ssec:descent}: any object $\C \in \ShvCatH(\Y)$ is determined by its restrictions along \emph{smooth} maps $S \to \Y$, with $S$ affine. This is a very convenient simplification. For instance, let $\ICoh_{/\Y} \in \ShvCatH(\Y)$ be the sheaf corresponding to $\ICoh(\Y) \in \H^\geom(\Y) \mmod$ via $\H$-affineness.
In Section \ref{ssec:ICOH}, we will show that the restriction of $\ICoh_{/\Y}$ along a smooth map $S \to \Y$ is the $\H(S)$-module $\ICoh(S)$, whereas the restriction along a non-smooth map does not admit such a simple description.

\sec{Coefficient systems for sheaves of categories} \label{sec:Coeff systems}

In this section, we introduce one of the central notions of this paper, the notion of \emph{coefficient system}, together with its companion notion of \emph{lax coefficient system}. 

We present a list of examples, and, in particular, we define the coefficient system $\H$ related to Hochschild cochains. Let us anticipate that $\H$ arises naturally as a lax coefficient system and some work is needed in order to prove that it is actually strict. (Here and later, the adjective \virg{strict} is used to emphasize that a certain coefficient system is a genuine one, not a lax one.)

\ssec{Definition and examples}

Consider the $\2$-category $\ALGBimod(\DGCat)$, whose objects are monoidal DG categories, whose $1$-morphisms are bimodule categories, and whose $2$-morphisms are functors of bimodules.
Recall that the $\1$-category underlying $\ALGBimod(\DGCat)$ will be denoted by $\AlgBimod(\DGCat)$.

\medskip

A coefficient system is an functor
$$
\A: \Aff \longto \AlgBimod(\DGCat).
$$
A lax coefficient system is a lax $\2$-functor
$$
\A: \Aff \longto \ALGBimod(\DGCat).
$$

\sssec{}

Thus, a lax coefficient system $\bbA$ consists of:
\begin{itemize}
\item
a monoidal category $\bbA(S)$, for each affine scheme $S$;
\item
an $(\bbA(S), \bbA(T))$-bimodule $\bbA_{S \to T}$ for any map of affine schemes $S \to T$;
\item
an $(\bbA(S), \bbA(U))$-linear functor
$$
\eta_{S \to T \to U}: \bbA_{S \to T} \usotimes{\bbA(T)} \bbA_{T \to U}
\longto
\bbA_{S \to U}
$$
for any string $S \to T \to U$ of affine schemes;
\item
natural compatibilities for higher compositions.
\end{itemize}
Clearly, such $\bbA$ is a strict (that is, non-lax) coefficient system if and only if all functors $\eta_{S \to T \to U}$ are equivalences.

\sssec{}

One obtains variants of the above definitions by replacing the source $\infty$-category $\Aff$ with a subcategory $\Aff_\type$, where \virg{$\type$} is a property of affine schemes. For instance, we will often consider $\Aff_\aft$ (the full subcategory of affine schemes almost of finite type) or $\Affevcoclfp$ (affine schemes that are bounded and locally of finite presentation).

\medskip

We now give a list of examples of (lax) coefficient systems, in decreasing order of simplicity.

\sssec{Example 1}

Any monoidal DG category $\CA$ yields a \virg{constant} coefficient system $\ul\CA$ whose value on $S \to T$ is $\CA$, considered as a bimodule over itself.

\sssec{Example 2}

Slightly less trivial: coefficient systems induced by a functor $\Aff \to \Alg(\DGCat)^\op$  via the functor $\iota_{\Alg \to \Bimod}$ defined in (\ref{eqn:from-Alg-to-AlgBimod}). These coefficient systems are automatically strict.

For instance, we have the coefficient system $\bbQ$ which sends 
$$
S \squigto \QCoh(S), 
\hspace{.8cm}
[S \to T] \squigto \QCoh(S) \in (\QCoh(S), \QCoh(T)) \bbimod.
$$
Similarly, we have $\bbD$, obtained as above using $\fD$-modules rather than quasi-coherent sheaves. This coefficient system is defined only out of $\Aff_\aft \subset \Aff$.

\sssec{Example 3}

Let us pre-compose the lax $\2$-functor 
$$
\Loop_{\Mod}: \Mod(\DGCat)^\op
\longto
\ALGBimod(\DGCat)
$$ 
of Section \ref{sssec:loop-mod} with the functor
$$
\Aff_\aft \longto \Mod(\DGCat)^\op,
\hspace{.5cm}
S \squigto (\Dmod(S) \circlearrowright \ICoh(S) )
$$
that encodes the action of $\fD$-modules on ind-coherent sheaves. 
Since $\ICoh(S)$ is self-dual as a $\Dmod(S)$-module (Corollary \ref{cor:ICOH-self-dual-over-Dmod}), we obtain a lax coefficient system 
$$
\ICohW: \Aff_{\aft}
\longto
 \AlgBimod(\DGCat)
$$
described informally by
\begin{eqnarray}
\nonumber
& &  S \squigto \ICoh(\Sform) 
\\
\nonumber
& & [S \to T] \squigto \ICoh(\form STT) \in (\ICoh(\Sform), \ICoh(\form TTT)) \bbimod,
\\
\nonumber
& & [S \to T \to U] \squigto 
\ICoh(\form STT) \usotimes{ \ICoh(\form TTT)} \ICoh(\form TUU)
\longto \ICoh(\form SUU).
\end{eqnarray}
In other words, $\ICohW$ is obtained by restricting the very general $\bbI^{\wedge,\geom}$ defined in Section \ref{sssec:ICOHW goem out of corresp} to $\Aff_\aft$.
We will prove that $\bbI^\wedge$ is strict in Proposition \ref{prop: ICOH-geom-strict-functor-on-schemes}.

\sssec{Example 4}

As a variation of the above example, let $\H$ be the lax coefficient system
$$
\H:
\Affevcoclfp
\longto
\ALGBimod(\DGCat)
$$
defined by 
\begin{eqnarray}
\nonumber
& &  S \squigto \H(S) := \ICoh_0(\Sform) 
\\
\nonumber
& & [S \to T] 
\squigto 
\Hdx ST := \ICoh_0(\form STT) \in (\H(S), \H(T)) \bbimod,
\\
\nonumber
& & [S \to T \to U] \squigto 
\Hdx ST  \usotimes{ \H(T) } \Hdx TU 
\longto 
\Hdx SU.
\end{eqnarray}
Similarly to $\ICohW$, this is the restriction of (\ref{funct:Hgeom out of corresp}) to affine schemes. We will show that $\H$ is strict too.

\medskip

The importance of $\H$ comes from the monoidal equivalence
$$
\H(S) \simeq \HC(S)^\op \mod.
$$ 
To be precise, we have the following. First, the equivalence $\H(S) \simeq \HC(\ICoh(S))^\op \mod$ is obvious. Second, \cite[Proposition F.1.5.]{AG1} provides a natural isomorphism $\HC(\ICoh(S)) \simeq \HC(\QCoh(S)) =: \HC(S)$ of $E_2$-algebras. 

\sssec{Example 5}

One last example arising in a geometric fashion. Let $\Y: \Aff \to \Corr(\PreStk)_{\all;\all}^{\all}$ be an arbitrary lax $\2$-functor, described informally by the assignments
$$
S \squigto \Y_S, 
\hspace{.5cm}
[S \to T] \squigto
\Y_S \leftto \Y_{S \to T} \to \Y_T.
$$
The lax structure amounts to the data of maps
\begin{equation} \label{eqn:maps-between-transfer-spaces}
\Y_{S \to T} \ustimes{\Y_T} \Y_{T \to U} \longto \Y_{S \to U}
\end{equation}
over $\Y_S \times \Y_U$, for any string $S \to T \to U$. Recalling now the paradigm of Section \ref{sssec:example of lax 2 functor}, we obtain a lax $\2$-functor 
$$
 \Corr(\PreStk)_{\all;\all}^{\all} 
 \longto
 \ALGBimod(\DGCat)
$$
defined by sending
$$
\Y_S \squigto \QCoh(\Y_S), 
\hspace{.5cm}
[\Y_S \leftto \Y_{S \to T} \to \Y_T] 
\squigto
\QCoh(\Y_{S \to T}).
$$
The combination of this with $\Y$ yields a lax coefficient system, which is strict if the maps (\ref{eqn:maps-between-transfer-spaces}) are isomorphisms and the prestacks $\Y_{S \to T}$ are nice enough\footnote{Namely, nice enough so that $\QCoh$ interchanges fiber products among these prestacks with tensor products of categories. For instance, $1$-affine algebraic stacks are nice enough.}.

\sssec{Sub-example: singular support}

The theory of singular support provides an important example of the above construction: the assignment
$$
[S \to T] \squigto 
\Sing(S)/\Gm \leftto S \times_T \Sing(T)/\Gm \to \Sing(T)/\Gm,
$$
where $\Sing(U) := \Spec (\Sym_{H^0(U,\O_U)} H^1(U,\Tang_U))$ is equipped with the obvious weight-$2$ dilation action.

\medskip

We obtain a coefficient system $\bbS': \Aff_{\qsmooth} \longto \AlgBimod(\DGCat)$ defined on quasi-smooth affine schemes.
By construction, if $\C$ is a module category over $\bbS'(U)$, then objects of $\C$ are equipped with a notion of support in $\Sing(U)$, see \cite{AG1} for more details.

%%%%\begin{rem}
%%%%The coefficient system $\SS'$ is related but not quite equal to the $\SS$ mentioned in the introduction. In fact, the $\Gm$-action incorporated in the former allowed us to eliminate the shifts by $-2$ appearing in Section \ref{sssec:hierarchy}. However, keeping the shifts is important for the discussion of Section \ref{ssec:change of coeff}, in particular to construct the arrows $\Q \to \SS \to \H$.
%%%%\end{rem}

\ssec{The coefficient system $\ICohW$}

Let us prove that $\ICohW$ and $\H$ are strict coefficient systems. We will need to use the following fact.

\begin{lem} \label{lem:tensor-product-ICOH-over-Dmod}
For any diagram $Y \to W \leftto Z$ in $\Sch_\aft$, exterior tensor product yields the equivalence
\begin{equation} \label{eqn:tensor-prod-ICOH-over-Dmod}
\ICoh(Y) \usotimes{\Dmod(W)} \ICoh(Z)
\xto\simeq
\ICoh(Y \times_{W_\dR} Z).
\end{equation}
\end{lem}

\begin{proof}
Note that $\form YWZ \simeq (Y \times Z)^\wedge_{Y \times_W Z}$. Hence, by \cite[Proposition 3.1.2]{AG2}, the RHS is equivalent to 
$$
\QCoh(Y \times_{W_\dR} Z)
\usotimes{\QCoh(Y \times Z)}
\ICoh(Y \times Z),
$$
while the LHS is obviously equivalent to 
$$
\bigt{
\QCoh(Y) \usotimes{\Dmod(W)} \QCoh(Z)
}
\usotimes{\QCoh(Y \times Z)}
\ICoh(Y \times Z).
$$
Now, the statement reduces to the analogous statement with $\ICoh$ replaced by $\QCoh$, in which case it is well-known.
\end{proof}

\begin{cor} \label{cor:ICOH-self-dual-over-Dmod}
For $Y \in \Sch_\aft$, the DG category $\ICoh(Y)$ is self-dual as a $\Dmod(Y)$-module.
\end{cor}

\begin{proof}
One uses the equivalence of the above lemma to write the evaluation and coevaluation as standard pull-push formulas.
\end{proof}

\begin{cor} \label{cor:functors ICoh over D as fiber products}
For any map $Y \to Z$ in $\Sch_\aft$, we obtain a natural equivalence
$$
\ICoh ( \form YZZ)
\simeq
\Fun_{\Dmod(Z)}(\ICoh(Y), \ICoh(Z) ).
$$
In the special case $Y=Z$, the \virg{composition} monoidal structure on the RHS corresponds to the \virg{convolution} monoidal structure on LHS.
\end{cor}

\sssec{}

The lax-coefficient system $\ICohW$ is the restriction of the lax $\2$-functor  $\ICoh^{\wedge,\geom}$ to $\Aff_\aft$. Consider now the intermediate  lax $\2$-functor
$\Sch_\aft
\to
\ALGBimod(\DGCat)$, denoted also $\ICohW$ by abuse of notation. Our present goal is to prove the following result.

\begin{prop} \label{prop: ICOH-geom-strict-functor-on-schemes}
The lax $\2$-functor
$$
\ICohW:
\Sch_\aft
\longto
\ALGBimod(\DGCat)
$$ 
is strict.
\end{prop}

 The proof of the above proposition will be explained after some preparation.

\sssec{} 

For $Y \in \Sch_\aft$, Corollary \ref{cor:functors ICoh over D as fiber products} shows that $\ICoh(Y)$ admits the structure of an $(\ICoh(\form YYY), \Dmod(Y))$-bimodule, as well as 
the structure of a $(\Dmod(Y), \ICoh(\form YYY))$-bimodule.
Now, one verifies directly that the latter bimodule is left dual to the former, i.e., there is an adjunction
\begin{equation} \label{eqn:ICoh-form-Dmod-adjunction}
\begin{tikzpicture}[scale=1.5]
\node (a) at (0,1) {$\ICoh(\form YYY) \mmod$};
\node (b) at (5,1) {$\Dmod(Y) \mmod.$};
\path[->,font=\scriptsize,>=angle 90]
([yshift= 1.8pt]a.east) edge node[above] {$\ICoh(Y) \usotimes{\ICoh(\form YYY)} -$} ([yshift= 1.8pt]b.west);
\path[->,font=\scriptsize,>=angle 90]
([yshift= -1.5pt]b.west) edge node[below] {$\ICoh(Y) \usotimes{\Dmod(Y)} -$} ([yshift= -1.5pt]a.east);
\end{tikzpicture}
\end{equation}

\begin{lem} \label{lem:ICoh--wedge-equivalence-Dmod}
These two adjoint functors form a pair of mutually inverse equivalences.
In particular, we also have an adjunction in the other direction:
\begin{equation} \label{eqn:ICoh-form-Dmod-opposite-adjunction}
\begin{tikzpicture}[scale=1.5]
\node (a) at (0,1) {$\Dmod(Y) \mmod$};
\node (b) at (5,1) {$\ICoh(\form YYY) \mmod.$};
\path[->,font=\scriptsize,>=angle 90]
([yshift= 1.8pt]a.east) edge node[above] {$\ICoh(Y) \usotimes{\Dmod(Y)} -$} ([yshift= 1.8pt]b.west);
\path[->,font=\scriptsize,>=angle 90]
([yshift= -1.5pt]b.west) edge node[below] {$\ICoh(Y) \usotimes{\ICoh(\form YYY)} -$} ([yshift= -1.5pt]a.east);
\end{tikzpicture}
\end{equation}
\end{lem}

\begin{proof}
The left adjoint in (\ref{eqn:ICoh-form-Dmod-adjunction}) is fully faithful by (\ref{eqn:tensor-prod-ICOH-over-Dmod}) and the right adjoint is colimit-preserving. By Barr-Beck, it suffices to show that the right adjoint in (\ref{eqn:ICoh-form-Dmod-adjunction})  is conservative, a statement which is the content of the next lemma.
\end{proof}

\begin{lem} \label{lem:ICOH(S)-Dmod(S)-conservative}
For $Y \in \Sch_\aft$, the functor
$$
\ICoh(Y) \usotimes{\Dmod(Y)} -: \Dmod(Y) \mmod \longto \DGCat
$$ 
is conservative. 
\end{lem}

\begin{proof}
Let $f: \M \to \N$ be a $\Dmod(Y)$-linear functor with the property that 
$$
\id \otimes f: \ICoh(Y) \usotimes{\Dmod(Y)} \M
\longto 
\ICoh(Y) \usotimes{\Dmod(Y)} \N
$$
is an equivalence. We need to show that $f$ itself is an equivalence.

\medskip

Denote by $\wh Y_\bullet$ the Cech nerve of $q: Y \to Y_\dR$. Recall that the natural arrow
$$
\Dmod(Y) := \ICoh(Y_\dR) \longto \ICoh(| \wh Y_\bullet |) \simeq \lim_{[n] \in \bDelta} \ICoh( \wh Y_n)
$$
is an equivalence and that each of the structure functors composing the above cosimplicial category admits a left adjoint (indeed, each structure map $\wh Y_m \to \wh Y_n$ is a nil-isomorphism between inf-schemes).
Consequently, the tautological functor
$$
\C \longto \lim_{[n] \in \bDelta} \Bigt{ \ICoh(\wh Y_n) \usotimes{\Dmod(Y)} \C }
$$
is an equivalence for any $\C \in \Dmod(S) \mmod$. Under these identifications, our functor $f: \M \to \N$ is the limit of the equivalences 
$$
\id \otimes f: 
\ICoh(\wh Y_n) \usotimes{\Dmod(Y)} \M
\longto
\ICoh(\wh Y_n) \usotimes{\Dmod(Y)} \N,
$$
whence it is itself an equivalence.
\end{proof}

\sssec{}

We are now ready for the proof of the proposition left open above.

\begin{proof} [Proof of Proposition \ref{prop: ICOH-geom-strict-functor-on-schemes}]
Thanks to (\ref{eqn:tensor-prod-ICOH-over-Dmod}), it suffices to prove that, for any $Y \in \Sch_\aft$, the obvious functor $q_*^{\ICoh} \circ \Delta^!: \ICoh(Y) \otimes \ICoh(Y) \to \Dmod(Y)$ induces an equivalence
\begin{equation} \label{eqn:ICohwedge-eval-into-Dmod}
\ICoh(Y) \usotimes{\ICoh \bigt{ \form YYY }} \ICoh(Y)
\xto{\;\; \simeq \;\;}
 \Dmod(Y).
\end{equation}
The latter is precisely the counit of the adjunction (\ref{eqn:ICoh-form-Dmod-adjunction}), which we have shown to be an equivalence.
\end{proof}

\ssec{The coefficient system $\H$}

Our present goal is to prove Theorem \ref{thm:H-strict-on-schemes}, which states that the lax coefficient system
$$
\H :  \Affevcoclfp 
\longto 
\ALGBimod(\DGCat)
$$ 
is \emph{strict}. Actually, such theorem proves something slightly stronger, i.e., the parallel statement for schemes that are not necessarily affine.

\sssec{}

We need a preliminary result, which is of interest in its own right.

\begin{prop} \label{prop:IndCoh-push-forward}
Let $f:X \to Y$ be a map in $\Schevcoclfp$. Then, the $(\Dmod(X), \H(Y))$-linear functor
\begin{equation} \label{eqn:ICoh-push-forwards}
\IndCoh(X) \usotimes{\H(X)} \Hdx XY
\longto
\IndCoh(Y^\wedge_X)
\end{equation}
obtained as the composition
\begin{eqnarray}
\nonumber
\IndCoh(X) \usotimes{\H(X)} \Hdx XY
& \longto &
\IndCoh(X) \usotimes{\ICohW(X) } \ICohdx XY \\
\nonumber
& \xto{\simeq} & 
\IndCoh(Y^\wedge_X).
\end{eqnarray} 
is an equivalence of categories.
\end{prop}

\begin{proof}
The source category is compactly generated by objects of the form $[C_X, ('f)_*^\ICoh(\omega_X)]$ for $C_X \in \Coh(X)$.
Hence, it is clear that the functor in question, denote it by $\phi$, admits a continuous and conservative right adjoint: indeed, $\phi$ sends 
$$
[C_X, ('f)_*^\ICoh(\omega_X)]
\squigto 
('f)_*^\ICoh(C_X),
$$
whence it preserves compactness and generates the target under colimits. 
It remains to show that $\phi$ is fully faithful on objects of the form $[C_X, ('f)_*^\ICoh(\omega_X)]$.
The nil-isomorphism $\beta: (X \times X)^\wedge_X \to (X \times Y)^\wedge_X$ induces the adjunction
$$ %%%best-adjunction-diagram-prototype
\begin{tikzpicture}[scale=1.5]
\node (a) at (0,1) {$\beta_*^\ICoh : \ICoh((X \times X)^\wedge_X)$};
\node (b) at (4,1) {$\ICoh((X \times Y)^\wedge_X) : \beta^!$.};
\path[->,font=\scriptsize,>=angle 90]
([yshift= 1.5pt]a.east) edge node[above] {$ $} ([yshift= 1.5pt]b.west);
\path[->,font=\scriptsize,>=angle 90]
([yshift= -1.5pt]b.west) edge node[below] {$ $} ([yshift= -1.5pt]a.east);
\end{tikzpicture}
$$ 
Observe that both functors are $\ICoh((X \times X)^\wedge_X)$-linear and preserve the $\ICoh_0$-subcategories.
To conclude the proof, just note that $('f)_*^\ICoh(\omega_X)$ is the image of the unit of $\H(X)$ under $\beta_*^\ICoh$, and use the above adjunction.
\end{proof}

\begin{cor} \label{cor:ICoh-formal-complet-push-forward}
For $f:X \to Y$ as above and $\C$ a right $\ICohW(X)$-module, the natural functor
$$
\C \usotimes{\H(X)} \Hdx XY
\longto
\C \usotimes{\ICohW(X) } \ICohdx XY
$$
is an equivalence.
\end{cor} 

\begin{proof}
It suffices to prove the assertion for $\C = \ICohW(X)$, viewed as a right module over itself. Thanks to the right $\ICohW(X)$-linear equivalence
$$
\ICohW(X) \simeq \ICoh(X) \usotimes{\Dmod(X)} \ICoh(X),
$$
the assertion reduces to the proposition above.
\end{proof}

\begin{thm} \label{thm:H-strict-on-schemes}
The lax $\2$-functor
$$
\H :
\Schevcoclfp 
\longto 
\ALGBimod(\DGCat),
$$ 
obtained by restricting $\Hgeom$ to schemes, is \emph{strict}.
\end{thm}

\begin{proof}
Let $U \to X \to Y$ be a string in $\Schevcoclfp$. We need to prove that the convolution functor
\begin{equation} \label{eqn:ICOH_0-composition-into-ICOh}
\Hdx UX 
\usotimes{\H(X) } 
\Hdx XY
\longto
\ICohdx UY
\end{equation}
is an equivalence onto the subcategory $\Hdx UY \subseteq \ICohdx UY$. One easily checks that the essential image of the functor is indeed $\Hdx UY$, whence it remains to prove fully faithfulness.
By construction, (\ref{eqn:ICOH_0-composition-into-ICOh}) factors as the composition
$$
\Hdx UX 
\usotimes{\H(X) } 
\Hdx XY
\longto
\ICohdx UX 
\usotimes{\H(X) } 
\Hdx XY
\longto
\ICohdx UY.
$$
Now, the first arrow is obviously fully faithful, while the second one is an equivalence by the above corollary.
\end{proof}

\ssec{Morphisms between coefficient systems}
\renc{\A}{\bbA}
\renc{\B}{\bbB}

Coefficient systems assemble into an $\infty$-category:
$$
\MPreStk := \Fun(\Aff, \AlgBimod(\DGCat)).
$$
Hence, it makes sense to consider \emph{morphisms} of coefficient systems.
This notion has already been discussed in Section \ref{ssec:change of coeff}, where some examples have been given.
Here we just recall the only morphism of interest in this paper, the arrow $\Q \to \H$.

\sssec{}

Let $\A$ and $\B$ be two coefficient systems. Consider the following pieces of data:
\begin{itemize}
\item
for each $S \in \Aff$, a monoidal functor $\A(S) \to \B(S)$;
\item
for each $S \to T$, an $(\A(S), \A(T))$-linear functor
\begin{equation} \label{eqn:lax-morphism-from-A-to-B}
\eta_{S \to T}: 
\A_{S \to T}
\longto
 \B_{S \to T}
\end{equation}
that induces an $(\A(S), \B(T))$-equivalence $\Adx ST \otimes_{\A(T)} \B(T) \to \Bdx ST$;
\item
natural higher compatibilities with respect to strings of affine schemes.
\end{itemize}
These data give rise to a morphism $\A \to \B$.

\sssec{}

It is easy to see that the morphism $\Q \to \H$ (defined on $\Affevcoclfp$) falls under this rubric. Indeed, we just need to verify that the tautological $(\QCoh(S), \H(T))$-linear functor
\begin{equation} \label{eqn:QCoh-H-tensor}
\QCoh(S) \usotimes{\QCoh(T)} \H(T) 
\longto
\Hdx ST
\end{equation}
is an equivalence, for any $S \to T$ in $\Affevcoclfp$. This has been proven in \cite{centerH} in greater generality.

\section{Coefficient systems: dualizability and base-change} \label{sec:base-change}

As mentioned in the introduction, a coefficient system $\A: \Aff_\type \longto \AlgBimod(\DGCat)$ yields a functor
$$
\ShvCatA := \mmod \circ \,\A^\op :(\Aff_\type)^\op \longto \AlgBimod(\DGCat)^\op \xto{\mmod}
\inftyCat
$$
and then, by right Kan extension, a functor
$$
\ShvCatA:
(\Stk_\type)^\op
\longto
\inftyCat
$$
where $\Stk_\type$ denotes the $\infty$-category of algebraic stacks with affine diagonal and with an atlas in $\Aff_\type$.

This is half of what we need to accomplish though: it is not enough to just have pullbacks functors in $\ShvCatA$, we want pushforwards too. Said more formally, we wish to extend $\ShvCatA$ to a functor out of 
$$
\Corr(\Stk_\type)_{\vert;\horiz},
$$
for an appropriate choice of vertical and horizontal arrows. In this section, we inquire this possibility for affine schemes. Actually, we will look for something stronger: we check under what conditions the coefficient system $\A$ itself admits an extension to a functor
\begin{equation} 
\Corr(\Aff_\type)_{\vert;\horiz} \longto \AlgBimod(\DGCat),
\end{equation}
or even better to an $\2$-functor
\begin{equation} \label{eqn:A-corr-affine-format}
\Corr(\Aff_\type)_{\vert;\horiz}^{\adm} \longto \ALGBimod(\DGCat).
\end{equation}

\ssec{The Beck-Chevalley conditions}
\nc{\utype}{{\mathit{u \!\on -\! type}}}

As we now explain, the \emph{(left or right) Beck-Chevalley conditions} are conditions on a coefficient system $\A$ that automatically guarantee the existence of an $\2$-functor $\A^\Corr$ extending $\A$.

\sssec{}

We say that $\A$ satisfies the \emph{right Beck-Chevalley condition} if the two requirements of Sections \ref{sssec:first requirement} and \ref{sssec:second requirement} are met.

\sssec{The first requirement} \label{sssec:first requirement}

We ask that, for \emph{any} arrow $S \to T$ in $\Aff_\type$, the $(\A(S), \A(T))$-bimodule $\Adx ST$ be \emph{right dualizable}, see Section \ref{sssec:right-left-dualizability of bimodules} for our conventions. Let us denote by $\Asx TS$ such right dual.

\sssec{}

Assume now that $\A$ satisfies the above requirement, so that the bimodules $\Asx ??$ are defined.
Before formulating the second requirement, we need to fix some notation. For a commutative (but not necessarily cartesian) diagram
\begin{gather} \label{diag:SUTV}
\xy
(0,15)*+{ U }="01";
(0,0)*+{ S }="00";
(25,15)*+{  T}="11";
(25,0)*+{ V }="10";
%
%horiz
{\ar@{->}^{ f } "00";"10"};
{\ar@{->}^{ F } "01";"11"};
%
%vert
{\ar@{->}^{ G } "01";"00"};
{\ar@{->}^{ g } "11";"10"};
\endxy
\end{gather}	
in $\Aff_\type$, define
$$
\Acorr SUT := \Asx SU \usotimes{\AA(U)} \Adx UT;
\hspace{.4cm}
\Aopcorr SVT := \Adx SV \usotimes{\AA(V)} \Asx VT. 
$$
Denote by $\utype$ the largest class of arrows in $\Aff_\type$ that makes $\Corr(\Aff_\type)_{\all;\utype}$ well-defined.\footnote{The letter \virg{u} in the notation $\utype$ stands for the word \virg{universal}. }
Namely, an arrow $S \to T$ in $\Aff_\type$ belongs to $\utype$ if, for any $T' \to T$ in $\Aff_\type$, the scheme $S \times_T T'$ belongs to $\Aff_\type$.

\sssec{}

Consider a commutative diagram like (\ref{diag:SUTV}). The resulting commutative diagram
\begin{gather} \nonumber
\xy
(0,15)*+{ \A(U) }="01";
(0,0)*+{ \A(S) }="00";
(25,15)*+{  \A(T) }="11";
(25,0)*+{ \A(V) }="10";
%
%horiz
{\ar@{<-}^{ \Adx SV  } "00";"10"};
{\ar@{<-}^{ ì\Adx UT } "01";"11"};
%
%vert
{\ar@{<-}_{ \Adx US } "01";"00"};
{\ar@{<-}^{ \Adx TV } "11";"10"};
\endxy
\end{gather}	
in $\AlgBimod(\DGCat)$ gives rise, by changing the vertical arrows with their right duals, to a \emph{lax} commutative diagram
\begin{gather} \nonumber
\xy
(0,15)*+{ \A(U) }="01";
(0,0)*+{ \A(S) }="00";
(25,15)*+{  \A(T) }="11";
(25,0)*+{ \A(V).}="10";
(12,8)*+{ \nwarrow}="centro";
%
%
%horiz
{\ar@{<-}^{ \Adx SV  } "00";"10"};
{\ar@{<-}^{ ì\Adx UT } "01";"11"};
%
%vert
{\ar@{->}_{ \Asx SU } "01";"00"};
{\ar@{->}^{ \Asx VT } "11";"10"};
\endxy
\end{gather}	
In other words, any commutative diagram (\ref{diag:SUTV}) yields a canonical $(\AA(S),\AA(T))$-linear functor
\begin{equation} \label{eqn:Acorr-natural-map-non-Cart}
\Aopcorr SVT \longto \Acorr S U T.
\end{equation}%

\sssec{The second requirement} \label{sssec:second requirement}

In particular, for $S \to V \in \utype$ and $T \to V$ arbitrary, we have
\begin{equation} \label{eqn:Acorr-natural-map}
\Aopcorr SVT \longto \Acorr S{S \times_V T}T
\end{equation}%
and we ask that such functor be an equivalence.

\sssec{}

Let us now explain what the right Beck-Chevalley condition is good for.
Tautologically, if $\A$ satifies the right Beck-Chevalley condition, the assignment 
\begin{equation}
\label{eqn:assignment-Acorr}
S \squigto \A(S),
\hspace{.5cm}
[S \leftto U \to T ]
\squigto
\Acorr SUT
\end{equation}
extends to a functor
$$
\Corr(\Aff_\type)_{\all;\utype}
\longto
\AlgBimod(\DGCat).
$$
Further, thanks to \cite[Chapter V.1, Theorem 3.2.2]{Book}, the latter automatically extends further to an $\2$-functor
$$
\A^{\RBeck}:
\Corr(\Aff_\type)_{\all;\utype}^{\utype, 2-\op}
\longto
\ALGBimod(\DGCat).
$$
Thus, for $\A$ satisfying the right Beck-Chevalley condition, the corresponding sheaf theory $\restr{\ShvCatA}{\Aff_\type^\op}$ admits $*$-pushforwards (defined to be right adjoint to pullbacks). Moreover, these pushforwards satisfy base-change against pullbacks along the appropriate fiber squares.

\sssec{}

The definition of \emph{left Beck-Chevalley condition} for $\A$ is totally symmetric: each $\Adx ST$ must admit a left dual $\Asx TS^{L}$ and, for any cartesian diagram (\ref{diag:SUTV}) with $T \to V$ in $\utype$, the structure functor
$$
\Asx SU ^L \usotimes{\AA(U)} \Adx UT
\longto
\Adx SV \usotimes{\AA(V)} \Asx VT ^L 
$$
must be an equivalence. 
Thus, if $\A$ satisfies the left Beck-Chevalley condition, the sheaf theory $\restr{\ShvCatA}{\Aff_\type^\op}$ admits $!$-pushforwards (defined to be left adjoint to pullbacks), again, satisfying base-change against pullbacks along the appropriate fiber squares.

\sssec{}

A coefficient system $\A$ is said to be \emph{ambidextrous} if it satisfies the right Beck-Chevalley condition and, for any $S \to T \in \Aff_\type$, the $(\A(T), \A(S))$-bimodule $\Adx ST$ is ambidextrous (see Section \ref{sssec:right-left-dualizability of bimodules} for the definition).
Any ambidextrous $\A$ automatically satisfies the left Beck-Chevalley condition as well. Thus, for $\A$ ambidextrous, we obtain two extensions of $\A$
$$
\A^{\RBeck}:
\Corr(\Aff_\type)_{\utype;\all}^{\utype, 2-\op}
\longto
\ALGBimod(\DGCat)
$$
$$
\A^{\LBeck}: 
\Corr(\Aff_\type)_{\all;\utype}^{\utype}
\longto
\ALGBimod(\DGCat)
$$
that are exchanged by duality. 

\sssec{} \label{sssec:spell out amibexterity}

Let us spell out these pieces of structure in more detail. First, up to switching vertical and horizontal arrows in $\A^{\RBeck}$ (see Remark \ref{rem:ambidexterity-comment}), the two $(\infty, 2)$-functors $\A^{\RBeck}, \A^{\LBeck}$ have a common underlying $(\infty, 1)$-functor
$$
\A^\Corr:
\Corr(\Aff_\type)_{\all;\utype}
\longto
\AlgBimod(\DGCat)
$$
$$
[S \leftto U \to T] \squigto \Acorr SUT := \Asx SU 
\usotimes{\A(U)} \Adx UT.
$$
Secondly, the two enhancements of $\A^\Corr$ to $\A^{\LBeck}$ and $\A^{\RBeck}$ amount to the following data: for $U' \to U$ of $\utype$ over $S \times T$, there are two mutually dual structure functors $\Acorr S{U'}T \rightleftarrows \Acorr S {U} T$, compatible in $U$ in the natural way.
Such enhancements will be used in Sections \ref{sssec: star pushforward shvcatH} and \ref{sssec: !-push for H}
to construct the two kinds of pushforwards in the setting of $\ShvCatH$ on stacks.

\sssec{Easy examples}

It is obvious that $\bbQ$ and $\bbD$ are ambidextrous. For instance, for the former,
$$
\Q^\Corr:
\Corr(\Aff)_{\all;\all}
\longto
\ALGBimod(\DGCat)
$$ 
is defined on $1$-arrows by $\Q_{S \leftto U \to T} \simeq \QCoh(U)$, the latter equipped with its obvious $(\QCoh(S), \QCoh(T))$-bimodule structure.  The two mutually dual structure functors $\Q_{S \leftto U' \to T} \rightleftarrows \Q_{S \leftto U \to T}$ are simply the pullback and pushforward functors along $U' \to U$.

We leave it as an exercise to show that the coefficient system $\bbS'$ responsible for singular support is ambidextrous: it extends to a functor out of $\Corr(\Aff_{\qsmooth})_{\all;\smooth}^{\smooth}$.

\sssec{}

Let us now turn to $\ICohW$. We have the following result, which will later help us understand base-change for $\H$.

\begin{prop}
The functor $\ICohW: \Sch_\aft \to \AlgBimod(\DGCat)$ satisfies the right Beck-Chevalley condition, so that it extends to an $\2$-functor
\begin{equation}
(\ICohW)^{\RBeck}: \Corr(\Sch_\aft)_{\all;\all}^{\all, 2-\op} 
\longto 
\ALGBimod(\DGCat).
\end{equation}
\end{prop}

\begin{proof}
We start by setting up some notation. For $X \to Y$ in $\Sch_\aft$, consider the maps
$$
\zeta: 
(X \times X)^\wedge_X 
\simeq
\form XXX 
\longto
(X \times X)^\wedge_{X \times_Y X}
 \simeq 
 \form XYX,
$$
\begin{equation} \label{eqn:eta}
\eta: 
(Y \times Y)^\wedge_X 
\simeq
Y \times_{Y_\dR} X_\dR \times_{Y_\dR} Y 
\longto
(Y \times Y)^\wedge_Y
\simeq
 \form YYY,
\end{equation}
where $\zeta$ is induced by $\Delta_{X/Y}: X \to X \times_Y X$.
With the help of Lemma \ref{lem:tensor-product-ICOH-over-Dmod}, one can easily check that the functors
$$
\zeta^!: \ICohdx XY \usotimes{\ICohW(Y)} \ICohsx YX
\longto \ICohW(X)
\hspace{.5cm}
\eta^!: \ICohW(Y) \longto 
\ICohsx YX \usotimes{\ICohW(X)} \ICohdx XY 
$$
exhibit $\ICohsx YX := \ICoh(\form YYX)$ as the right dual of the $(\ICohW(X), \ICohW(Y))$-bimodule $\ICohdx XY$.

\medskip

Let now
\begin{gather} 
\xy
(0,10)*+{ U }="01";
(0,0)*+{ T}="00";
(15,10)*+{  S }="11";
(15,0)*+{ V}="10";
%
%horiz
{\ar@{->}^{ f } "00";"10"};
{\ar@{->}^{ F } "01";"11"};
%
%vert
{\ar@{->}^{ G } "01";"00"};
{\ar@{->}^{ g } "11";"10"};
\endxy
\end{gather}	
be a commutative square in $\Sch_\aft$. By Lemma \ref{lem:tensor-product-ICOH-over-Dmod}, one easily gets equivalences
$$
\ICohW_{S \leftto U \to T} 
\simeq
\ICoh(S \times_{S_\dR} U_\dR \times_{T_\dR} T),
\hspace{.4cm}
\ICohW_{S \to V \leftto T} 
\simeq
\ICoh( \form SVT),
$$
compatible with the natural $(\ICohW(S), \ICohW(T))$-bimodule structures on both sides.
Further, the structure arrow induced by the right Beck-Chevalley condition
$$
\ICohW_{S \to V \leftto T}  
\longto 
\ICohW_{S \leftto U \to T} 
$$
is the $!$-pullback functor along the natural map $U_\dR \to (S \times_V T)_\dR$, whence it is an equivalence whenever the square is nil-Cartesian (that is, Cartesian at the level of reduced schemes).
\end{proof}

\begin{rem}
The same argument with the functors $\zeta_*^\ICoh$ and $\eta_*^\ICoh$ shows that $\ICohW$ satisfies the left Beck-Chevalley condition, too. It follows that $\ICohW$ is ambidextrous.
\end{rem}

\ssec{Base-change for $\H$}

The goal of this section is to show that $\bbH$ is ambidextrous (Theorem \ref{thm:H-ambidextrous}). After this is proven, we will summarize the important consequences of this result.

\sssec{}

Observe that, for any $S \in \Affevcoclfp$, the monoidal category $\H(S)$ is rigid and compactly generated. Recall now the definition of $1_{\H(S)}^{fake} \in \H(S)^*$ and the notion of very rigid monoidal category, see Section \ref{sssec:very rigid}.

\begin{prop}
For any $S \in \Affevcoclfp$, the monoidal DG category $\H(S)$ is \emph{very rigid}.
\end{prop}

\begin{proof}
It suffices to show that $1_{\H(S)}^{fake} \in \H(S)^*$ admits a lift through the forgetful functor
$$
\Fun_{\H(S) \otimes \H(S)^\rev} 
\bigt{
\H(S), 
\H(S)^*
}
\longto
\H(S)^*.
$$
Recall from \cite{centerH} that the functor
$$
\Dmod(S)
\xto{\oblv_L}
\QCoh(S)
\xto{\Upsilon_S}
\ICoh(S)
\xto{'\Delta_*^\ICoh}
\H(S)
$$
factors as the composition 
$$
\Dmod(S)
\longto
\Fun_{\H(S) \otimes \H(S)^\rev} 
\bigt{
\H(S), 
\H(S)
}
\longto
\H(S),
$$
where the DG category in the middle is by definition the Drinfeld center of $\H(S)$.
A variation of the argument there shows that 
$$
\Dmod(S)
\xto{\oblv_L}
\QCoh(S)
\xto{\Xi_S}
\ICoh(S)
\xto{'\Delta_*^\ICoh}
\H(S)^*
$$
factors as the composition 
$$
\Dmod(S)
\longto
\Fun_{\H(S) \otimes \H(S)^\rev} 
\bigt{
\H(S), 
\H(S)^*
}
\longto
\H(S)^*.
$$
Finally, one computes $1_{\H(S)}^{fake} \in \H(S)^*$ explicitly: it is readily checked that 
$$
1_{\H(S)}^{fake}
 \simeq 
 \Delta_*^\ICoh(\Xi_S(\O_S)), 
$$
a fact that concludes the proof.
\end{proof}

\sssec{}

Coupling this with Corollary \ref{cor:very rigid has ambidextrous modules}, we obtain that each bimodule $\Hdx ST$ is ambidextrous: moreover, its left and right duals are canonically identified with $(\Hdx ST)^*$. 

Let us now determine the right dual to $\Hdx ST$ explicitly. By the above, we already know what the DG category underlying $\Hsx TS :=(\Hdx ST)^R$ must be: it is the dual of the DG category $\ICoh_0((S \times T)^\wedge_S)$. The latter is self-dual as a plain DG category, so we are just searching for the correct $(\H(T), \H(S))$-bimodule structure on $\ICoh_0((S \times T)^\wedge_S)$.

We claim that $\Hsx TS$ is equivalent to $\ICoh_0((T \times S)^\wedge_S)$, equipped with the obvious $(\H(T), \H(S))$-bimodule structure.
We will establish this fact directly, by constructing the evaluation and coevaluation that make $\ICoh_0((T \times S)^\wedge_S)$ right dual to $\Hdx ST$.

\begin{lem}
For $S \to T$ a map in $\Affevcoclfp$, the natural functor
$$
\Hdx ST 
\usotimes{\H(T)}
\ICoh_0((T \times S)^\wedge_S)
\longto
\ICohdx ST \usotimes{\ICohW(T)} \ICohsx TS 
\simeq
\ICoh(\form STS)
\xto{\zeta^!}
\ICohW(S)
$$
lands into the full subcategory $\H(S) \subseteq \ICohW(S)$.
\end{lem}

\begin{proof}
We will use the following commutative diagram
$$ %%%best-double-arrows-diagram-prototype
\begin{tikzpicture}[scale=1.5]
\node (00) at (0,0) {$S \times_T S$};
\node (10) at (2,0) {$(S \times S)^\wedge_{S \times_T S} $};
\node (11) at (4,.7) {$ (S \times S)^\wedge_S $};
\node (12) at (2,1.4) {$ S $};
\node (02) at (0,1.4) {$ (S \times_{S \times T} S)^\wedge_S $};
%%%HORIZ
\path[->,font=\scriptsize,>=angle 90]
(00.east) edge node[above] {$\xi$} (10.west);
\path[->,font=\scriptsize,>=angle 90]
(02.east) edge node[above] {$\pi$} (12.west);
%%%VERT
\path[->,font=\scriptsize,>=angle 90]
(02.south) edge node[right] {} (00.north);
\path[->,font=\scriptsize,>=angle 90]
(12.east) edge node[above] {$\;\;'\Delta$} (11.north west);
\path[->,font=\scriptsize,>=angle 90]
(11.south west) edge node[below] {$\zeta$} (10.north east);
\path[->,font=\scriptsize,>=angle 90]
(12.south) edge node[right] {$\wt\Delta_{S/T}$} (10.north);
\end{tikzpicture}
$$
with cartesian square. 
The DG category 
$$
\Hdx ST \usotimes{\H(T)}
\ICoh_0((T \times S)^\wedge_S)
$$ 
is generated by a single canonical compact object, which is sent by our functor to $\zeta^! \circ \xi_*^\ICoh(\omega_{S \times_T S}) \in \ICohW(S)$. 
Hence, it suffices to show that the object 
$$
(\wt\Delta_{S/T})^! \circ \xi_*^\ICoh(\omega_{S \times_T S})
\simeq
\pi_*^\ICoh \circ \pi^! (\omega_S)
$$ 
belongs to the image of $\Upsilon_S: \QCoh(S) \hto \ICoh(S)$. This is clear: $\pi_*^\ICoh  \pi^!$ is equivalent as a functor to the universal envelope of the Lie algebroid $\Tang_{S/ S \times T} \to \Tang_S$, and by assumption $\Tang_{S/S \times T}$ belongs to $\Upsilon_S(\Perf(S))$. We conclude as in \cite[Proposition 3.2.3]{AG2}.
\end{proof}

\sssec{}

Hence, we have constructed an $(\H(S), \H(S))$-linear functor
\begin{equation} \label{eqn:evaluation-for-HdxST}
\Hdx ST \usotimes{\H(T)} \ICoh_0((T \times S)^\wedge_S)
\longto
\H(S),
\end{equation}
which will be our evaluation.
To construct the coevaluation, we need another lemma.

\begin{lem} \label{lem:Hcorr_SUT}
For a diagram $S \leftto U \to T$ in $\Affevcoclfp$, the functor
$$
\ICoh_0((S \times U)^\wedge_U) \usotimes{\H(U)} \Hdx UT
\to
\ICohW_{S \leftto U} \usotimes{\ICohW(U)} \ICohW_{U \to T}
\xto{\simeq}
\ICoh(S \times_{S_\dR} U_\dR \times_{T_\dR} T)
\simeq
\ICoh((S \times T)^\wedge_U)
$$
is an equivalence onto the subcategory $\ICoh_0((S \times T)^\wedge_U) \subseteq \ICoh((S \times T)^\wedge_U)$.
\end{lem}

\begin{proof}
Denote by $\phi: U \to S \times T$ and by $'\phi: U \to (S \times T)^\wedge_U$ the obvious maps.
The source DG category is compactly generated by a single canonical object. Base-change along the pullback square
\begin{gather} \nonumber
\xy
(0,00)*+{ U \times U }="00";
(40,0)*+{ \form SUU \times \form UTT }="10";
(0,16)*+{ U}="01";
(40,16)*+{ \form SSU \times_{T_\dR} T }="11";
%
%horiz
{\ar@{<-}_{   } "10";"00"};
{\ar@{->}^{ } "01";"11"};
%
%vert
{\ar@{->}^{ \Delta } "01";"00"};
{\ar@{->}^{  \Delta } "11";"10"};
\endxy
\end{gather}	
shows that such object is sent to $'\phi_*^\ICoh(\omega_U) \in \ICoh((S \times T)^\wedge_U)$, which is a compact generator of $\ICoh_0((S \times T)^\wedge_U)$. It remains to show that the functor
$$
\ICoh_0((S \times U)^\wedge_U) \usotimes{\H(U)} \Hdx UT 
\longto
\ICohsx SU \usotimes{\ICohW(U)} \ICohdx UT 
$$
is fully faithful. This is evident: the functor in question arises as the composition
$$
\ICoh_0((S \times U)^\wedge_U)  \usotimes{\H(U)} \Hdx UT 
\hto
\ICohsx SU \usotimes{\H(U)} \Hdx UT 
\xto{\simeq }
\ICohsx SU \usotimes{\ICohW(U)} \ICohdx UT,
$$
where the second arrow is an equivalence thanks to Corollary \ref{cor:ICoh-formal-complet-push-forward}.
\end{proof}

\sssec{}

We now use $\eta^!: \ICoh((T \times T)^\wedge_T) \longto \ICoh((T \times T)^\wedge_S)$ as in (\ref{eqn:eta}), together with the equivalence 
$$
\theta:
\ICoh_0((T \times S)^\wedge_S)
 \usotimes{\H(S)} \Hdx ST \to  \ICoh_0( (T \times T)^\wedge_S)
$$
of the above lemma, to construct the functor
\begin{equation} \label{eqn:coevaluaiton-for-HdxST}
\H(T) 
\xto{\eta^!} 
\ICoh_0( (T \times T)^\wedge_S)
\xto{\theta^{-1} }
\ICoh_0((T \times S)^\wedge_S)
 \usotimes{\H(S)} \Hdx ST.
\end{equation}
As the next proposition shows, this is the coevaluation we were looking for.

\begin{prop}
Let $f: S \to T$ be a map in $\Affevcoclfp$. Then the functors (\ref{eqn:evaluation-for-HdxST}) and (\ref{eqn:coevaluaiton-for-HdxST}) exhibit $\ICoh_0((T \times S)^\wedge_S)$, with its natural $(\H(T), \H(S))$-bimodule structure, as the right dual of $\Hdx ST$.
\end{prop}

\begin{proof}
This follows formally from the analogous statement for $\ICohdx ST$.
\end{proof}

\sssec{}

Henceforth, we will freely use the $(\H(T),\H(S))$-linear equivalence $\Hsx TS \simeq \ICoh_0((T \times S)^\wedge_S)$.
We are finally ready to settle the ambidexterity of the coefficient system $\H$. 

\begin{thm} \label{thm:H-ambidextrous}
The coefficient system $\H: \Affevcoclfp \longto \AlgBimodDGCat$ is ambidextrous.
\end{thm}

Half of the proof of this theorem has been done in Lemma \ref{lem:Hcorr_SUT}. It remains to add the following statement.

\begin{lem} 
Let $S \to V \leftto T$ be a diagram in $\Affevcoclfp$, with either $S \to V$ or $T \to V$ bounded.\footnote{This ensures that $S \times_V T$ is bounded, so that $\ICoh_0((S \times T)^\wedge_{S \times_V T})$ is well-defined.}
Then the functor
$$
\Hdx SV \usotimes{\H(V)} \Hsx VT 
\longto
\ICohdx {S}{V} \usotimes{\ICohW(V)} \ICohsx VT 
\xto\simeq
\ICoh(\form SVT)
$$
is an equivalence onto the subcategory 
$$
\ICoh_0((S \times T)^\wedge_{S \times_V T})
\subseteq \ICoh(\form SVT).
$$
\end{lem}

\begin{proof}
Let $\xi: S \times_V T \to (S \times T)^\wedge_{S \times_V T} \simeq \form SVT$ be the canonical map.
As before, $\Hdx SV \usotimes{\H(V)} \Hsx VT $ is compactly generated by its canonical object. 
Now, the functor in question sends such object to $\xi_*^\ICoh(\omega_{S \times_V T})$, which is a compact generator of $\ICoh_0((S \times T)^\wedge_{S \times_V T})$.
Hence, it remains to verify that the functor
$$
\Hdx SV \usotimes{\H(V)} \Hsx VT 
\longto
\ICohdx {S}{V} \usotimes{\ICohW(V)} \ICohsx VT 
$$
is fully faithful. Assume that $S \to V$ is bounded, the argument for the other case is symmetric. 
We have the following sequence of left $\QCoh(S)$-linear fully faithful functors:
\begin{eqnarray}
\nonumber
\Hdx SV \usotimes{\H(V)} \Hsx VT 
& \simeq & 
\QCoh(S) \usotimes{\QCoh(V)} \Hsx VT
\\
\nonumber
& \hto & 
\QCoh(S) \usotimes{\QCoh(V)} \ICohsx VT
\\
\nonumber
& \simeq &
\QCoh(S) \usotimes{\QCoh(V)} \ICoh(V) \usotimes{\Dmod(V)} \ICoh(T).
\end{eqnarray}
To conclude, recall (\cite[Proposition 4.4.2]{ICoh}) that the tautological functor $\QCoh(S) \otimes_{\QCoh(V)} \ICoh(V) \to \ICoh(S)$ is fully faithful whenever $S \to V$ is bounded.
\end{proof}

\sssec{}

Following the template of Section \ref{sssec:spell out amibexterity}, let us summarize the consequences of the ambidexterity of $\H$. First off, we obtain that $\H$ extends to a functor
$$
\H^\Corr: 
\Corr(\Affevcoclfp)_{\all; \evcoc} 
\longto 
\AlgBimod(\DGCat),
$$
which has been shown to send
$$
[S \leftto U \to T] 
\squigto \Hcorr SUT
:= 
\Hsx SU 
\usotimes{\H(U)} 
\Hdx UT
\simeq 
\ICoh_0( (S \times T)^\wedge_U).
$$
In other words, $\H^\Corr$ coincides with the restriction of $\H^\geom$ on $\Corr(\Affevcoclfp)_{\all; \evcoc}$. Therefore:

\begin{cor} \label{cor:H-geom-strict-on-Corr Aff}
The lax $\2$-functor $\H^\geom$ is \emph{strict} when restricted to $\Corr(\Affevcoclfp)_{\all; \evcoc}$.
\end{cor}

\sssec{} \label{sssec:summary of B-Chev for H}

Secondly, $\H^\Corr$ admits two extensions to $\2$-functors
$$
\bbH^\RBeck: \Corr(\Affevcoclfp)_{\all; \evcoc}^{\evcoc, 2-\op} \longto \ALGBimod(\DGCat)
$$
and
$$
\bbH^\LBeck: \Corr(\Affevcoclfp)_{\all; \evcoc}^{\evcoc} \longto \ALGBimod(\DGCat),
$$
described as follows. To a $2$-morphism
$$
[S \leftto U' \to T] \to [S \leftto U \to T]
$$
induced by $U' \to U$ bounded, $\H^\RBeck$ assigns the $!$-pullback
$$
\ICoh_0( (S \times T)^\wedge_U)
\longto \ICoh_0( (S \times T)^\wedge_{U'}),
$$
while the $\H^\LBeck$ assigns the dual $(*,0)$-pushforward
$$
\ICoh_0( (S \times T)^\wedge_{U'})
\longto \ICoh_0( (S \times T)^\wedge_{U}),
$$
which is well-defined thanks to boundedness, see Theorem \ref{thm:1-functoriality of ICohzero}.

\section{Sheaves of categories relative to $\H$} \label{sec:shvcatH}

\renc{\A}{\bbA}

The coefficient system $\H$ allows to define the $\infty$-category $\ShvCatH(\X)$, for any prestack $ \X \in\Fun((\Affevcoclfp)^\op, \inftyGrpd)$. As we are only interested in studying $\ShvCatH$ on algebraic stacks, we only consider the functor
$$
\ShvCatH:
(\Stkevcoclfp)^\op
\longto
\inftyCat,
$$
where $\Stkevcoclfp$ consists of those bounded algebraic stacks that have affine diagonal and perfect cotangent complex.

In this section, we explain several constructions regarding $\ShvCatH$, which we then use to prove our main theorems.
We first show that $\ShvCatH$ satisfies smooth descent. 
Secondly, we discuss pushforwards and base-change as follows: by Theorem \ref{thm:H-ambidextrous}, $\H$ is ambidextrous; accordingly, $\ShvCatH$ will admit extensions to categories of correspondences in two mutually dual ways.
Next, we discuss the notion of $\H$-affineness of objects of $\Stkevcoclfp$: we show that $\ShvCatH(\Y)$ is the $\infty$-category of modules over the monoidal DG category $\Hgeom(\Y)$.
Finally, we deduce that the lax $\2$-functor $\Hgeom$ is actually strict.

\ssec{Descent} \label{ssec:descent}

Define
$$
\ShvCatH: 
(\Stkevcoclfp)^\op
\longto
\inftyCat
$$
to be the right Kan extension of 
$$
\ShvCatH = \mmod \circ \H : (\Affevcoclfp)^\op 
\longto
\inftyCat
$$
along the inclusion $\Affevcoclfp \hto \Stkevcoclfp$.
The purpose of this section is to show that the functor $\ShvCatH$ satisfies smooth descent.

\sssec{}

Objects of 
$$
\ShvCatH(\Y) 
\simeq 
\lim_{S \in (\Affevcoclfp)_{/\Y}} \H(S) \mmod
$$
will be often represented simply by 
$\C \simeq \{\C_S\}_{S \in (\Affevcoclfp)_{/\Y}}$, leaving the coherent system of compatibilities $\Hdx ST \otimes_{\H(T)} \C_T \simeq \C_S$ implicit.
For any $f: \X \to \Y$ in $\Stkevcoclfp$, denote by $f^{*,\H}$ the structure functor. 
Explicitly (and tautologically), $f^{*,\H}$ sends 
$$
\{\C_S\}_{S \in (\Affevcoclfp)_{/\Z}}
\squigto
\{\C_S\}_{S \in (\Affevcoclfp)_{/\Y}}.
$$ 
In what follows, elements of $S \in (\Affevcoclfp)_{/\Y}$ will be denoted by $\phi_{S \to \Y}: S \to \Y$. It is obvious that $(\phi_{S \to \Y })^{*,\H}(\C) = \C_S$.

\begin{thm} \label{thm:smooth-descent-for-ShvCatH}
The functor $\ShvCatH: (\Stkevcoclfp)^\op \to \inftyCat$ satisfies smooth descent.
In particular, for any $\Y$, the restriction functor 
$$
\ShvCatH(\Y)
\longto
\lim_{S \in  (\Affevcoclfp)_{/\Y, \smooth} }
\H(S) \mmod
$$
is an equivalence. Here, $(\Affevcoclfp)_{/\Y, \smooth}$ is the subcategory of $(\Affevcoclfp)_{/\Y}$ whose objects are smooth maps $S \to \Y$ and whose morphisms are triangles $S \to T \to \Y$ with all maps smooth.
\end{thm}

\sssec{}

We will need a few preliminary results that will be stated and proven after having fixed some notation. 

Let $\phi: U \to S$ be a smooth cover in $\Affevcoclfp$ and let $U_\bullet$ be its associated Cech simplicial scheme.
For any arrow $[m] \to [n]$ in $\bDelta^\op$, denote by $\phi_{[m] \to [n]}: U_m \to U_n$ and $\phi_n: U_n \to S$ the induced (smooth) maps. 

\medskip

Now, let $\Y \in \Stkevcoclfp$ be a stack under $S$. The above maps induce functors 
$$
(\Phi_{[m] \to [n]})_{*,0}:
\ICoh_0(\Y^\wedge_{U_n}) 
\longto
\ICoh_0(  \Y^\wedge_S)
$$
$$
(\Phi_n)_{*,0}:
\ICoh_0( \Y^\wedge_{U_n}) 
\longto
\ICoh_0(  \Y^\wedge_S).
$$
We obtain a functor
\begin{equation} \label{eqn:from-ICoh-star-to-ICOH-shriek}
\eps: \uscolim{[n] \in \bDelta^\op} \,
\ICoh_0(\Y^\wedge_{U_n}) 
\longto
\ICoh_0(  \Y^\wedge_S).
\end{equation}

\begin{lem} \label{lem:help for descent}
The functor (\ref{eqn:from-ICoh-star-to-ICOH-shriek}) is an equivalence. 
\end{lem}

\begin{proof}
\renc{\epsilon}{\eps}
Denote by 
$$
\ICoh_0(\Y^\wedge_S)_{[U,*]}
$$
the colimit category appearing in the LHS of (\ref{eqn:from-ICoh-star-to-ICOH-shriek}).
We will proceed in several steps.

\sssec*{Step 1}

We need to introduce an auxiliary category. 
Denote by $(\Phi_n)^?$ and $(\Phi_{[m] \to [n]})^{?}$ the possibly discontinuous right adjoints to $(\Phi_n)_{*,0}$ and $(\Phi_{[m] \to [n]})_{*,0}$.
Consider the cosimplicial DG category 
\begin{equation} \label{cat:cosimpl-category-wtih-?}
\bigt{
\ICoh_0( \Y^\wedge_{U_\bullet}),
(\Phi_{[m] \to [n]})^{?}
}
\end{equation}
and define $\ICoh_0(\Y^\wedge_S)^{[U,?]}$ to be its totalization.
Of course, 
$$
\ICoh_0(\Y^\wedge_S)^{[U,?]}
\simeq
\ICoh_0(\Y^\wedge_S)_{[U,*]}
$$ 
via the usual limit-colimit procedure. However, the former interpretation allows to write $\eps^R$ as the functor
$$
\epsilon^R: 
\ICoh_0(  \Y^\wedge_S)
\longto 
\ICoh_0(\Y ^\wedge_S)^{[U,?]}
$$
given by the limit of the $(\Phi_n)^?$'s.

\sssec*{Step 2}

We will prove the lemma by showing that $\epsilon^R$ is an equivalence. By a standard argument, it suffices to check two facts:
\begin{itemize}
\item
the (discontinuous) forgetful functor
$$
(\Phi_0)^?
:
\ICoh_0(\Y^\wedge_S)^{[U,?]}
\longto
\ICoh_0(\Y^\wedge_U)
$$
is monadic;

\item
the cosimplicial category (\ref{cat:cosimpl-category-wtih-?}) satisfies the monadic Beck-Chevalley condition.
\end{itemize}

\sssec*{Step 3}

In this step, we will prove the first item above.
To this end, we define 
$$
\QCoh(S)_{[U,*]} := \uscolim{[n], \phi_*} \; \QCoh(U_n)
\hspace{.4cm}
\QCoh(S)^{[U,?]} := \lim_{[n], \phi^?} \QCoh(U_n),
$$
where $(\phi_{[m] \to [n]})^?$ is the discontinuous right adjoint to $(\phi_{[m] \to [n]})_*$.
It is easy to see that there is a commutative square
$$
\begin{tikzpicture}[scale=1.5]
\node (00) at (0,0) {$\QCoh(S)^{[U,?]}$};
\node (10) at (4,0) {$\QCoh(U)$.};
\node (01) at (0,1.2) {$\ICoh_0(\Y^\wedge_S)^{[U,?]}$};
\node (11) at (4,1.2) {$\ICoh_0(\Y^\wedge_U)$};
%%HORIZ
\path[->,font=\scriptsize,>=angle 90]
(00.east) edge node[above] {$(\phi_0)^? := ((\phi_0)_*)^R$} (10.west);
\path[->,font=\scriptsize,>=angle 90]
(01.east) edge node[above] {$(\Phi_0)^?$} (11.west);
%%%VERT
\path[->,font=\scriptsize,>=angle 90]
(01.south) edge node[right] { } (00.north);
\path[->,font=\scriptsize,>=angle 90]
(11.south) edge node[right] { } (10.north);
\end{tikzpicture} 
$$
where the vertical arrows are the structure (conservative) functors induced by the morphism $\Q \to \H$. Hence, it suffices to show that the bottom horizontal arrow is monadic, and the latter has been established in \cite[Section 8.1]{shvcat}.

\sssec*{Step 4}
It remains to verify the second item of Step 2 above. This is a particular case of the lemma below.
\end{proof}

\begin{lem} \label{lem:Beck-Chevalley-for-ICoh-formal-complet-?}
Consider a diagram
\begin{equation}
\nonumber
\begin{tikzpicture}[scale=1.5]
\node (U) at (0,0) {$U$ };
\node (V) at (1,0) {$V$};
\node (UU) at (0,1) {$U'$};
\node (VV) at (1,1) {$V'$};
\node (Y) at (2,0) {$Z$};
%%%vertical maps
\path[->,font=\scriptsize,>=angle 90]
(UU.south) edge node[right] { $v'$ } (U.north);
\path[->,font=\scriptsize,>=angle 90]
(VV.south) edge node[right] {  $v$ } (V.north);
%%%horiz maps
\path[->,font=\scriptsize,>=angle 90]
(UU.east) edge node[above] { $ h' $ } (VV.west);
\path[->,font=\scriptsize,>=angle 90]
(U.east) edge node[above] { $ h $ } (V.west);
\path[->,font=\scriptsize,>=angle 90]
(V.east) edge node[above] {  } (Y.west);
\end{tikzpicture}
\end{equation}
in $\Affevcoclfp$, where the square is cartesian with all maps smooth. We do not require that $V \to Z$ be smooth. Then the natural lax commutative diagram
\begin{equation} \label{diag: ? pullbacks for ICOH0 base-change}
\begin{tikzpicture}[scale=1.5]
\node (VZ) at (0,0) { $\ICoh_0(Z^\wedge_V)$ };
\node (VVZ) at (0,1) {$\ICoh_0(Z^\wedge_{V'})$};
\node (UZ) at (3,0) {$\ICoh_0(Z^\wedge_U)$ };
\node (UUZ) at (3,1) {$\ICoh_0(Z^\wedge_{U'})$ };
%%%vertical maps
\path[<-,font=\scriptsize,>=angle 90]
(UUZ.south) edge node[right] { $(\Phi_{v'})^?$ } (UZ.north);
\path[<-,font=\scriptsize,>=angle 90]
(VVZ.south) edge node[right] { $(\Phi_{v})^?$ } (VZ.north);
%%%horiz maps
\path[<-,font=\scriptsize,>=angle 90]
(VZ.east) edge node[above] { $(\Phi_{h})_{*,0}$ } (UZ.west);
\path[<-,font=\scriptsize,>=angle 90]
(VVZ.east) edge node[above] { $(\Phi_{h'})_{*,0}$ } (UUZ.west);
\end{tikzpicture}
\end{equation}
is commutative.\footnote{As usual, for $f: X \to V$ one of the above maps, we have denoted by $(\Phi_f)_{*,0}$ and $(\Phi_f)^?$ the induced functors.}
\end{lem}

\begin{proof}

We proceed in steps here as well.

\sssec*{Step 1}

For $f: X \to V$ a map in $\Sch_\aft$, denote by $\Phi_f: Z^\wedge_X \to Z^\wedge_V$ the induced functor. Recall the equivalence
\begin{equation} \label{eqn:ICOh-formal-completion-as-tensor-product}
\ICoh(Z^\wedge_V) \usotimes{\Dmod(V)} \Dmod(X) 
\xto{\;\;\simeq\;\;} \ICoh(Z^\wedge_X)
\end{equation}
given by exterior tensor product (Lemma \ref{lem:tensor-product-ICOH-over-Dmod}).
One immediately checks that, under such equivalence, $(\Phi_{f})_*^\ICoh$ goes over to the functor
$$
\ICoh(Z^\wedge_V) \usotimes{\Dmod(V)} \Dmod(X) 
\xto{ \id \otimes f_{*,\dR}} 
\ICoh(Z^\wedge_V) \usotimes{\Dmod(V)} \Dmod(V) 
\simeq
\ICoh(Z^\wedge_V).
$$
Thus, whenever $f$ is smooth,  $(\Phi_f)_*^\ICoh$ admits a left adjoint which we denote by $(\Phi_f)^{*,\ICoh}$: this is obtained from the $\fD$-module $*$-pullback  $f^{*,\dR} \simeq f^{!, \dR}[-2 \dim_f]$ by tensoring up. Hence, for $f$ smooth, we have an equivalence
\begin{equation} \label{eqn:shift for * pullbacks and ! in smooth case}
(\Phi_f)^{*,\ICoh} \simeq (\Phi_f)^![-2 \dim_f].
\end{equation}

\sssec*{Step 2}

Applying the above to $h$ and $h'$, we see that the functors $(\Phi_h)^{*,\ICoh}$ and $(\Phi_{h'})^{*,\ICoh}$ preserve the $\ICoh_0$-subcategories. We thus have a diagram
\begin{equation}
\label{diag:*pullbacks for ICoh0 basechange smooth}
\begin{tikzpicture}[scale=1.5]
\node (VZ) at (0,0) { $\ICoh_0(Z^\wedge_V)$ };
\node (VVZ) at (0,1) {$\ICoh_0(Z^\wedge_{V'})$};
\node (UZ) at (3,0) {$\ICoh_0(Z^\wedge_U)$, };
\node (UUZ) at (3,1) {$\ICoh_0(Z^\wedge_{U'})$ };
%%%vertical maps
\path[->,font=\scriptsize,>=angle 90]
(UUZ.south) edge node[right] { $(\Phi_{v'})_{*,0}$ } (UZ.north);
\path[->,font=\scriptsize,>=angle 90]
(VVZ.south) edge node[right] { $(\Phi_{v})_{*,0}$ } (VZ.north);
%%%horiz maps
\path[->,font=\scriptsize,>=angle 90]
(VZ.east) edge node[above] { $(\phi_{h})^{*, \ICoh}$ } (UZ.west);
\path[->,font=\scriptsize,>=angle 90]
(VVZ.east) edge node[above] { $(\phi_{h'})^{*,\ICoh}$ } (UUZ.west);
\end{tikzpicture}
\end{equation}
which is immediately seen commutative thanks to (\ref{eqn:shift for * pullbacks and ! in smooth case}) and base-change for $\ICoh_0$. 

\sssec*{Step 3}

We leave it to the reader to check that the horizontal arrows in the commutative diagram (\ref{diag:*pullbacks for ICoh0 basechange smooth}) are left adjoint to the horizontal arrows of (\ref{diag: ? pullbacks for ICOH0 base-change}). Hence, we obtain the desired assertion by passing to the diagram right adjoint to (\ref{diag:*pullbacks for ICoh0 basechange smooth}).
\end{proof}

\sssec{}

Let us finally prove Theorem \ref{thm:smooth-descent-for-ShvCatH}.

\begin{proof}[Proof of Theorem \ref{thm:smooth-descent-for-ShvCatH}.]

It suffices to prove that the functor $\ShvCatH: (\Affevcoclfp)^\op \to \inftyCat$ satisfies smooth descent.
For $S \in\Affevcoclfp$, let $f: U \to S$ be a smooth cover and $U_\bullet$ the corresponding Cech resolution. Denote by $f_n: U_n \to S$ the structure maps. 
We are to show that the natural functor
$$
\alpha: \H(S) \mmod \longto \lim_{[n] \in \bold\Delta}
\H(U_n) \mmod,
\hspace{.4cm}
\C \squigto \{\Hdx{U_n}{S} \otimes_{\H(S)} \C  \}_{n \in \bDelta }
$$
is an equivalence. 

\medskip

Note that $\alpha$ admits a left adjoint, $\alpha^L$, which sends 
$$
\{\C_n\}_{n \in \bold\Delta} 
\squigto 
\uscolim{[n] \in \bold\Delta^\op} 
\Bigt{
\Hsx S{U_n} \usotimes{\H(U_n)} \C_n
}
,
$$
where we have used the left dualizability of the $\Hdx {U_n}S$.
We will show that $\alpha$ and $\alpha^L$ are both fully faithful.

\medskip

For $\alpha$, it suffices to verify that the natural functor $\alpha^L \circ \alpha(\H(S)) \to \H(S)$ is an equivalence. 
Such functor is readily rewritten as
$$
\eps: \uscolim{[n] \in \bold\Delta^\op} 
\Bigt{
\Hsx S{U_n} \usotimes{\H(U_n)} \Hdx {U_n}S
}
\longto
\H(S).
$$ 
By Lemma \ref{lem:Hcorr_SUT}, our claim is exactly the content of Lemma \ref{lem:help for descent} applied to $\Y = S \times S$.

\medskip

Next, we prove $\alpha^L$ is fully faithful: it suffices to check that the natural functor
$$
\Hdx US \otimes_{\H(S)}
\uscolim{[n] \in \bold\Delta^\op} 
\Bigt{\Hsx S{U_n} \usotimes{\H(U_n)} \C_n}
\longto
\C_0
$$
is an equivalence. Using base-change for $\H$, this reduces to proving that 
$$
\uscolim{[n] \in \bold\Delta^\op} 
\Hcorr {U}{U\times_S U_n}U
\longto
\H(U)
$$ 
is an equivalence. This is again an instance of Lemma \ref{lem:help for descent}.
\end{proof}

\ssec{Localization and global sections} \label{ssec:localiz-global-sect-H-affineness}

Let $\Y \in \Stkevcoclfp$.  In this section, we equip $\ShvCatH(\Y)$ with a canonical object that we denote $\H_{/\Y}$. We then use such object to define a fundamental adjunction and the notion of $\H$-affineness.

\sssec{}

For $S \in \Affevcoclfp$ mapping to $\Y$, consider the left $\H(S)$-module
$$
\Hdx S\Y := \Hgeomdx S\Y = \ICoh_0( (S \times \Y)^\wedge_\Y).
$$
Let $U \to \Y$ be an affine atlas with induced Cech complex $U_\bullet$. By \cite{centerH}, there is a natural left $\H(S)$-linear equivalence
\begin{equation} \label{eqn:tensor up ICOH0 with QCoh for atlas}
\ICoh_0((S \times \Y)^\wedge_S)
\usotimes{\QCoh(\Y)}
\QCoh(U_n)
\simeq
\ICoh_0( (S \times U_n)^\wedge_{S \times_\Y  U_n})
\simeq
\Hcorr S {S \times_\Y U}U
\end{equation}
from which we obtain a left $\H(S)$-linear equivalence
\begin{equation} \label{eqn:writing Hdx SY}
\Hdx S\Y
\simeq
\lim_{U \in (\Affevcoclfp)_{/\Y, \smooth}}
\Hcorr S {S \times_\Y U}U,
\end{equation}
where the limit on the RHS is formed using the $(!,0)$-pullbacks. 
We now show that the same category $\Hdx S\Y$ can be expressed as a colimit.

\begin{lem} \label{lem:Hdx SY as a colimit along star-zero push}
Let $S$, $\Y$, $U_\bullet$ be as above. Then the natural functor
$$
\uscolim{[n] \in \bDelta^\op} 
\ICoh_0( (S \times U_n)^\wedge_{S \times_\Y U_n})
\longto 
\ICoh_0((S \times \Y)^\wedge_S)
$$
given by the $(*,0)$-pushforward functors is an equivalence. 
\end{lem}

\begin{proof}
Under the equivalence (\ref{eqn:tensor up ICOH0 with QCoh for atlas}), the LHS becomes
$$
\uscolim{[n] \in \bDelta^\op} 
\Bigt{ 
\ICoh_0((S \times \Y)^\wedge_S)
\usotimes{\QCoh(\Y)}
\QCoh(U_n)
}
\simeq
\ICoh_0((S \times \Y)^\wedge_S)
\usotimes{\QCoh(\Y)}
\Bigt{
\uscolim{[n] \in \bDelta^\op} 
 \QCoh(U_n)}
,
$$
where the colimit on the RHS is taken with respect to the $*$-pushforward functors.
It suffices to recall again that the obvious functor
$$
\uscolim{[n] \in \bDelta^\op} \QCoh(U_n)
\longto
\QCoh(\Y)
$$
is a $\QCoh(\Y)$-linear equivalence, see \cite[Proposition 6.2.7]{shvcat}.
\end{proof}

\begin{lem} \label{lem:definition of H over Y}
The collection $\{\Hdx S\Y \}_{S \in (\Affevcoclfp)_{/\Y}}$ assembles to an object of $\ShvCatH(\Y)$ that we shall denote by $\H_{/\Y}$.
\end{lem}

\begin{proof}
We need to prove that, for $S' \to S$ a map in $\Affevcoclfp$, the canonical arrow
$$
\Hdx {S'}{S}
\usotimes{\H(S)}
\Hdx S\Y
\longto
\Hdx{S'}\Y
$$
is an equivalence. We use the canonical left $\H(S)$-linear equivalence
$$
\Hdx S\Y
:=
\ICoh_0( (S \times \Y)^\wedge_\Y)
\simeq
\lim_{U \in (\Affevcoclfp)_{/\Y, \smooth}}
\Hcorr S {S \times_\Y U}U,
$$
discussed above. Since the left leg of each correspondence above is smooth, base-change for $\H$ can be applied to yield
$$
\Hdx {S'}{S}
\usotimes{\H(S)}
\Hdx S\Y
\simeq
\Hdx {S'}{S}
\usotimes{\H(S)}
\lim_{U \in (\Affevcoclfp)_{/\Y, \smooth}}
\Hcorr S {S \times_\Y U}U
\simeq
\lim_{U \in (\Affevcoclfp)_{/\Y, \smooth}}
\Hcorr {S'} {S' \times_\Y U}U.
$$
The latter is $\Hdx {S'}\Y$, as desired.
\end{proof}

\sssec{}

Set $\H(\Y) := \Hgeom(\Y)$. Recall then that the left $\H(S)$-module category $\Hdx S \Y := \Hgeomdx S\Y$ is actually an $(\H(S), \H(\Y))$-bimodule, where both actions are given by convolution. Since $\Hdx S \Y$ is dualizable as a DG category and the monoidal DG categories $\H(S)$ and $\H(\Y)$ are both \emph{very rigid}, Corollary \ref{cor:very rigid has ambidextrous modules} implies that $\Hdx S \Y$ is ambidextrous.

\medskip

By Lemma \ref{lem:Hdx SY as a colimit along star-zero push} and the ambidexterity of $\H$, its (right, as well as left) dual is easily seen to be the obvious $(\H(\Y), \H(S))$-bimodule
$$
\Hsx \Y S 
:=
\Hgeomsx \Y S
:=
\ICoh_0((\Y \times S)^\wedge_S).
$$

\sssec{}

We can now introduce the fundamental adjunction
\begin{equation} \label{eqn:adj-Loc-bGamma-H}
\begin{tikzpicture}[scale=1.5]
\node (a) at (0,1) {$\bLoc^\H_\Y: \H(\Y) \mmod $};
\node (b) at (3,1) {$\ShvCatH(\Y): \bGamma^\H_\Y$.};
\path[->,font=\scriptsize,>=angle 90]
([yshift= 1.5pt]a.east) edge node[above] { } ([yshift= 1.5pt]b.west);
\path[->,font=\scriptsize,>=angle 90]
([yshift= -1.5pt]b.west) edge node[below] {} ([yshift= -1.5pt]a.east);
\end{tikzpicture}
\end{equation}
The left adjoint sends $\C \in \H(\Y) \mmod$ to the $\H$-sheaf of categories represented by 
$$
\{ \Hdx S\Y \usotimes{\H(\Y)} \C  \}_{{S \in  (\Affevcoclfp){/\Y} } }:
$$ 
this makes sense in view of Lemma \ref{lem:definition of H over Y}. The right adjoint sends $\C  = \{\C_S\}_S \in \ShvCatH(\Y)$ to the $\H(\Y)$-module
%%%
\begin{equation} \label{eqn:formula for bGamma}
\bGamma^\H_\Y (\C)
=
\lim_{S \in  ((\Affevcoclfp)_{/\Y, \smooth})^\op } 
\Hsx \Y S
\usotimes{\H(S)}
\C_S,
\end{equation}
where we have used Theorem \ref{thm:smooth-descent-for-ShvCatH}.

\medskip

We say that $\Y$ is \emph{$\H$-affine} if the adjoint functors (\ref{eqn:adj-Loc-bGamma-H}) are mutually inverse equivalences.

\begin{rem} \label{rem:bGamma via homs out of basic object H over Y}
Note that $\bGamma^\H_\Y (\C)$ can be computed as 
$$
\HHom_{\ShvCatH(\Y)}(\H_{/\Y}, \C),
$$
where $\ShvCatH(\Y)$ is regarded as an $\2$-category and $\HHom$ denotes the $\1$-category of $1$-arrows in an $\2$-category.
\end{rem}

\ssec{Pushforwards and the Beck-Chevalley conditions}

For any arrow $f: \Y \to \Z$ in $\Stkevcoclfp$, the functor $f^{*,\H}$ commutes with colimits, whence it admits a right adjoint, denoted by $f_{*,\H}$. Moreover, since $\H$ satisfies the left Beck-Chevalley condition, $f^{*,\H}$ commutes with limits as well, whence it also admits a left adjoint, denoted by $f_{!,\H}$.

In this section we give formulas for these pushforward functors and discuss base-change for $\ShvCatH$.

\sssec{} \label{sssec: star pushforward shvcatH}

Let $f: \Y \to \Z$ be an arrow in $\Stkevcoclfp$. For $\C \in \ShvCatH(\Y)$, we will compute the $\H$-sheaf of categories $f_{*,\H}(\C)$. 
By Theorem \ref{thm:smooth-descent-for-ShvCatH}, it suffices to specify the value of $f_{*,\H}(\C)$ on affine schemes $U \in \Affevcoclfp$ mapping \emph{smoothly} to $\Z$.
For each such $\phi_{U \to \Y}: U \to \Y$, consider the $\H(U)$-module
$$
\E_U := \lim_{V \in ((\Affevcoclfp)_{/U \times_\Z \Y, \smooth})^\op}
\Hsx UV \usotimes{\H(V)} \C_V.
$$
The limit is well-defined thanks to the left Beck-Chevalley condition, that is, exploiting the $\2$-functor $\H^{\LBeck}$ of Section \ref{sssec:summary of B-Chev for H}. Next, using the right Beck-Chevalley condition, one readily checks that the natural functor
$$
\Hdx {U'}{U} \usotimes{\H(U)} \E_U 
\longto
\E_{U'}
$$
is an equivalence for any smooth map $U' \to U$ in $\Aff$.  
This guarantees that $\{\E_U\}_{U \in (\Affevcoclfp)_{/\Z, \smooth}}$ is a well-defined object of $\ShvCatH(\Z)$.
We leave it to the reader to verify that such object in the sought-for pushforward $f_{*,\H}(\C)$. 

\sssec{} \label{sssec: !-push for H}

Similarly, the $!$-pushforward of $\C$ is written as 
$$
f_{!,\H}(\C) \simeq \{\D_U\}_{U \in (\Affevcoclfp)_{/\Z, \smooth}},
$$
where $\D_U$ is defined, using the $\2$-functor $\H^{\RBeck}$, as 
$$
\D_U 
:= 
\uscolim{V \in (\Affevcoclfp)_{/U \times_\Z \Y, \smooth}}
\Hsx UV \usotimes{\H(V)} \C_V.
$$

\sssec{}

It is then \emph{tautological} to verify that the $\ShvCatH$ has the right Beck-Chevalley condition with respect to bounded arrows, that is, the assignment 
$$
[\X \xleftarrow{h} \W \xto{v} \Y]
\squigto 
v_{*,\H} \circ h^{*,\H}
$$
upgrades to an $\2$-functor
\begin{equation} \label{eqn:shvcatA-corr-**}
\ShvCatH_{*,*}:
\Corr(\Stk_\type)_{\evcoc;\all}^{\evcoc, 2-\op}
\longto 
\inftyCat
\end{equation}
where $\inftyCat$ is regarded here as an $\2$-category.
Symmetrically, the assignment
$$
[\X \xleftarrow{h} \W \xto{v} \Y]
\squigto 
v_{!,\H} \circ h^{*,\H}
$$
upgrades to an $\2$-functor
\begin{equation} \label{eqn:shvcatA-corr-!*}
\ShvCatH_{!,*}:
\Corr(\Stk_\type)_{\all;\evcoc}^{\evcoc}
\longto 
\inftyCat.
\end{equation}
\begin{rem}
Combining the two functors together, we deduce that we have base-change isomorphisms
$$
g^{*,\H} \circ f_{*,\H}
\simeq
F_{*,\H} \circ G^{*,\H},
\hspace{.4cm}
g^{*,\H} \circ f_{!,\H}
\simeq
F_{!,\H} \circ G^{*,\H},
$$
as soon as at least one between $f$ and $g$ is bounded. 
\end{rem}

\begin{rem}
We will show later that $!$- and $*$-pushforwards of $\H$-sheaves of categories are naturally identified, see Corollary \ref{cor:ambidexterity}.
\end{rem}

\ssec{Extension/restriction of coefficients}

In this section, we relate $\H$-sheaves of categories with the more familiar quasi-coherent sheaves of categories developed in \cite{shvcat}. The latter are the ones obtained from the coefficient system $\Q$.

\sssec{}

The relation between $\ShvCatH$ and $\ShvCatQ$ is induced by the map $\Q \to \H$ of coefficient systems on $\Affevcoclfp$.
Precisely, $\Q \to \H$ induces a natural transformation
$$
{\oblv}^{\Q \to \H} : \ShvCatH \implies \ShvCatQ
$$
between functors out of $(\Stkevcoclfp)^\op$. In other words, this means that ${\oblv}^{\Q \to \H}$ is compatible with the pullback functors.

\begin{lem}
For $\Y \in \Stkevcoclfp$, the functor $\oblv_\Y^{\Q \to \H}: \ShvCatH(\Y) \to \ShvCatQ(\Y)$ is conservative and admits a left adjoint, which we will call $\ind_\Y^{\Q \to \H}$.
\end{lem}

\begin{proof}
Conservativity is obvious.
The existence of the left adjoint is clear thanks to the fact that $\oblv_\Y^{\Q \to \H}$ commutes with limits.
\end{proof}

\sssec{}

The functor
$$
\ind_\Y^{\Q \to \H}: 
\ShvCatQ(\Y) 
\longto 
\ShvCatH(\Y)
$$ 
is really easy to describe explicitly. Namely,
$$
\ind_\Y^{\Q \to \H} (\C)
\simeq
\uscolim{S \in (\Affevcoclfp)_{/\Y, \smooth} }
(\phi_{S \to \Y})_{!,\H} \bigt{ \H(S) \otimes_{\QCoh(S)} \C_S }.
$$

\begin{lem} \label{lem:computation of induction of Q over Y}
The induction functor $\ind_{\Y}^{\Q \to \H}:
\ShvCat^\Q(\Y) \to \ShvCatH(\Y)$ sends $\Q_{/\Y}$ to $\H_{/\Y}$.
\end{lem}

\begin{proof}
The above formula and Section \ref{sssec: !-push for H} yield
$$
\ind_{\Y}^{\Q \to \H} (\Q_{/\Y})
\simeq
\uscolim{S \in (\Affevcoclfp)_{/\Y} }
(\phi_{S \to \Y})_{!,\H} (\H(S))
\simeq
\uscolim{S \in (\Affevcoclfp)_{/\Y, \smooth} }
\left\{
\uscolim{V \in (\Affevcoclfp)_{/U \times_\Y S, \smooth}}
\Hsx UV \usotimes{\H(V)} \Hdx VS
\right\}_U.
$$
We now apply Lemma \ref{lem:help for descent} twice. First,
$$
\uscolim{V \in (\Affevcoclfp)_{/U \times_\Y S, \smooth}}
\Hsx UV \usotimes{\H(V)} \Hdx VS
=
\uscolim{V \in (\Affevcoclfp)_{/U \times_\Y S, \smooth}}
\ICoh_0((U \times S)^\wedge_V)
$$
is equivalent to $\ICoh_0((U \times S)^\wedge_{U \times_\Y S})$. Secondly,
$$
\uscolim{S \in (\Affevcoclfp)_{/\Y, \smooth} }
\ICoh_0((U \times S)^\wedge_{U \times_\Y S})
\simeq
\ICoh_0((U \times \Y)^\wedge_{U}) =: \Hdx U\Y.
$$
This concludes the computation.
\end{proof}

\ssec{$\H$-affineness}

In this section, we prove our  main theorem, the $\H$-affineness of algebraic stacks, and deduce that $\H^\geom$ is a strict $\2$-functor.

\begin{thm} \label{thm:H-affineness-stacks}
Any $\Y \in \Stkevcoclfp$ is $\H$-affine, that is, the adjunction
\begin{equation}
\begin{tikzpicture}[scale=1.5]
\node (a) at (0,1) {$\bLoc^\H_\Y: \H(\Y) \mmod $};
\node (b) at (3,1) {$\ShvCatH(\Y): \bGamma^\H_\Y$.};
\path[->,font=\scriptsize,>=angle 90]
([yshift= 1.5pt]a.east) edge node[above] { } ([yshift= 1.5pt]b.west);
\path[->,font=\scriptsize,>=angle 90]
([yshift= -1.5pt]b.west) edge node[below] {} ([yshift= -1.5pt]a.east);
\end{tikzpicture}
\end{equation}
is an equivalence of $\infty$-categories.
\end{thm}

\begin{proof}
Our strategy is to reduce to the known $\Q$-affineness of such stacks, see \cite[Theorem 2.2.6]{shvcat}, using the adjunction
\begin{equation}
\nonumber
\begin{tikzpicture}[scale=1.5]
\node (a) at (0,1) {$\ind_\Y^{\Q \to \H}: 
\ShvCatQ(\Y) $};
\node (b) at (3,1) {$\ShvCatH(\Y): \oblv_\Y^{\Q \to \H}$.};
\path[->,font=\scriptsize,>=angle 90]
([yshift= 1.5pt]a.east) edge node[above] { } ([yshift= 1.5pt]b.west);
\path[->,font=\scriptsize,>=angle 90]
([yshift= -1.5pt]b.west) edge node[below] {} ([yshift= -1.5pt]a.east);
\end{tikzpicture}
\end{equation}

\sssec*{Step 1}

For a monoidal functor $f: \CA \to \CB$, we denote by $\ind[f]: \CA \mmod \rightleftarrows \CB \mmod:\oblv[f]$ the standard adjunction.
Let $\delta_\Y: \QCoh(\Y) \to \H(\Y)$ be the usual monoidal functor. 

\medskip

By Lemma \ref{lem:computation of induction of Q over Y}, the diagram
\begin{gather} \label{diag:LOC-HC-ind-commute}
\xy
(0,20)*+{\QCoh(\Y) \mmod }="01";
(0,0)*+{ \ShvCat(\Y)}="00";
(42,20)*+{ \H(\Y) \mmod }="11";
(42,0)*+{ \ShvCatH(\Y).}="10";
 %horiz
{\ar@{<-}_{ \ind_\Y^{\Q \to \H} } "10";"00"};
{\ar@{<-}_{ \ind[\delta_\Y]  } "11";"01"};
%
%vert
{\ar@{->}_{ \bLoc_\Y } "01";"00"};
{\ar@{->}_{ \bLoc_\Y^\H } "11";"10"};
\endxy
\end{gather}	 
is commutative. It follows that the square
\begin{gather} \label{diag:bGamma-H-oblv-commute}
\xy
(0,20)*+{\QCoh(\Y) \mmod }="01";
(0,0)*+{ \ShvCat(\Y)}="00";
(42,20)*+{ \H(\Y) \mmod }="11";
(42,0)*+{ \ShvCatH(\Y)}="10";
% horiz
{\ar@{->}_{ \oblv_\Y^{\Q \to \H} } "10";"00"};
{\ar@{->}_{ \oblv[\delta_\Y]  } "11";"01"};
%
%vert
{\ar@{<-}_{ \bGamma^{\Q}_\Y } "01";"00"};
{\ar@{<-}_{ \bGamma_\Y^\H } "11";"10"};
\endxy
\end{gather}	 
is commutative too.

\sssec*{Step 2}

By changing the vertical arrows with their left adjoints, we obtain a lax commutative diagram
\begin{gather} \label{diag:LOC-HC-oblv-commute}
\xy
(0,20)*+{\QCoh(\Y) \mmod }="01";
(0,0)*+{ \ShvCat(\Y)}="00";
(40,20)*+{ \H(\Y) \mmod }="11";
(40,0)*+{ \ShvCatH(\Y).}="10";
(20,10)*+{ \searrow}="centro";
%
%horiz
{\ar@{->}_{ \oblv^{\Q \to \H}_\Y} "10";"00"};
{\ar@{->}_{ \oblv[\delta_\Y] } "11";"01"};
%
%vert
{\ar@{->}_{ \bLoc_\Y } "01";"00"};
{\ar@{->}_{ \bLoc^\H_\Y } "11";"10"};
\endxy
\end{gather}	 
However, this diagram is genuinely commutative thanks to the  canonical $(\QCoh(S), \H(\Y))$-linear equivalence
$$
\QCoh(S) \usotimes{\QCoh(\Y)} \H(\Y)
\xto{\; \; \simeq \; \; }
\Hdx S \Y.
$$

\sssec*{Step 3}

We are now ready to prove the theorem by checking that the two compositions $\bLoc^\H_\Y \circ \bGamma^\H_\Y$ and $\bGamma^\H_\Y \circ \bLoc^\H_\Y$ are isomorphic to the corresponding identity functors.
This is easily done by using the commutative diagrams (\ref{diag:bGamma-H-oblv-commute}) and (\ref{diag:LOC-HC-oblv-commute}), the conservativity of the functors
$$
\oblv_\Y^{\Q \to \H}: \ShvCatH(\Y) \longto \ShvCatQ(\Y),
\hspace{.5cm}
\oblv[\delta_\Y]: \H(\Y) \mmod \longto \QCoh(\Y) \mmod,
$$
and the $\Q$-affineness of $\Y$.
\end{proof}

\sssec{}

Combining the $\2$-functor 
$$
\ShvCatH_{*,*}:
\Corr(\Stkevcoclfp)_{\evcoc;\all}
^{\evcoc, 2 -\op} 
\longto
\inftyCat
$$ 
of (\ref{eqn:shvcatA-corr-**}) with Theorem \ref{thm:H-affineness-stacks}, we obtain  
another \emph{strict} $\2$-functor 
\begin{equation} \label{eqn:ACorr-Stacks}
\Hcat: 
\Corr(\Stkevcoclfp)_{\evcoc;\all}^{\evcoc, 2-\op} 
\longto
\ALGBimod(\DGCat),
\end{equation}
defined by 
$$
\X 
\squigto
\Hcat(\X) := \H(\X)
$$
$$
[\X \xleftarrow{v} \W \xto{h} \Y]
\squigto 
(\Hcat)_{\X \leftto \W \to \Y}
:=
 \bGamma^\H_\Y \circ (h_{*,\H} \circ v^{*,\H}) \circ \bLoc^\H_\Y(\H(\Y)) .
$$

\begin{thm} \label{main-thm-Hgoem is strict}
The lax $\2$-functor 
$$
\H^\geom: 
\Corr \bigt{\Stkevcoclfp }_{\evcoc;\all}
^{\schem \& \evcoc \& \proper}
\longto
\ALGBimod(\DGCat)
$$
of Section \ref{ssec:2-cat functoriality Hgoem} is naturally equivalent to the restriction of $\Hcat$ to $\Corr \bigt{\Stkevcoclfp }_{\evcoc;\all}
^{\schem \& \evcoc \& \proper}$.
Hence, $\Hgeom$ is strict.
\end{thm}

Henceforth, we will denote both $\2$-functors simply by $\H$.

\begin{proof}
By Remark \ref{rem:bGamma via homs out of basic object H over Y}, the DG category underlying $\Hcatcorr{\X}{\W}{\Y}$ is computed as follows:
\begin{eqnarray}
\nonumber
\Hcatcorr{\X}{\W}{\Y}
& \simeq & 
\HHom_{\ShvCatH(\Y)}
\bigt{
\H_{/\Y},
h_{*,\H} \circ  v^{*,\H} (\H_{/\X})
}
\\
\nonumber
& \simeq & 
\HHom_{\ShvCatH(\W)}  
\bigt{
 h^{*,\H} (\H_{/\X}) ,
v^{*,\H} (\H_{/\Y})
}
\\
\nonumber
& \simeq &
\lim_{U \in ((\Affevcoclfp)_{/\W, \smooth})^\op} 
\HHom_{\H(U)} \bigt{ \Hdx U \X, \Hdx U \Y }
\\
\nonumber
& \simeq &
\lim_{U \in ((\Affevcoclfp)_{/\W, \smooth})^\op} 
\Hsx \X U 
\usotimes{\H(U)} \Hdx U \Y 
\\
\nonumber
& \simeq &
\lim_{U \in ((\Affevcoclfp)_{/\W, \smooth})^\op} 
\lim_{S \in ((\Affevcoclfp)_{/\X, \smooth})^\op} 
\lim_{T \in ((\Affevcoclfp)_{/\Y, \smooth})^\op} 
\Hcorr S {S \times_\X U} U \usotimes{\H(U)}  \Hcorr U {U \times_\Y T}T.
\end{eqnarray} 
By base-change for $\H$, we have
$$
\Hcorr S {S \times_\X U} U \usotimes{\H(U)}  \Hcorr U {U \times_\Y T}T
\simeq
\Hcorr S {S \times_\X U \times_\Y T} T
\simeq
\ICoh_0
\bigt{
(S \times T)^\wedge_{S \times_\X U \times_\Y T}
}.
$$
By taking the limit, we obtain
$$
\Hcatcorr{\X}{\W}{\Y}
\simeq
\ICoh_0((\X \times \Y)^\wedge_\W)
=:
\Hgeomcorr \X \W \Y,
$$
as desired.
\end{proof}

\begin{cor} \label{cor:push-pull-under-H-affineness}
For $f: \Y \to \Z$ in $\Stkevcoclfp$. Then the functors $f_{*,\H}$ and $f^{*,\H}$ correspond under $\H$-affineness to the functors of $\Hsx \Z \Y \otimes_{\H(\Y)} -$ and $\Hdx \Y \Z \otimes_{\H(\Z)} -$, respectively. 
\end{cor}

\begin{proof}
Let $\C \in \H(\Y) \mmod$. We need to exhibit a natural equivalence
$$
\bGamma_\Z^\H
\circ
f_{*,\H}
\circ 
\bLoc^\H_\Y(\C)
\simeq
\Hsx \Z \Y 
\usotimes{\H(\Y)}
\C.
$$
This easily reduces to the case $\C= \H(\Y)$, where it holds true by construction.
The assertion for $f^{*,\H}$ is proven similarly.
\end{proof}

\begin{cor} \label{cor:ambidexterity}
Pullbacks of $\H$-sheaves of categories are \emph{ambidextrous}: for any $f: \Y \to \Z$ in $\Stkevcoclfp$, there is a canonical equivalence $f_{!,\H} \simeq f_{*,\H}$.
\end{cor}

\begin{proof}
Recall the formulas for $f_{!,\H}$ and $f_{*,\H}$ from Sections \ref{sssec: star pushforward shvcatH} and \ref{sssec: !-push for H}.
By $\H$-affineness, it suffices to exhibit a natural equivalence  $f_{!,\H}(\H_{/\Y}) \simeq f_{*,\H}(\H_{/\Y})$. The latter is constructed as in Lemma \ref{lem:help for descent}.
\end{proof}

\ssec{The $\H$-action on $\ICoh$} \label{ssec:ICOH}

This final section contains an example of our techniques. We view $\ICoh(\Y)$ as a left module for $\H(\Y)$ and compute $\H$-pullbacks along smooth maps, as well as $\H$-pushforwards along arbitrary maps.

\begin{lem} \label{lem:H-pullback-ICoh-stacks}
For a smooth map $\X \to \Y$ in $\Stkevcoclfp$, the natural $\H(\X)$-linear functor
$$
\Hdx \X \Y \usotimes{\H(\Y)} \ICoh(\Y)
\longto
\ICoh(\X)
$$
is an equivalence.  
\end{lem}

\begin{proof}
This is just a consequence of the $(\QCoh(\X), \H(\Y))$-bilinear equivalence 
$$
\Hdx \X \Y 
\simeq 
\QCoh(\X) 
\usotimes{\QCoh(\Y)} 
\H(\Y),
$$
together with \cite[Proposition 4.5.3]{ICoh}.
\end{proof}

\begin{rem}
The example of $\Y=\pt$ shows that we should not expect this result to be true for non-smooth maps. 
\end{rem}

\begin{prop} \label{prop:H-push-ICoh-stacks}
For a map $f:\Y \to \Z$ in $\Stkevcoclfp$, the natural $\H(\Z)$-linear functor
$$
\Hsx \Z \Y \usotimes{\H(\Y)} \ICoh(\Y)
\longto
\ICoh(\Z^\wedge_\Y)
$$
is an equivalence.
\end{prop}

\begin{proof}
Let 
$$
\ICoh_{/\Y} := \bLoc_\Y^\H(\ICoh(\Y)) \in \ShvCatH(\Y).
$$
Lemma \ref{lem:H-pullback-ICoh-stacks} gives the equivalence $(\phi_{V \to \Y})^{*,\H}(\ICoh_{/\Y}) \simeq \ICoh(V)$ for any affine scheme $V$ mapping smoothly to $\Y$. 
We then have: 
\begin{eqnarray}
\nonumber
\bGamma_\Z^\H f_{*,\H} (\ICoh_{/\Y}) 
& \simeq &
\lim_{V \in ((\Affevcoclfp)_{/\Y, \smooth})^\op}
\Hsx \Z V \usotimes{\H(V)} \ICoh(V)
\\
\nonumber
& \simeq &
\lim_{V \in ((\Affevcoclfp)_{/\Y, \smooth})^\op}
\lim_{U \in ((\Affevcoclfp)_{/\Z, \smooth})^\op}
\Hcorr U {U \times_\Z V} V \usotimes{\H(V)} \ICoh(V)
\\
\nonumber
& \simeq &
\lim_{V \in ((\Affevcoclfp)_{/\Y, \smooth})^\op}
\lim_{U \in ((\Affevcoclfp)_{/\Z, \smooth})^\op}
\ICoh(U^\wedge_{U \times_\Z V})
\\
\nonumber
& \simeq &
\lim_{V \in ((\Affevcoclfp)_{/\Y, \smooth})^\op}
\ICoh(\Z^\wedge_V)
\\
\nonumber
& \simeq &
\ICoh(\Z^\wedge_\Y).
\end{eqnarray}
Here we have used the self-duality of $\ICoh(S)$, the rigidity of $\H(S)$, Proposition \ref{prop:IndCoh-push-forward} (i.e., the special case of the assertion for affine schemes), Lemma \ref{lem:H-pullback-ICoh-stacks} and smooth descent for $\ICoh$.
The conclusion now follows from Corollary \ref{cor:push-pull-under-H-affineness}.
\end{proof}

%%%%%%
%%%%%%
%%%%%%%%%%%%%%%%%%%%%%%%%%%%%%
%%%%%%%%%%%%%%%%%%%%%%%%
%%%%%%%%%%%%%%%%%%%%%%%%%%%%%%%%%%%%%%%%%%%%%%%%
%%%%%%%%%%%%%%%%%%%%%%%%
%%%%%%%%%%%%%%%%%%%%%%%%%%%%%%%%%%%%%%%%%%%%%%%%
%%%%%%%%%%%%%%%%%%%%%%%%
%%%%%%%%%%%%%%%%%%%%%%%%%%%%%%%%%%%%%%%%%%%%%%%%
%%%%%%%%%%%%%%%%%%%%%%%%
%%%%%%%%%%%%%%%%%%%%%%%%%%%%%%%%%%%%%%%%%%%%%%%%
%%%%%%%%%%%%%%%%%%%%%%%%
%%%%%%%%%%%%%%%%%%

%%%%%
%%%%%%%%%%
%%%%%%%%%%
%%%%%%%%%%
%%%%%%%%%%
%%%%%%%%%%
%%%%%%%%%%

%%%\noindent
%%%Dario Beraldo beraldo@maths.ox.ac.uk \\
%%%Mathematical Institute, University of Oxford, Woodstock Rd, Oxford, OX2 6GG, UK

\end{document}